\documentclass[letterpaper]{article}
\usepackage[margin=2.5cm]{geometry}
\usepackage{float}
\usepackage{xcolor}
\usepackage{graphicx}
\usepackage{epstopdf}
\usepackage{subfigure}
\usepackage{amssymb,amsmath,amsthm}
\usepackage{authblk}

\def\d{\mathrm{d}}

\def\eps{{\epsilon}} 
\def\R{\mathbb{R}}

\renewcommand{\vec}[1]{\mbox{\boldmath$#1$}}
\def\sech{{\, \mathrm{sech} \, }}

\newtheorem{theorem}{Theorem}
\newtheorem{lemma}{Lemma}
\newtheorem{corollary}{Corollary}
\newtheorem{remark}{Remark}

\begin{document}
		
	\title{Melnikov theory for two-dimensional manifolds in three-dimensional flows}
	
	\author[1,4]{K.G.D. Sulalitha Priyankara}
	\affil[]{Department of Mathematics, Clarkson University, Potsdam, NY 13699, USA.}
	
	\author{Sanjeeva Balasuriya}
	\affil[]{School of Mathematical Sciences, University of Adelaide\\ Adelaide, SA 5005, Australia.}
	
	\author[3,4]{Erik Bollt}
	\affil[]{Department of Electrical and Computer Engineering\\ Clarkson University, Potsdam, NY 13699, USA.}
	\affil[4]{Clarkson Center for Complex Systems Science ($\mathrm{C}^3 \mathrm{S}^2)$, Potsdam, NY 13699, USA
}

\maketitle

\begin{abstract}
We present a Melnikov method to analyze two-dimensional  stable or unstable manifolds associated with a saddle point in three-dimensional  non-volume preserving autonomous systems. The time-varying perturbed locations of such manifolds is obtained under very general, non-volume preserving and with arbitrary time-dependence, perturbations.  In unperturbed situations with a two-dimensional heteroclinic manifold, we adapt our theory to
quantify the splitting into a stable and unstable manifold, and thereby obtain a
Melnikov function characterizing the time-varying locations of transverse intersections of these manifolds.  Formulas for lobe volumes arising from such intersections, as well as the instantaneous flux across the broken heteroclinic manifold, are obtained in terms of the Melnikov function.  Our theory has specific application to transport in fluid mechanics, where the flow is in three dimensions and flow separators are two-dimensional stable/unstable manifolds.   We demonstrate our theory using both the classical and the swirling versions of Hill's spherical vortex.
\end{abstract}


\section{Introduction}
\label{sec:introduction}

Melnikov methods were originally introduced \cite{melnikov} to analyze how a homoclinic connection (a coincident stable and unstable manifold of a saddle fixed point) splits when an autonomous system is perturbed. The original theory was confined to area-preserving flows under periodic perturbations \cite{melnikov}, thereby
allowing for the perturbed flow to be considered via a Poincar\'e map \cite{nonosci}. 
The Melnikov function is a scaled distance between stable and unstable manifolds into which the original homoclinic splits, and simple zeros of this function identify the places where stable and unstable manifolds intersect transversely \cite{melnikov,nonosci,wigi_book}. 
Establishing the existence of simple zeros of a Melnikov function serve as one of the
few methods for `proving chaos'; the proof works classically via the Smale-Birkhoff theorem \cite{arrowsmith_place,nonosci,wigi_book}
in two-dimensional maps, or time-periodic flows, with a one-dimensional {\em homoclinic} connection. 
Extensions of the `proof of chaos' idea also exist to some heteroclinic (when the manifolds are associated with different fixed points) situations \cite{bertozzi,article_chaosdg}.

Many reinterpretations and extensions of Melnikov theory exist (see the review chapters in \cite{siam_book}). The Melnikov function has been extended for implicitly
defined differential equations \cite{battellifeckanimplicit}, heteroclinic situations  \cite{tangential,article_chaosdg,3d}, for stochastic perturbations \cite{shen}, singular perturbations \cite{gruendlersingular,oka}, nonhyperbolic situations \cite{wechselberger,zang}, fixed points at infinity \cite{wechselberger}, degenerate homoclinics \cite{vanderbauwhede} and discontinuous 
\cite{duzhang,calamaifranca,battellifeckanhomoclinic,kukucka} and impulsive \cite{impulsive} vector fields.  Higher-dimensional (greater than two) extensions
are also available under some restrictions: usually volume-preserving or Hamiltonian flows  \cite{gruendler,3d,yamashita,lin,gideadelallave}.   In particular we highlight those in three dimensions associated with various
additional conditions \cite{gruendler,holmes1984,3d}, since these have strong applicability in realistic fluid flows in three dimensions.  Higher-order Melnikov methods have also been
developed \cite{duzhang}. 
It is also possible to couch the transverse intersection
problem in a functional analytic, rather than a geometric, form instead \cite{chowhalemalletparet,palmer,battellilazzari,finite,lin}---a method which transforms naturally to higher dimensions.  However, in such cases the kernel of the Melnikov integral
contains an {\em abstract} function which is not in general expressible explicitly for actual computation.

Most applications of Melnikov theory described above relate to defining a Melnikov function whose simple zeros imply the persistence of a homo/hetero-clinic connection \cite[e.g.]{melnikov,chowhalemalletparet,palmer,battellilazzari,finite,gruendler,holmes1984,3d,duzhang,calamaifranca,battellifeckanhomoclinic,kukucka,wechselberger,gruendlersingular,oka,battellifeckanimplicit,gideadelallave}.  In these cases, the Melnikov function does not necessarily express exactly the physical distance between stable and unstable manifolds which separate off the broken homo/hetero-clinic manifold. It generically represents a nonuniformly scaled version of this distance.  In the functional analytical developments in particular, this scaling is hidden; it is only the function's zeros which give the pertinent information on where the stable and unstable manifolds intersect.  It is less well-known that Melnikov developments can be adapted to characterize the {\em locations} of each of the perturbed stable and unstable manifolds---and not just the locations of their intersections in a homo/hetero-clinic situation.  One existing development of this in two dimensions enables locating perturbed stable or unstable manifolds of a saddle fixed point, due to the presence of a time-varying perturbation \cite{tangential}.  In this paper, we extend this theory to locating two-dimensional time-varying stable (or unstable) manifolds of a hyperbolic fixed point in a three-dimensional flow, due to the inclusion of a perturbation whose spatial derivatives are bounded for all time.  Note that we do not require volume preservation in either the unperturbed of perturbed flows, nor time-periodicity in the perturbation.  We emphasize that this development does not require a homo- or hetero-clinic situation, and is the first of the two main results of this paper.  We present this in Theorems~\ref{theorem:melnikov_unstable} and
\ref{theorem:melnikov_stable}, for respectively the unstable and stable manifolds.

An important aspect of most Melnikov developments is in obtaining an integral which is often called a Melnikov function.  Such an integral is well-known to play a role when one wishes to determine intersections in a broken homo/heteroclinic situation.  A similar definite integral, but with non-infinite limits, also appears when we locate perturbed stable/unstable manifolds.  In general, such integrals contain as kernel a particular function the knowledge of which is crucial to represent the Melnikov integral.  In Hamiltonian \cite[e.g.]{gruendler,holmes1984,gideadelallave}, as well as in volume-preserving unperturbed situations with a nondegenerate conserved quantity \cite{3d}, explicit forms for this kernel function can be determined. For more general situations, the kernel function can be expressed in more abstract terms: it is related to the fundamental matrix solution to the adjoint of the variational equation along the relevant homo/hetero-clinic trajectory.  Given that this adjoint equation is {\em nonautonomous}, its solutions {\em cannot} usually be written down explicitly, unless in special situations such as in two dimensions.  Therefore, while a Melnikov function might be expressible for such situations in an abstract sense \cite{chowhalemalletparet,lin,yamashita}, it is usually not computable.  Put another way, most Melnikov developments in dimensions greater than two, or which are not Hamiltonian, provide a theoretical result which is difficult to apply.  Knowing an explicit formula for the kernel function is such situations is therefore valuable. In our development, we are able to provide an explicit expression for it in our three-dimensional setting.  The formula is related to a triple scalar product associated with a parametrization of the two-dimensional manifold.  Thus, the Melnikov function that we develop for locating the perturbed version of such a manifold is {\em computable}, unlike that in many higher-dimensional non-Hamiltonian Melnikov developments. 

We have mentioned that most Melnikov developments work to determine intersections between the stable and unstable manifolds resulting from a broken homo/hetero-clinic manifold.  The new theory that we develop specializes to such a situation as well, and thus we are able to present a computable Meknikov function in a non-Hamiltonian, non-volume-preserving situation, in a dimension greater than two. Moreover, we are able to quantify transport across the broken heteroclinic in terms of this Melnikov function.  This is the second of the main results of this paper, which we present in Theorem~\ref{theorem:flux}.

Quantifying transport when a heteroclinic (a flow-separating curve) in two-dimensions is broken is a well-studied problem.  In two dimensions, the interweaving of the stable and unstable manifolds which split off the heteroclinic generates lobes, and transport can be characterized via 
the beautiful theory of lobe dynamics and turnstiles \cite{romkedar,wigi_book}.   This theory is confined to two-dimensional flows, and for an area of a lobe to be a well-defined characterizer of the transport engendered across the broken heteroclinic, several other features need to be in place: the flow needs to be area-preserving, and the perturbation `harmonic' in that it can be written as a spatially-varying two-dimensional function multiplied by a sinusoid in time.  The area of a lobe then expresses the amount of fluid transported across the broken heteroclinic during the time-periodicity of the perturbation, and can be expressed in terms of a definite integral of an appropriate Melnikov function \cite{romkedar,wigi_book}.  More general time-periodic situations generically do not have well-defined lobe areas because there can be many, differently sized lobes relevant to one iteration of the time-periodic map, or indeed no lobes at all because the perturbed manifolds do not intersect \cite{periodic}.  Obtaining a transport characterization in more general time-{\em aperiodic} situations therefore requires a slightly different approach, and has been provided in two-dimensional flows via 
a time-dependent flux idea \cite{aperiodic,siam_book,rossbyflux}.  As befitting any assessment of transport, this takes into account the {\em Lagrangian} motion of trajectories, rather than an Eulerian flux.  (This terminology stems from fluid mechanics in which `Lagrangian' refers to following the flow, while `Eulerian' in this context would mean measuring transport across fixed surfaces in space, without taking into account that these surfaces are themselves moving due to the flow.)  The instantaneous flux is shown to be characterized in these instances by the relevant Melnikov function, and not its integral.  

We are able to extend these broken heteroclinic results to our current three-dimensional setting. In time-harmonic, volume-preserving situations, a nice analogue of lobe dynamics is seen to occur; in this case, it is lobe {\em volumes} rather than areas that is relevant. We specifically obtain an analytic
formula for leading-order lobe volume in terms of an appropriate integral of the Melnikov function, thereby extending
a well-known two-dimensional result for lobe areas \cite{romkedar,wigi_book}. When volume-preservation and time-harmonicity are relaxed, we are able to define the instantaneous flux (volume per unit time) crossing the broken heteroclinic, extending the two-dimensional ideas in \cite{aperiodic,siam_book,rossbyflux}.  The instantaneous flux function is once again shown to have a direct connection to the Melnikov integral.

We remark that the transport characterization we provide for three-dimensional flows is motivated strongly by fluid mechanics.  Realistic flows in fluids are inherently three-dimensional, and internal flow separators must therefore be two-dimensional entities.  Two-dimensional stable and unstable manifolds are primary candidates for such flow separators.  Locating them and tracking their motion is therefore fundamental in determining boundaries between coherently moving regions of fluids; this is related to the field of `Lagrangian coherent structures' \cite{glcs}.  In particular, characterizing a flow rate (a flux, i.e., a volume of fluid per unit time) across a broken heteroclinic provides a direct assessment of the transport between two previously separated coherent regions.  It is precisely this which we are able to provide with our flux theory.  Similar theory has been used extensively for {\em two}-dimensional flows with one-dimensional flow separators due to the existence of pertinent Melnikov theory \cite[e.g.]{romkedar,romkedarpoje,periodic,aperiodic}, and can even give insight into how to perturb a flow to optimize mixing \cite{l2mixer,optimal}. However, genuine fluid flows are {\em three}-dimensional, and hence our current theory can extend these methods to significantly more realistic flows.

This paper is organized as follows.  In Sec.~\ref{sec:manifolds}, we build the general Melnikov theory for two-dimensional invariant manifolds of a three-dimensional
non-volume preserving flow.   We develop computable spatiotemporal expressions for locating such a manifold under general time-aperiodic perturbation.  This is the first of our main results (Theorems~\ref{theorem:melnikov_unstable} and \ref{theorem:melnikov_stable}).  This theory is adapted in Sec.~\ref{sec:heteroclinic} for the situation when the
unperturbed flow possesses a two-dimensional heteroclinic manifold.  The Melnikov function
we formulate can be used to identify transverse intersections of the perturbed stable and unstable manifolds, as well as to characterize instantaneous flux.  We emphasize that there is no requirement for
either time-periodicity or volume preservation, neither is it necessary for lobes to form. The flux theory still applies if there are no intersections of perturbed stable and unstable manifolds.
This development we use to rationalize the flux, and the accompanying simple formula we obtain 
in terms of the Melnikov function, is the second of our main results (Theorem~\ref{theorem:flux}).  We also establish connections
to more standard situations in two dimensions (sinusoidal perturbations with area-preservation) in which lobe dynamics applied \cite{romkedar,wigi_book}; in this case, transport is measured in terms of lobe {\em volumes}, which we express
in terms of the Melnikov function as well.
In Sec.~\ref{sec:hill}, we apply the theory to
both the classical Hill's spherical vortex  \cite{hill1894}, and a modification incorporating swirl \cite{article_Swhill}, respectively.  We conclude in Sec.~\ref{sec:conclude} with
some remarks on extensions and applications.
    
\section{Melnikov theory for 2-D manifolds}
\label{sec:manifolds}

In this section, we build a Melnikov theory for two-dimensional invariant manifolds that are attached to saddle points in three-dimensional autonomous dynamical systems.  We emphasize that the theory does {\em not} require a homo/hetero-clinic framework, which is the focus of most classical Melnikov approaches.  Rather, our
theory serves to characterize the {\em location}, as it varies with time, of a perturbed two-dimensional invariant manifold when the flow is subject to a very general perturbation.  We consider the system 
\begin{equation} 
	\dot{\vec{x}} = \vec{f}\left(\vec{x}\right) + \epsilon \, \vec{g}\left(\vec{x}, t \right), 
	\label{sys}
\end{equation} 
in which $\vec{x} \in \Omega \subset \mathbb{R}^3,$ $\vec{f} : \Omega \rightarrow \mathbb{R}^3$, $\vec{g}: \Omega \times \mathbb{R} \rightarrow \mathbb{R}^3$ and $0 < \epsilon \ll 1 $.  The $ \epsilon = 0 $ system of (\ref{sys})
is considered the {\em unperturbed} system. During this work we assume the following.
\begin{enumerate}
\item The function $\vec{f} \in \mathbf{C}^2(\Omega)$, and $ \vec{D} \vec{f} $ is bounded in $ \Omega $.
\item The point $\vec{a} \in \mathbb{R}^3$ is a saddle fixed point of the unperturbed system (i.e., (\ref{sys}) when $\epsilon = 0$). Thus, $ \vec{f}(\vec{a}) = \vec{0} $, and the eigenvalues of $ \vec{D} \vec{f} (\vec{a}) $ fall into one
of the following categories:
\begin{itemize}
		\item Case 1: one is negative, and the other two have positive real parts, or
		\item Case 2: one is positive, and the other two have negative real parts.
\end{itemize} 
\item The eigenvectors associated with the two-dimensional (unstable or stable, corresponding to cases~1 or 2 respectively) subspace of $ D \vec{f}(\vec{a}) $ are linearly independent.
\item The two-dimensional stable or unstable manifold identified above is $ \mathrm{C}^2 $-smooth.
\item For any $t \in \mathbb{R}$, the perturbing function $\vec{g}\left(\vec{x}, t \right) \in \mathbf{C}^2(\Omega).$ 
	Additionally,  both $\vec{g}$ and $ \vec{D} \vec{g}$ are bounded in $\Omega \times \mathbb{R}$.     
\end{enumerate}
In seeking expressions for the perturbed two-dimensional invariant manifold, we will focus on the two possibilities for the eigenvalues separately.

\subsection{Displacement of 2D unstable manifold}
\label{sec:unstable}

First, consider case~1, when the system (\ref{sys}) when $\epsilon = 0$ has one negative eigenvalue and two eigenvalues with positive real parts at the point $\vec{a}$. So the unperturbed system possesses a one-dimensional stable manifold and a two-dimensional unstable manifold. We are interested in characterizing the impact of the
perturbation (i.e., $ \epsilon \ne 0 $ in (\ref{sys})) on the two-dimensional manifold, $ \Gamma^u(\vec{a}) $.  We will identify different trajectories on $ \Gamma^u(\vec{a}) $ by the parameter $ \alpha \in {\mathrm{S}}^1$, that is,
$ \alpha \in [0,1) $, periodically extended with interval $ 1 $.  To explain this identification, consider the tangent plane
to $ \Gamma^u(\vec{a}) $ at $ \vec{a} $, and consider a small circle of radius $ \delta $ centered at $ \vec{a} $. 
We can think of $ \alpha $ as the angle going around the circle divided by $ 2 \pi $ (having chosen an $ \alpha = 0 $ location),
and at each $ \alpha $ value, the circle will intersect exactly one trajectory which lies on $ \Gamma^u(\vec{a}) $.
This is so whether $ \vec{D} \vec{f}(\vec{a}) $ has two negative, or two complex with negative real part, eigenvalues;
in the former case there will be non-spiralling trajectories, and in the latter case there will be spiralling trajectories,
going in to $ \vec{a} $.  In either situation $ \alpha $ as explained can be used to parametrize the choice of
trajectory.  Next, the time-variation along each trajectory will be parametrized by $ p $.  
Thus, if  $\bar{\vec{x}}^u(p, \alpha)$ is such a trajectory indexed by $ \alpha $, we have
\[
\frac{\partial \bar{\vec{x}}^u(p,\alpha)}{\partial p} = \vec{f} \left( \bar{\vec{x}}^u(p, \alpha) \right) \, , 
\]
because it is a trajectory of (\ref{sys}) with $ \epsilon = 0 $.  Now, the trajectory can extend outwards in various ways
(e.g., may be a heteroclinic trajectory and thus approach a different critical point, or may go to infinity).  To account for this, we limit
$ p $ to be in the set $ (-\infty,P] $ for any finite $ P $, and choose the $ p $-parametrization of nearby trajectories 
continuously.  Thus, we can use $ (p,\alpha) \in [-\infty,P] \times \mathrm{S}^1 $ to parametrize a restricted version of
$ \Gamma^u(\vec{a}) $ which avoids having to specify the limiting behavior of the unstable manifold trajectories
$  \bar{\vec{x}}^u(p,\alpha) $ in the limit $ p \rightarrow \infty $, while realizing that 
\[
\lim_{p \rightarrow -\infty} \bar{\vec{x}}^u(p, \alpha) = \vec{a} \, . 
\]
 Fig.~\ref{fig:unstable_unp} demonstrates the two-dimensional unstable manifold attached to the fixed point $\vec{a}$, and
 illustrating the roles of $ (p,\alpha) $ in parametrizing the manifold.  We assume that the parametrization by $ (p,\alpha) $ is $ \mathrm{C}^1 $-smooth.  This picture corresponds to $ \vec{D} \vec{f} (\vec{a}) $ having two negative eigenvalues; if they are complex with negative real parts instead, the trajectories will spiral
 (swirl) into $ \vec{a} $ instead.  
 
\begin{figure}[t]
\centering
\includegraphics[scale=0.33]{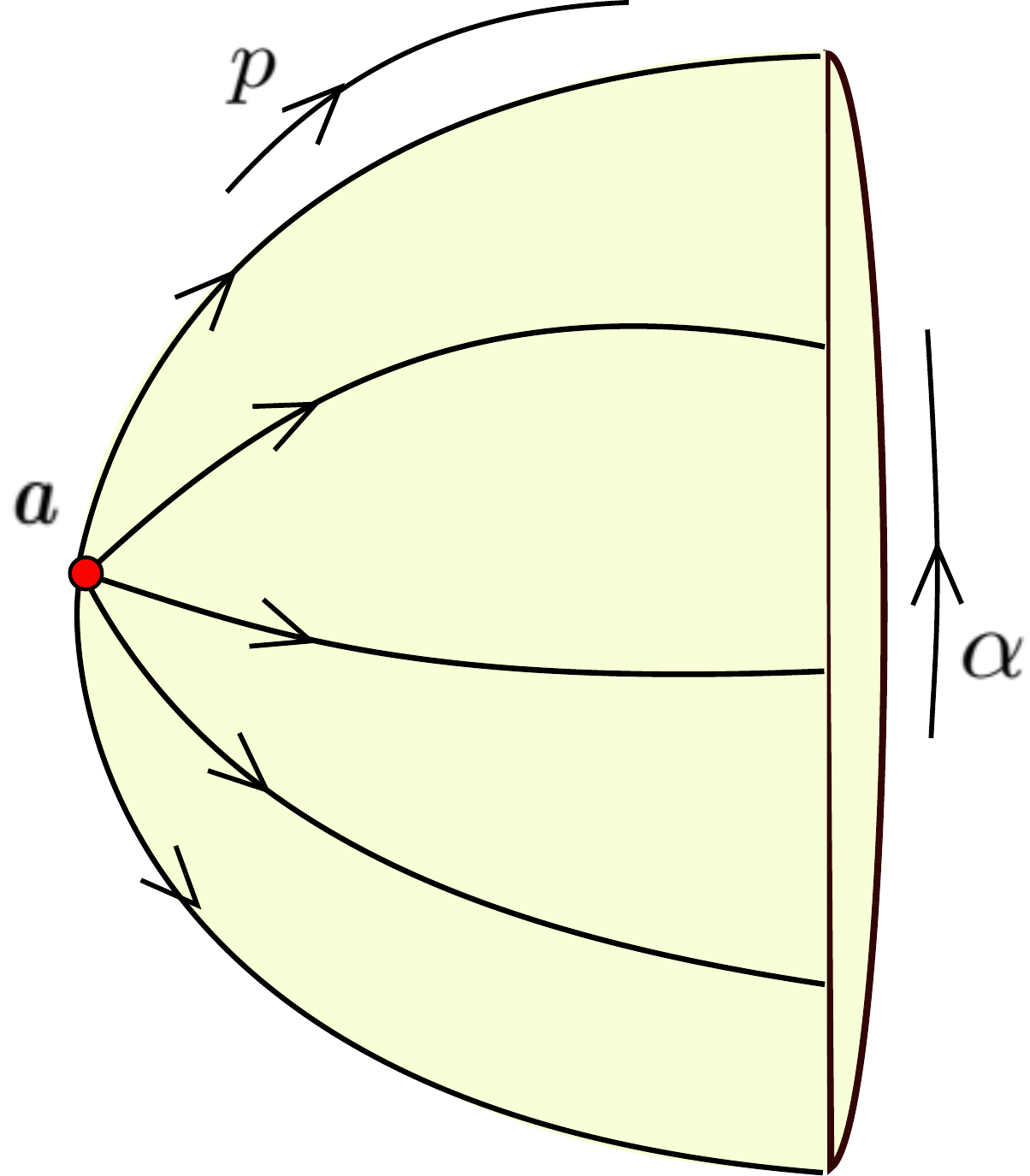}
\caption{The two-dimensional unperturbed unstable manifold $ \Gamma^u(\vec{a}) $ for case~1, with the red dot representing the saddle fixed point. The $ (p,\alpha) $-dependence of trajectories $(\bar{\vec{x}}^u(p, \alpha))$ can be used to parametrize $ \Gamma^u(\vec{a}) $. }
\label{fig:unstable_unp}
\end{figure}

While so far we have described the situation with respect to the system $ \dot{\vec{x}} = \vec{f}(\vec{x}) $, we may
consider instead the behavior within the {\em augmented system} where we append the equation $\dot{t} = 1$. In this situation, the phase space is now $ \Omega \times \mathbb{R} $, and the saddle fixed point $\vec{a}$ becomes a hyperbolic trajectory $(\vec{a}, t)$. Now, the unstable manifold in the four-dimensional augmented phase space will be parameterized as $(p, \alpha, t ) \in (-\infty, P] \times \mathrm{S}^1 \times (-\infty, T]$ for any finite $ T $.   Specifically, the point $\left(  \bar{\vec{x}}^u(p, \alpha), t \right) $, where $\alpha~\text{and}~p$ are spatial parameters and $t$ is time, now parameterizes the augmented unstable manifold. 

Now consider the impact of introducing the perturbation by setting $ \epsilon \ne 0 $ in (\ref{sys}).  Since $ \vec{g} $ 
is sufficiently smooth and bounded, the hyperbolic trajectory $ \left( \vec{a}, t \right) $ perturbs to  $ \left( \vec{a}_{\epsilon}(t), t \right)$ which is $\mathcal{O}(\epsilon) $-close to $\vec{a}$ for $t \in [-\infty,T] $. We caution that $ \vec{a}_\epsilon(t) $ {\em cannot} be obtained by seeking instantaneous fixed points of (\ref{sys}), but instead is defined in terms of
exponential dichotomies \cite{coppel,palmer}, and is in general difficult to compute.  Next, persistence
results of invariant manifolds associated with hyperbolic fixed points \cite{yi} indicate the presence of a perturbed unstable manifold, $\Gamma_{\epsilon}^u(\vec{a}_\epsilon,t) $ that is $\mathcal{O}(\epsilon) $-close to $\Gamma^u(\vec{a})$ at finite times $ t $. 
More specifically, suppose that
$ t \in (-\infty, T] $ is fixed, and we view the projections of the relevant manifolds on this time-slice $ t $.  See Fig.~\ref{fig:unstable_pert}, where the unperturbed manifold is in black, and the perturbed manifold is indicated in red.  
There is a point $ \vec{x}^u \left( p, \alpha, \epsilon, t \right) $ on the perturbed manifold which is $ {\mathcal O}(\epsilon) $-close to $ \bar{\vec{x}}^{u} \left( p, \alpha \right) $.  Our aim is to quantify the distance $ d^u(p,\alpha,\epsilon,t) $
obtained by projecting the vector $ \vec{x}^u \left( p, \alpha, \epsilon, t \right)  -  \bar{\vec{x}}^{u} \left( p, \alpha \right) $
on to the normal vector drawn to $ \Gamma^u(\vec{a}) $ at the point $  \bar{\vec{x}}^{u} \left( p, \alpha \right) $.  By
doing so, we will be able to give the location of the perturbed manifold $\Gamma_{\epsilon}^u(\vec{a}_\epsilon,t) $
parametrized by time $ t $ and the spatial variables $ (p,\alpha) $, to leading-order in $ \epsilon $.

\begin{figure}[t]
	\centering
	\includegraphics[width=0.7 \textwidth]{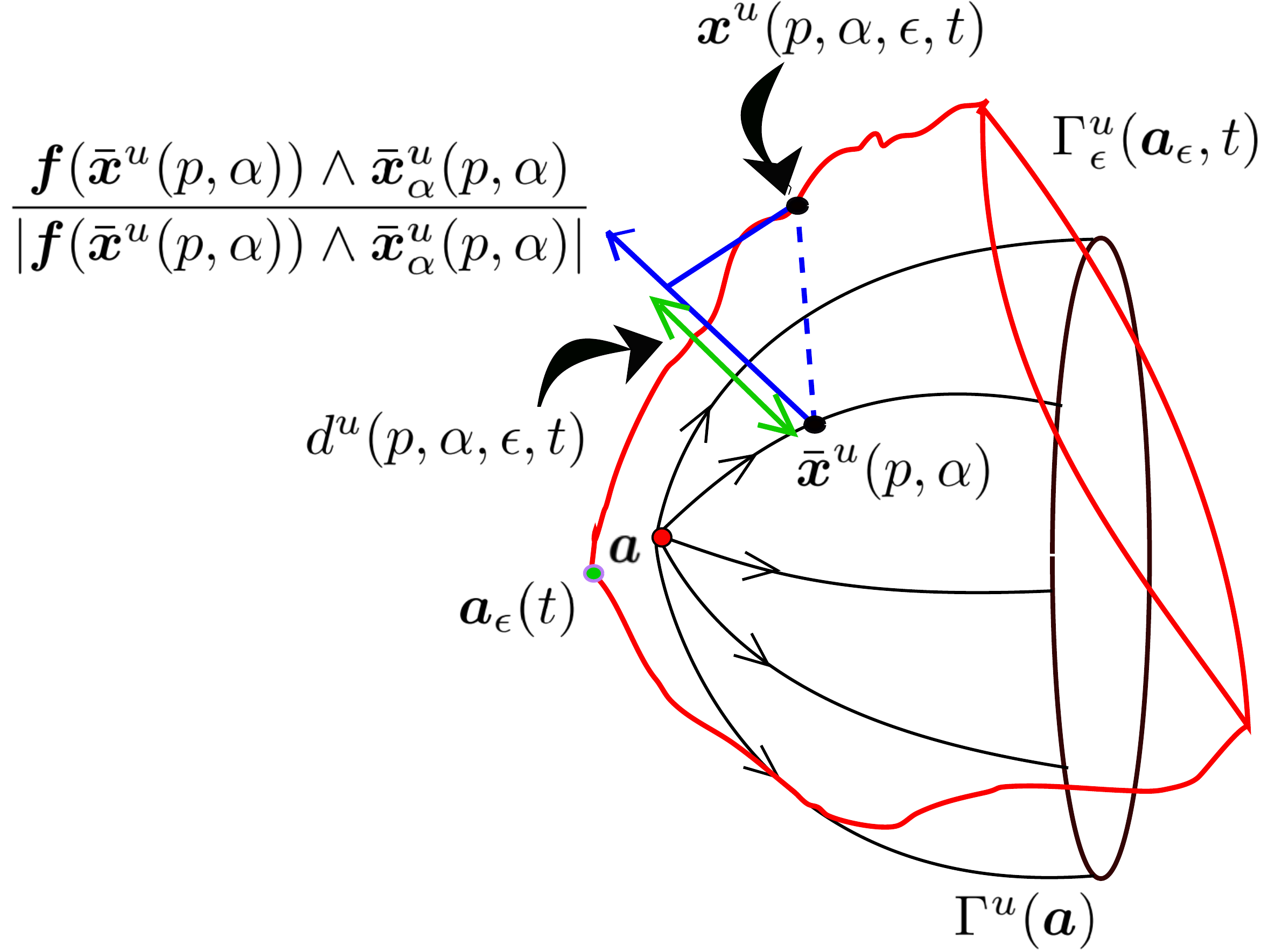}
	\caption{The perturbed unstable manifold $\Gamma^u_{\epsilon}(\vec{a}_\epsilon,t)$ (in red) at
	a general time $ t \in (-\infty,T] $, with the unperturbed unstable manifold, $\Gamma^u(\vec{a})$, shown in black.  
	The distance between perturbed and unperturbed unstable manifold denoted by $d^{u}(p,\alpha,\epsilon,t)$ is\ measured perpendicular to the original manifold.}
	\label{fig:unstable_pert}
\end{figure}

We note that the vector $\vec{f} \left( \bar{\vec{x}}^u(p, \alpha) \right)$ lies along a unstable manifold trajectory, since 
this is the velocity field at the point $  \bar{\vec{x}}^u(p, \alpha) $.  Moreover, $\bar{\vec{x}}^u_\alpha(p, \alpha)$, where
the subscript $ \alpha $ represents the partial derivative in this instance, is another vector which is tangential to
$\Gamma^u(\vec{a})$.  This vector must be transverse to  $\vec{f} \left( \bar{\vec{x}}^u(p, \alpha) \right)$; if tangential
at any value $ (p,\alpha) \in (-\infty,P] \times \mathrm{S}^1 $, that would relate to a failure of the trajectories (labelled by $ 
\alpha $) to foliate $ \Gamma^u(\vec{a}) $.  Thus, the standard cross
product between these two vectors is nonzero, and normal to $\Gamma^u(\vec{a})$ at $  \bar{\vec{x}}^u(p, \alpha) $.  We will
use the wedge notation for the cross product. Hence at time $t \in (-\infty, T]$, the distance between $\vec{x}^{u}(p,\alpha,\epsilon,t)~\text{and}~\bar{\vec{x}}^{u}(p, \alpha)$ is measured perpendicular to original unperturbed manifold can be represented as 
\begin{equation}
	d^{u}(p,\alpha,\epsilon,t) = \frac{ \vec{f}(\bar{\vec{x}}^{u}(p, \alpha)) \wedge \bar{\vec{x}}_\alpha^{u}(p, \alpha)}{\lvert \vec{f}(\bar{\vec{x}}^{u}(p, \alpha)) \wedge \bar{\vec{x}}_\alpha^{u}(p, \alpha)\rvert} \cdot \left[\vec{x}^{u}(p,\alpha,\epsilon,t)-\bar{\vec{x}}^{u}(p, \alpha)\right] 
	\quad , \quad (p,\alpha,t)  \in (-\infty,P] \times  \mathrm{S}^1  \times (-\infty,T] \, .
\label{eq:unstable_distance}
\end{equation}  

\begin{theorem}[Displacement of unstable manifold]
For $ (p,\alpha,t)  \in (-\infty,P] \times  \mathrm{S}^1 \times (-\infty,T] $, 
the distance $d^{u}(p,\alpha,\epsilon,t)$ can be expanded in $\epsilon$ in the form 
\begin{equation}
		d^{u}(p,\alpha,\epsilon,t) = \epsilon \frac{ M^u(p, \alpha, t)}{\lvert \vec{f}(\bar{\vec{x}}^{u}(p,\alpha)) \wedge \bar{\vec{x}}_\alpha^{u}(p, \alpha)\rvert} + \mathcal{O}(\epsilon^2),
\label{eq:distance_unstable}
\end{equation} 
where the unstable Melnikov function is the convergent improper integral 
\begin{equation}
			M^{u}(p,\alpha,t) = \int_{-\infty}^p \exp\left[{\int_{\tau}^p \nabla \cdot \vec{f}(\bar{\vec{x}}^{u}(\xi, \alpha)) d\xi}\right] \left[ \vec{f}(\bar{\vec{x}}^{u}(\tau, \alpha)) \wedge \bar{\vec{x}}_{\alpha}^{u}(\tau, \alpha)\right] \cdot \vec{g}(\bar{\vec{x}}^{u}(\tau, \alpha),\tau + t -p)~ \d \tau.
\label{eq:melnikov_unstable}
\end{equation}
\label{theorem:melnikov_unstable} 
\end{theorem}

\begin{proof}
This lengthy proof requires many stages, and is therefore given in Appendix~\ref{sec:normaldis3}.  Several results
which are ingredients in the proof are separated out into additional appendices for clarity.
\end{proof}

For volume-preserving unperturbed flows, the term $\nabla \cdot \vec{f}(\bar{\vec{x}}^{u}(\xi, \alpha))$ is zero,
and consequently the integrand of $ M^u $ loses the exponential term.  Additionally, we note that it is only the
normal component of the perturbation $ \vec{g} $, evaluated in appropriate retarded time, that contributes to the
leading-order normal displacement which is captured by $ M^u $.

\begin{remark}[Approximation of $ \Gamma_\epsilon^u $]
{\em
Theorem~\ref{theorem:melnikov_unstable} enables a natural approximation for $ \Gamma_\epsilon^u $, with
knowledge of the unperturbed flow and the perturbation velocity alone.
If $ (p,\alpha,t)  \in (-\infty,P] \times \mathrm{S}^1 \times (-\infty,T] $ are the parameters for the parametric representation
of a general point $ \vec{r}^u $ on  $ \Gamma_\epsilon^u $ at time $ t $, then
\begin{equation}
\vec{r}^u(p,\alpha,\epsilon,t) \approx  \bar{\vec{x}}^{u}(p, \alpha) + \epsilon \, M^u(p, \alpha, t) \frac{ \vec{f}(\bar{\vec{x}}^{u}(p, \alpha)) \wedge \bar{\vec{x}}_\alpha^{u}(p, \alpha)}{\lvert \vec{f}(\bar{\vec{x}}^{u}(p, \alpha)) \wedge \bar{\vec{x}}_\alpha^{u}(p, \alpha)\rvert^2} 
\label{eq:unstableapprox}
\end{equation}
provides a (leading-order in $ \epsilon $) parametric representation of $ \Gamma_\epsilon^u $. (While each 
trajectory on the manifold is known to exhibit a $ {\mathcal O}(\epsilon) $ tangential displacement as well, see \cite{tangential} for a quantification in two-dimensions, in a global view of the manifold as a collection of trajectories, using the normal displacement by itself provides an excellent approximation to the manifold.)
}
\label{remark:unstableapprox}
\end{remark}

\begin{remark}[Kernel of the Melnikov integral]
{\em
Melnikov functions in general dimensions using functional-analytic approaches \cite{chowhalemalletparet,palmer,battellilazzari,finite,lin}
 for determining persistent heteroclinic intersections usually
take the form
\[
M \sim \int_{-\infty}^\infty \vec{\Phi} \, \vec{g} \, \d \tau \, .
\]
Here, the entity $ \vec{\Phi} $ is associated with the fundamental matrix solution of the adjoint of the equation of
variations along the heteroclinic trajectory \cite{chowhalemalletparet,palmer,battellilazzari,finite,lin},
and is usually not computable except in two dimensions, or else if there is a Hamiltonian structure in the
unperturbed system \cite{gideadelallave}.  The reason for non-computability in general is that this is a {\em nonautonomous} linear equation, for which generally solutions cannot be explicitly written down; hence
the Melnikov approach is a interesting theoretical tool which replaces one issue (finding persistent heteroclinics) with another (finding zeros of a function with a kernel which satisfies a certain property, but which cannot in general be explicitly given by a formula).  Our geometric approach in three dimensions, leading to (\ref{eq:melnikov_unstable}), is the first insight into an explicit form of this kernel function when these conditions are relaxed.  (Note however that our limits
are not over all of $ \mathbb{R} $, because at this stage we are seeking the {\em location} of the perturbed manifold
rather than a persistent heteroclinic connnection.)
\label{remark:kernel}
}
\end{remark}

\subsection{Displacement of 2D stable manifold}
\label{sec:stable}

\begin{figure}[h]
\centering
	\includegraphics[scale=0.33]{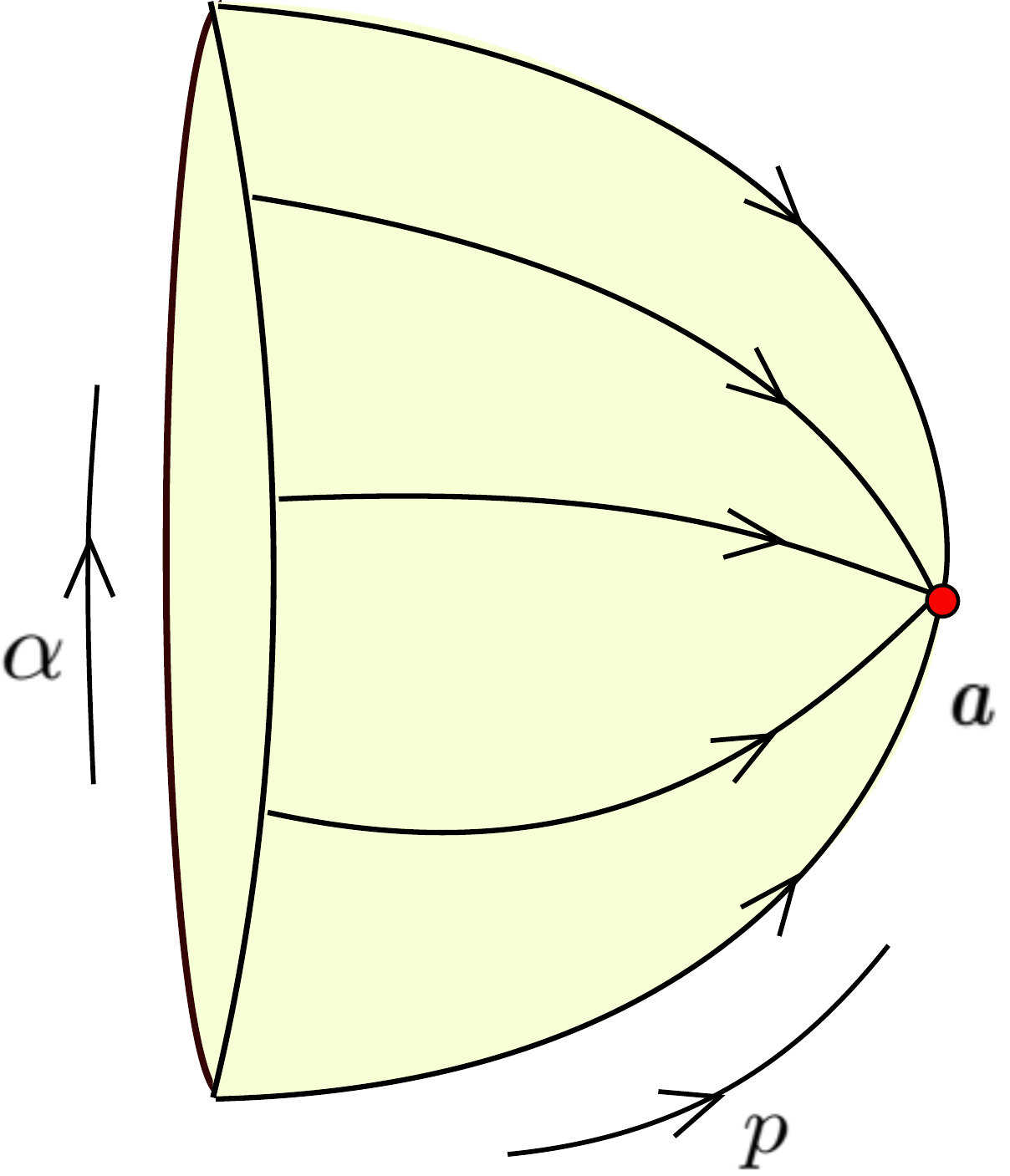}
	\caption{The two-dimensional unperturbed stable manifold $ \Gamma^s(\vec{a}) $ for case~2, with the red dot representing the saddle fixed point. The $ (p,\alpha) $-dependence of trajectories $(\bar{\vec{x}}^s(p, \alpha))$ can be used to parametrize $ \Gamma^s(\vec{a}) $.}
	\label{fig:stable_unp}
\end{figure}

Secondly, consider case~2,  when $ \vec{D} \vec{f} \left( \vec{a} \right) $ has one positive eigenvalue and two eigenvalues with negative real parts at the point $\vec{a}$. So the unperturbed system posses a one-dimensional unstable manifold and a two-dimensional stable manifold associated with the fixed point $ \vec{a} $. It is once again the displacement of the
two-dimensional entity that we capture, in this case, the stable manifold.  Rather than repeat the
development in detail, we will rely on Figs.~\ref{fig:stable_unp} and \ref{fig:stable_pert} which are exactly  
analogous to case~1's Figs.~\ref{fig:unstable_unp} and \ref{fig:unstable_pert}. The unperturbed stable
manifold $ \Gamma^s(\vec{a}) $ is foliated by trajectories $ \bar{\vec{x}}^s(p,\alpha) $ which forward asymptote to
$ \vec{a} $ as $ p \rightarrow \infty $ (Fig.~\ref{fig:stable_unp}). The perturbed stable manifold $ \Gamma^s_\epsilon(\vec{a}_\epsilon,t) $
is attached to the hyperbolic trajectory $ \left( \vec{a}_\epsilon(t), t \right) $ (Fig.~\ref{fig:stable_pert}).
A time $ t \in [T, \infty) $ (where $ T $ is finite) is chosen, and then the parameters $ p \in [P, \infty) $ (for finite $ P $)
and $ \alpha \in \mathrm{S}^1 $  parametrize $ \Gamma^s_\epsilon(\vec{a}_\epsilon,t) $.

As shown in Fig.~\ref{fig:stable_pert}, the perpendicular distance between perturbed stable manifold $(\Gamma_\epsilon^s(\vec{a}))$ and unperturbed stable manifold $(\Gamma^s(\vec{a}))$ at the location of $\bar{\vec{x}}^s(p, \alpha)$ in the time instance $t$ is given by 
\begin{equation}
	d^{s}(p,\alpha,\epsilon,t) = \frac{ \vec{f}(\bar{\vec{x}}^{s}(p, \alpha)) \wedge \bar{\vec{x}}_\alpha^{s}(p, \alpha)}{\lvert \vec{f}(\bar{\vec{x}}^{s}(p, \alpha)) \wedge \bar{\vec{x}}_\alpha^{s}(p, \alpha)\rvert} \cdot \left[\vec{x}^{s}(p,\alpha,\epsilon,t)-\bar{\vec{x}}^{s}(p, \alpha)\right] 
	\quad , \quad (p,\alpha,t)  \in [P, \infty) \times \mathrm{S}^1 \times [T, \infty) \, .
\label{eq:distance_stable}
\end{equation}

\begin{figure}[t]
	\centering
	\includegraphics[width=0.7 \textwidth]{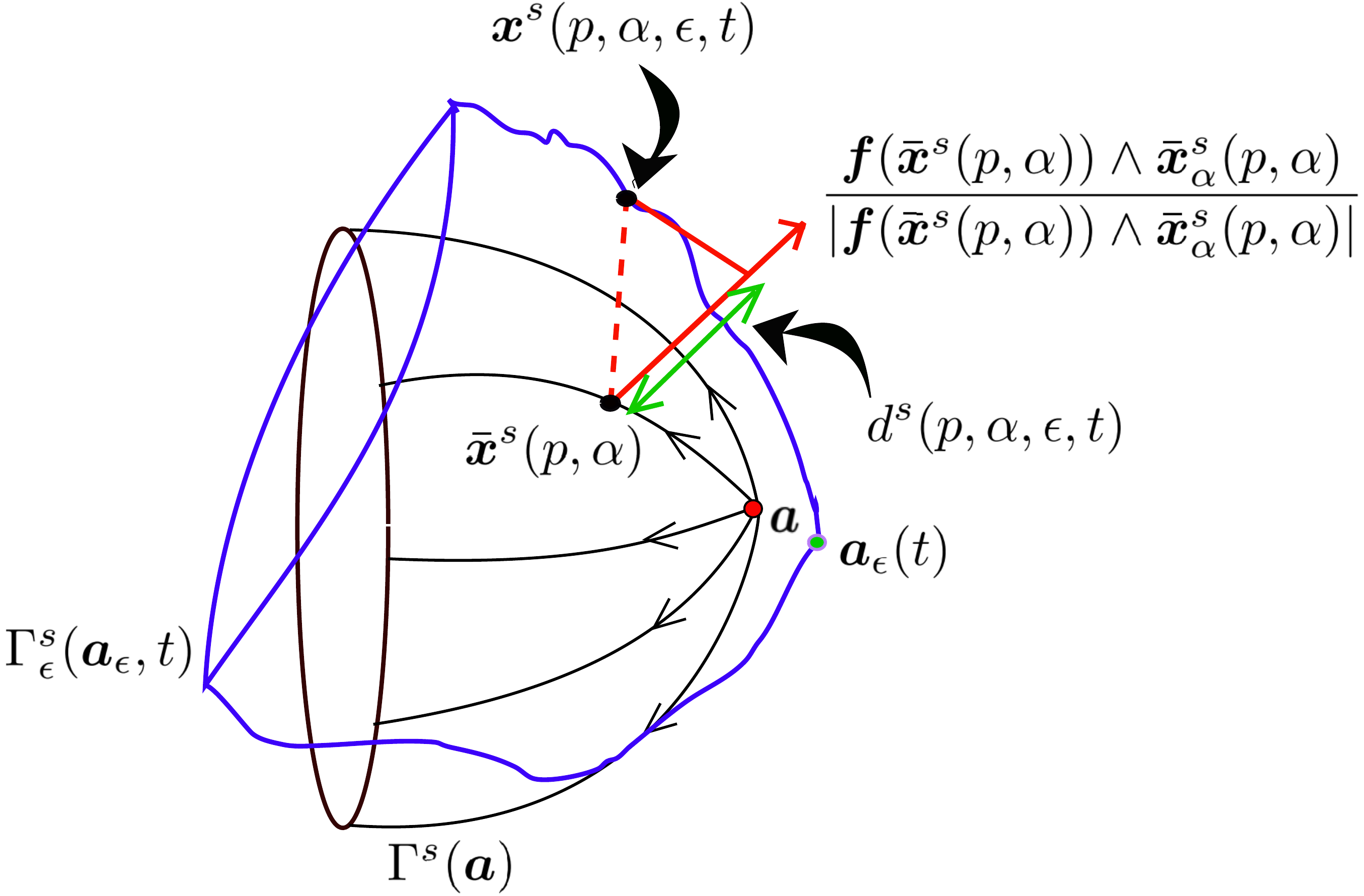}
	\caption{The perturbed stable manifold, which is denoted by $\Gamma^s_{\epsilon}(\vec{a}_\epsilon,t)$ (in blue), at
	a general time $ t \in [T, \infty) $, with the unperturbed stable manifold, $\Gamma^s(\vec{a})$ shown in black.  
	The distance between perturbed and unperturbed stable manifold denoted by $d^{s}(p,\alpha,\epsilon,t)$ is
	measured perpendicular to the original manifold.}
	\label{fig:stable_pert}
\end{figure}

\begin{theorem}[Displacement of stable manifold]
For $ (p,\alpha,t)  \in [P,\infty) \times  \mathrm{S}^1 \times [T,\infty)  $, 
the distance $d^{s}(p,\alpha,\epsilon,t)$ can be expanded in $\epsilon$ in the form \begin{equation}
		d^{s}(p,\alpha,\epsilon,t) = \epsilon \frac{ M^s(p, \alpha, t)}{\lvert \vec{f}(\bar{\vec{x}}^{s}(p, \alpha)) \wedge \bar{\vec{x}}_\alpha^{s}(p, \alpha)\rvert} + \mathcal{O}(\epsilon^2),
		\label{eq513}
	\end{equation} 
	where the stable Melnikov function is the convergent improper integral 
	\begin{equation}
			M^{s}(p,\alpha,t) = - \int^{\infty}_p \exp\left[{\int_{\tau}^p \nabla \cdot \vec{f}(\bar{\vec{x}}^{s}(\xi, \alpha)) d\xi}\right] \left[ \vec{f}(\bar{\vec{x}}^{s}(\tau, \alpha)) \wedge \bar{\vec{x}}_{\alpha}^{s}(\tau, \alpha)\right] \cdot \vec{g}(\bar{\vec{x}}^{s}(\tau, \alpha),\tau + t -p)~ \d \tau \, .
			\label{eq:melnikov_stable}
			\end{equation}
				\label{theorem:melnikov_stable} 
\end{theorem}

\begin{proof}
The proof is exactly analogous to that for Theorem~\ref{theorem:melnikov_unstable}, when we think of simply reversing
time.  Thus we think of $ p \rightarrow - p $ (this reverses the time parametrization along an unperturbed trajectory
on the 2D manifold), and $ t \rightarrow - t $.  All the ingredients of the proof essentially go through with this
understanding.   For the sake of brevity, no further details will be given.
\end{proof}

\begin{remark}[Approximation of $ \Gamma_\epsilon^s $]
{\em
If $ (p,\alpha,t)  \in  [P, \infty) \times  \mathrm{S}^1 \times [T, \infty)$ are the parameters for the parametric representation
of a general point $ \vec{r}^s $ on  $ \Gamma_\epsilon^s $ at time $ t $, then
\begin{equation}
\vec{r}^s(p,\alpha,\epsilon,t) \approx  \bar{\vec{x}}^{s}(p, \alpha) + \epsilon \, M^s(p, \alpha, t) \frac{ \vec{f}(\bar{\vec{x}}^{s}(p, \alpha)) \wedge \bar{\vec{x}}_\alpha^{s}(p, \alpha)}{\lvert \vec{f}(\bar{\vec{x}}^{s}(p, \alpha)) \wedge \bar{\vec{x}}_\alpha^{s}(p, \alpha)\rvert^2} 
\label{eq:stableapprox}
\end{equation}
provides a (leading-order in $ \epsilon $) parametric representation of $ \Gamma_\epsilon^s $. 
}
\end{remark}

\section{Melnikov function for heteroclinic manifolds} 
\label{sec:heteroclinic}

\begin{figure}[p]
	\centering
	\subfigure[]{\includegraphics[scale=0.29]{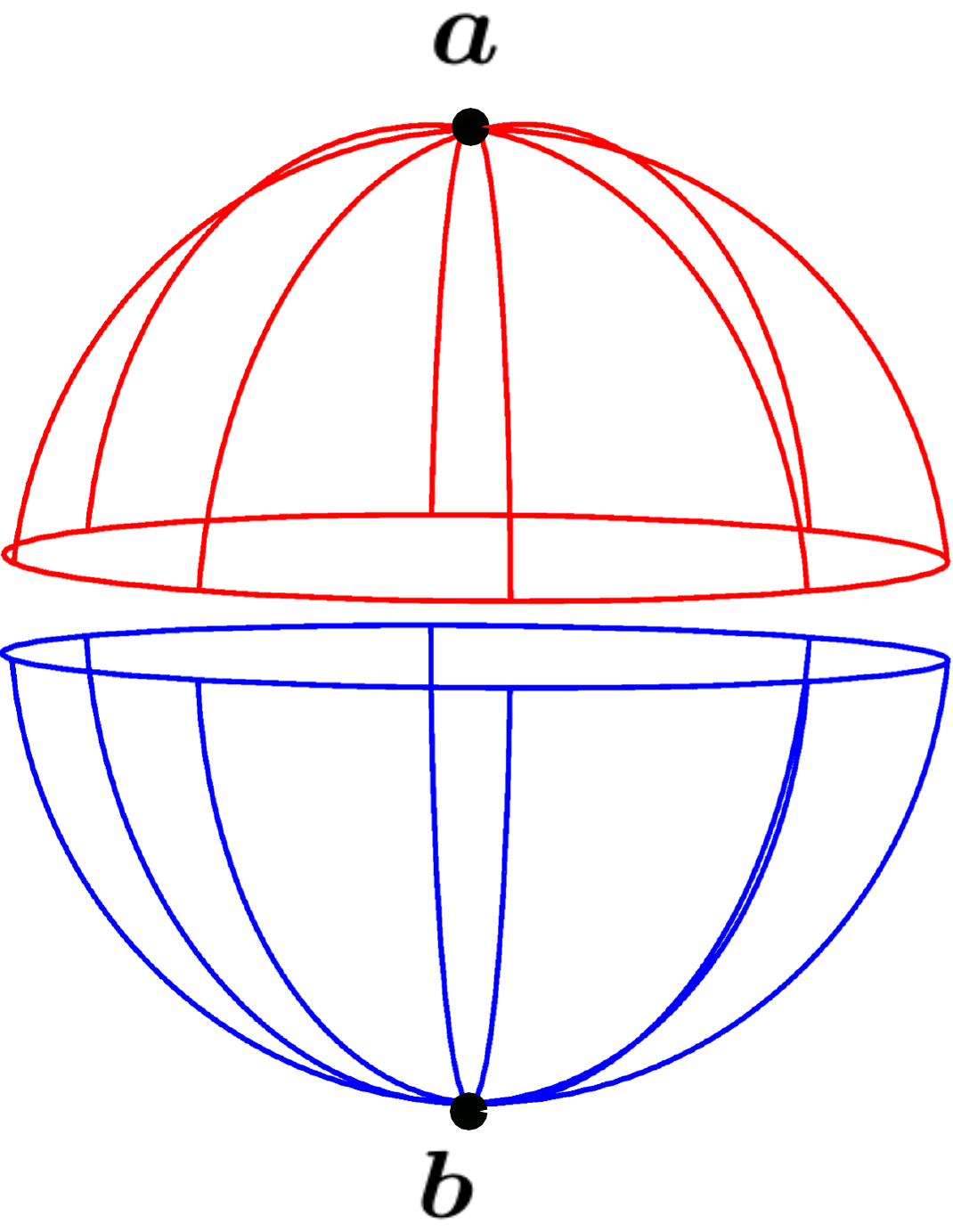}}
	\hspace{0.35in}
	\subfigure[]{\includegraphics[scale=0.29]{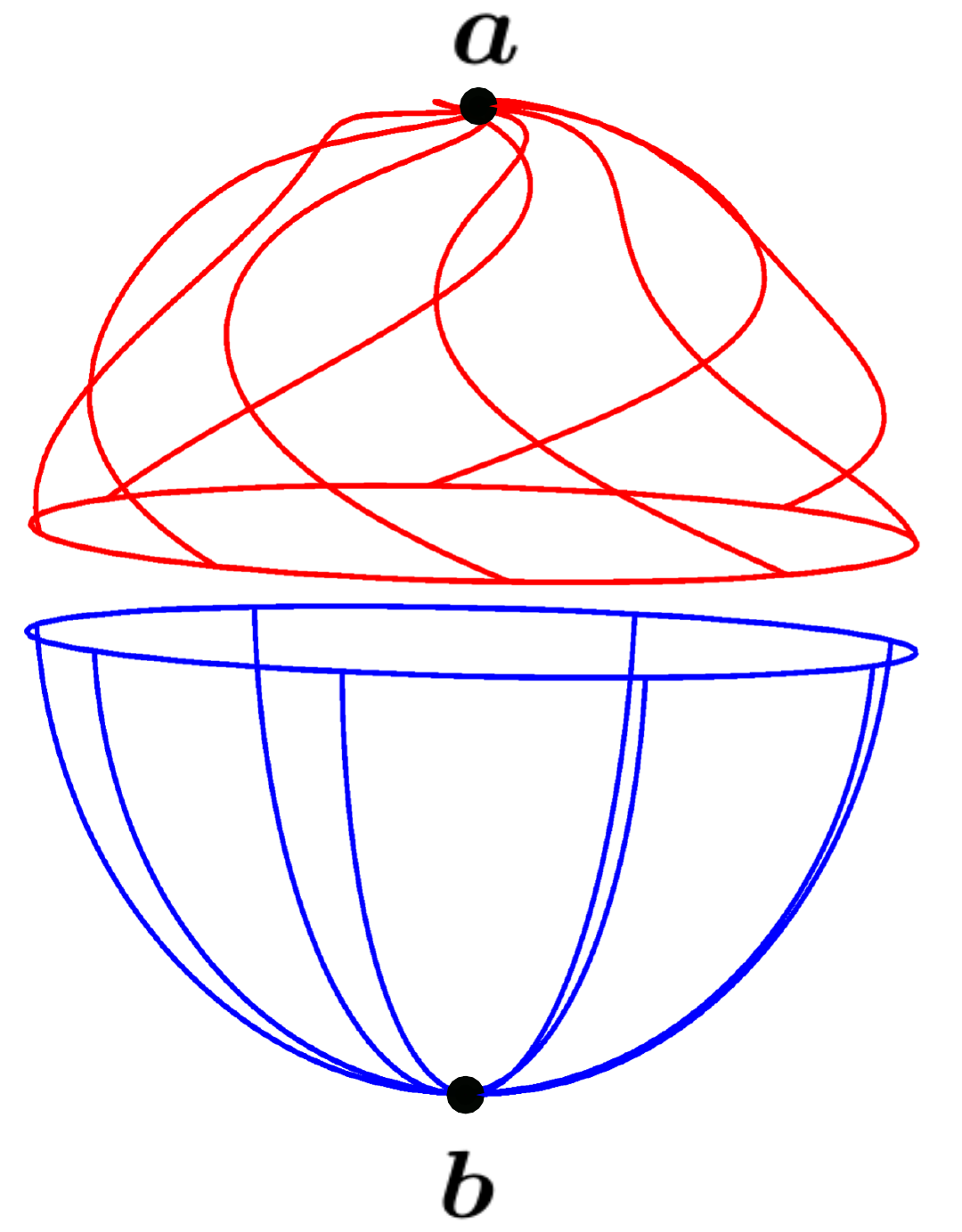}}
	\subfigure[]{\includegraphics[scale=0.285]{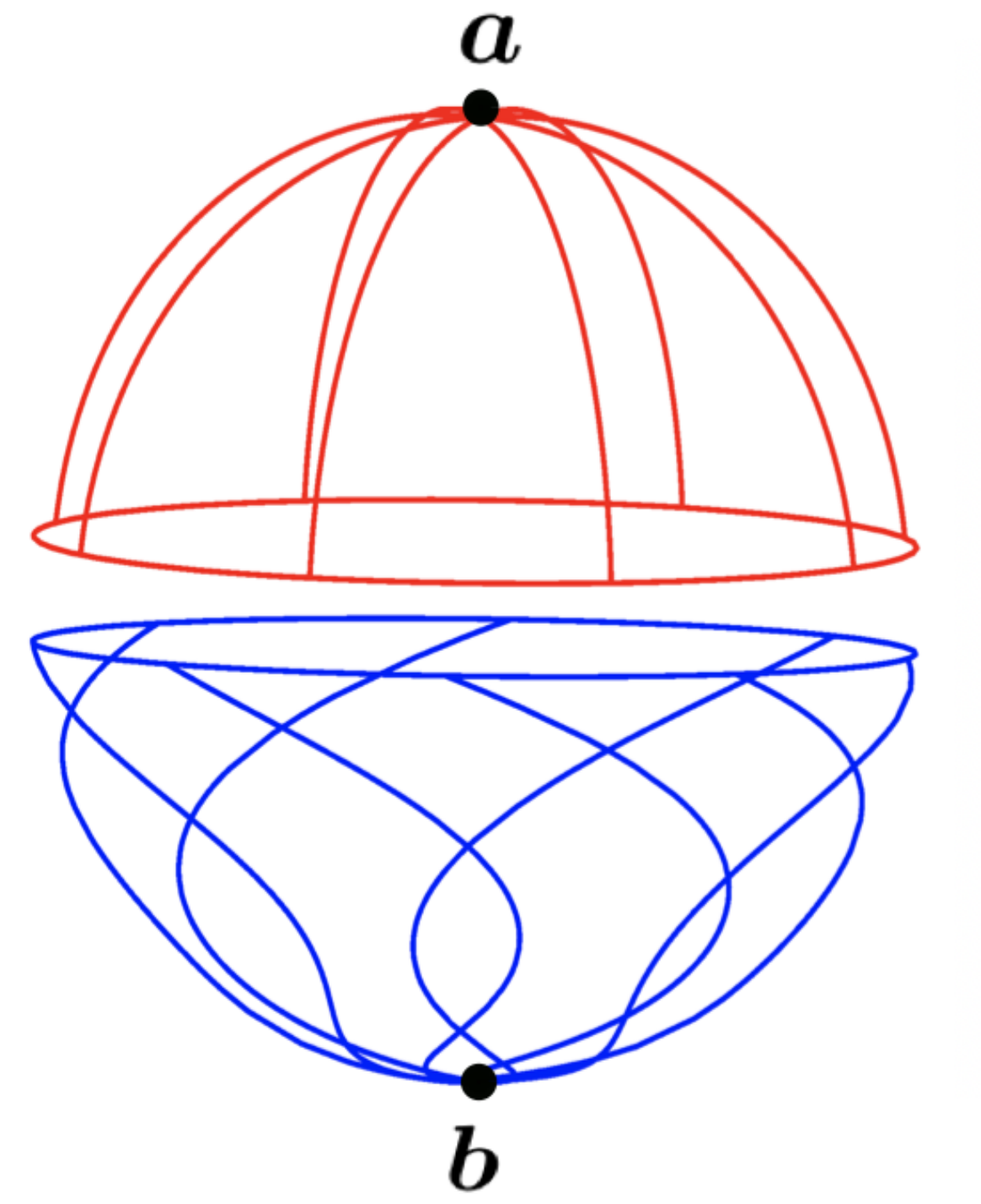}}
	\hspace{0.35in}
	\subfigure[]{\includegraphics[scale=0.29]{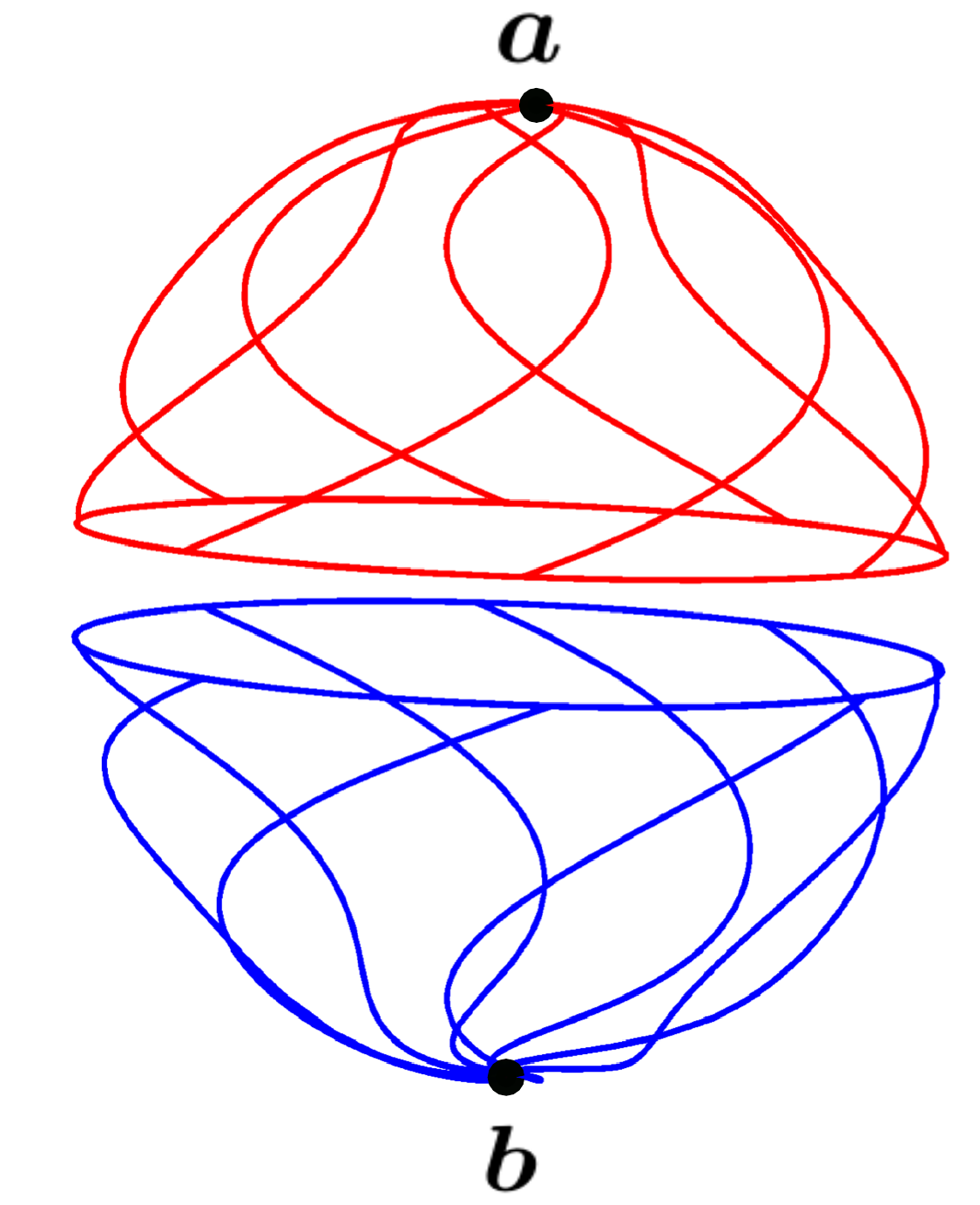}}
	\caption{Several trajectories $ \bar{\vec{x}}(p,\alpha) $ on two-dimensional unperturbed stable and unstable manifolds are given for four different situations. The case when no complex eigenvalues are present is shown in (a). The cases in which the complex conjugate eigenvalues are only on the unstable manifold (respectively only on the stable manifold), are given in (b) and (c). Finally, if
	the eigevalues at both $ \vec{a} $ and $ \vec{b} $ are complex-conjugate, we get the behavior 
	shown in (d).
	}
	\label{fig:hetero_unp}
\end{figure}

We now consider the implications of our theory in instances where the unperturbed system has a 
two-dimensional {\em heteroclinic} manifold.  This is the more `classical' setup of Melnikov theory.
However, rather than simply seeking a Melnikov function whose zeros imply persistence of a heteroclinic connection (which we will obtain), we do more:
\begin{itemize}
    \item Find ways of characterizing {\em lobes} created between perturbed stable and unstable manifolds,
    at each instance in time;
    \item Obtain a formula in terms of the Melnikov function for the volume of such a lobe;
    \item In this instance when the perturbation $ g $ is {\em harmonic}, that is when it can be written in the form $ \tilde{\vec{g}}(x) \cos \left[ \omega t + \phi \right] $, extend the two-dimensional theory of lobe dynamics via a turnstile \cite{romkedar,wigi_book} to three-dimensions;
    \item In the more general instance in which $ g $ has general time-dependence, extend the two-dimensional theory for instantaneous flux across the broken heteroclinic (formerly a flow barrier) to three dimensions;
    \item In the above instance, express the formula in simple terms in terms of the Melnikov function.
\end{itemize}
The last two of these issues is particularly important within the context of fluid flows: these are inevitably three-dimensional, and flow barriers are consequently two-dimensional.  By rationalizing an instantaneous flux across a flow barrier, we are quantifying an easily computable time-varying transport due to a perturbation which has general (aperiodic) time-dependence.

We assume that the smoothness hypotheses as stated in 
Section~\ref{sec:manifolds} continue to hold.  However,  assume now that unperturbed situation of the system (\ref{sys}) when $\epsilon = 0$ has two distinct fixed points $\vec{a}$ and $\vec{b}$, such that $ \vec{a} $ possesses a two-dimensional unstable manifold $ \Gamma^u $, and $ \vec{b} $ a two-dimensional stable manifold
$ \Gamma^s $.  Moreover, we assume that these
two manifolds coincide, to form a two-dimensional orientable heteroclinic manifold $ \Gamma $ as shown in Fig.~\ref{fig:hetero_unp}.  
Thus, a trajectory $ \bar{\vec{x}}^u(p,\alpha) $ on $ \vec{a} $'s unstable manifold coincides with a trajectory
$ \bar{\vec{x}}^s(p,\alpha) $ on $ \vec{b} $'s stable manifold.  We denote this heteroclinic trajectory by
$ \bar{\vec{x}}(p,\alpha) $, and insist that $ (p,\alpha) \in \R \times \mathrm{S}^1 $ be chosen
such that $ (p,\alpha) $ provides a $ \mathrm{C}^2 $-smooth parameterization of the heteroclinic manifold. Once
again, $ p $ can be thought of as the time-evolution along the trajectory, with $ \alpha $ choosing {\em which}
trajectory.  Thus we have the limiting behavior
\[
\lim_{p \rightarrow -\infty} \bar{\vec{x}}(p,\alpha) = \vec{a} \quad {\mathrm{and}} \quad 
\lim_{p \rightarrow \infty} \bar{\vec{x}}(p,\alpha) = \vec{b} \, 
\]
for each and every $ \alpha $.
We have shown the spiralling situation in Fig.~\ref{fig:hetero_unp}(d), corresponding to $ 
\lambda_1^u(\vec{a}) $ and $ \lambda_2^u(\vec{a}) $ (eigenvalues with the positive real parts obtained at $\vec{a}$) being complex conjugates of one another, as are
$ \lambda_1^s(\vec{b}) $ and $ \lambda_2^s(\vec{b}) $ (eigenvalues with the negative real parts obtained at $\vec{b}$).  However, the heteroclinic trajectories need not be spiralling
in general; that is, these sets of eigenvalues could be purely positive (for $ \vec{a} $) and purely negative (for $ \vec{b} $). This situation is shown in Fig.~\ref{fig:hetero_unp}(a).
Indeed, it is allowable to have complex eigenvalues only at one of the endpoints $ \vec{a} $ or $ \vec{b} $ and at the other endpoint, we can have purely positive or negative eigenvalues appropriately. These are given in Figs.~\ref{fig:hetero_unp}(b) and (c).

The unperturbed heterclinic manifold $ \Gamma $ is a closed surface, separating phase space into a part which is inside $ \Gamma $, and another which is outside.  To distinguish between the two, we will assume that the $ \alpha $-labelling in the
$ (p,\alpha) $ parametrization of $ \Gamma $ is chosen such that
\[
\frac{\partial \bar{\vec{x}}(p, \alpha)}{\partial p} \wedge \frac{\partial \bar{\vec{x}}(p, \alpha)}{\partial \alpha}
= \vec{f}(\bar{\vec{x}}(p, \alpha)) \wedge \bar{\vec{x}}_\alpha(p, \alpha)
\]
(which we know is normal to $ \Gamma $) points in the {\em outwards} direction to $ \Gamma $ at each point $ \vec{x}(p,\alpha) $ on $ \Gamma $. 

Now, when $ \epsilon \ne 0 $ but is small in size, we know that $ \vec{a} $ becomes a heteroclinic trajectory,
and its unstable manifold persists as a time-parameterized entity $ \Gamma^u_\epsilon(\vec{a}_\epsilon,t) $.
This can be parameterized by $ (p,\alpha,t) \in (-\infty,P^u] \times \mathrm{S}^1 \times (-\infty,T^u] $ for
any finite $ P^u $ and $ T^u $.  
 We note that we can take $ P^u $ as large as we like (but remaining finite), indicating that we cannot find an approximation to the manifold globally.  Intuitively, 
this means that a perturbed version of the top-hemisphere of Fig.~\ref{fig:hetero_unp} persists; we may 
approach the south pole, but begin to lose control as we do so.  The manifold may continue beyond this region, and if
so, our theory is not able to approximate it. Similarly, the stable manifold of $ \vec{b} $
persists as $ \Gamma^s_\epsilon(\vec{b}_\epsilon,t) $, and we can parameterize this by $(p,\alpha,t) \in [P^s,\infty) \times 
\mathrm{S}^1 \times [T^s, \infty) $ for any finite $ P^s $ and $ T^s $, which we can choose to be as negative as we like.

\begin{figure}[t]
	\centering
	\includegraphics[scale=0.45]{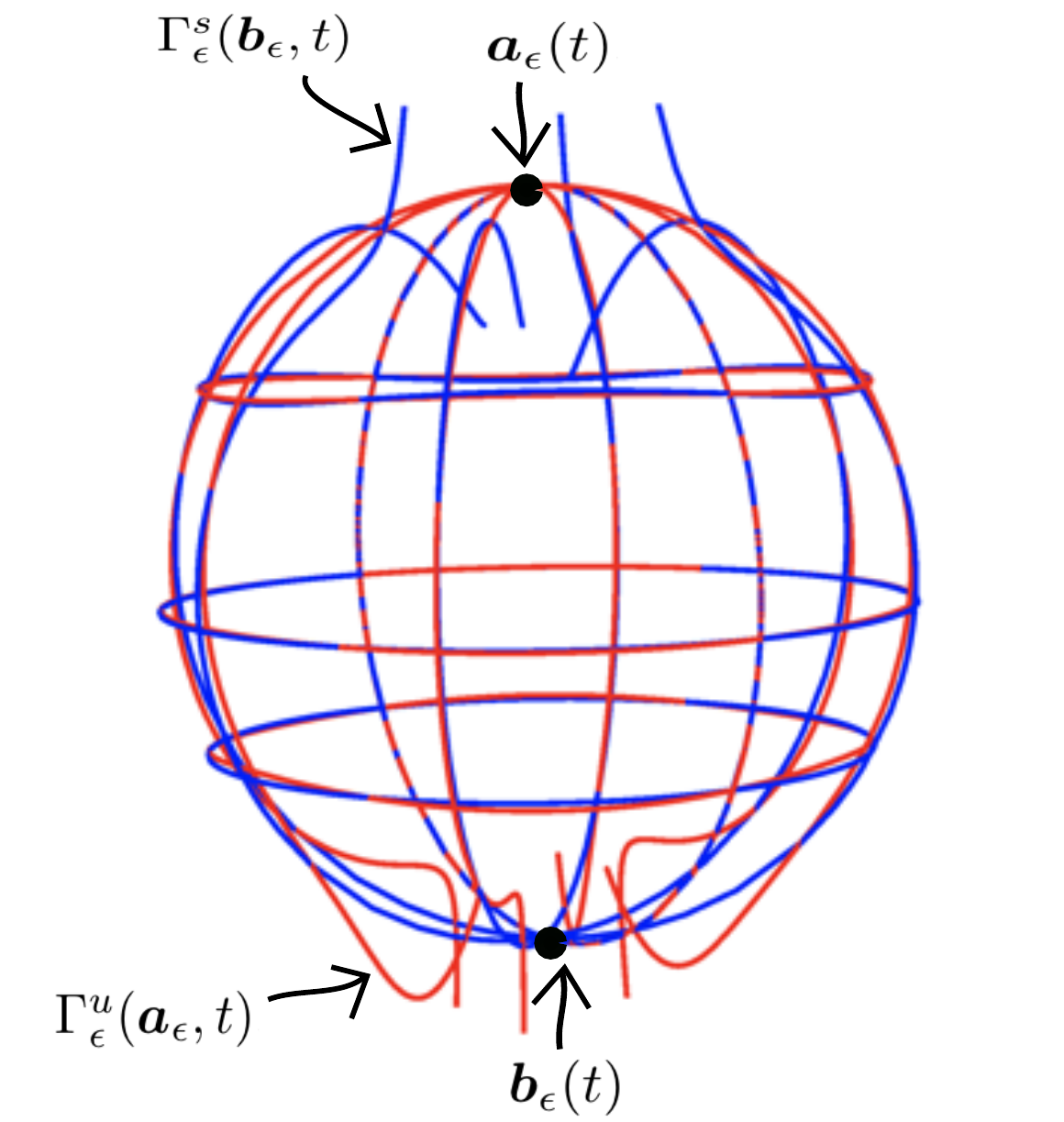}
	\caption{A generic intersection pattern of the perturbed unstable ($ \Gamma_\epsilon^u(\vec{a}_\epsilon,t) $, red) and stable ($ \Gamma_\epsilon^s(\vec{b}_\epsilon,t) $, blue) manifolds of the Hill's spherical vortex with no swirl at time $ t=1 $ and $\eps=0.1$. 
	}
	\label{fig:hetero_pert}
\end{figure}

Take a time $ t \in [T^s, T^u] $. The unstable manifold of $ \vec{a}_\epsilon $
for time in $ (-\infty, t] $ and the stable manifold of $ \vec{b}_\epsilon $ for time in $ [t ,\infty) $ can both be approximated using our previous results.  The manifolds at time $ t $ no longer need to
coincide, and we show a generic situation in  Fig.~\ref{fig:hetero_pert}. There is now a perturbed version of $ \bar{\vec{x}}(p,\alpha) $ on $ \Gamma_\epsilon^u(\vec{a}_\epsilon,t) $, which we can call $ \vec{x}^u(p,\alpha,\epsilon,t) $, 
which exists for $ p \in (-\infty, P^u] $.  Similarly, there exists a perturbed version of $ \bar{\vec{x}}(p,\alpha) $ on
$ \Gamma_\epsilon^s(\vec{b}_\epsilon,t) $, which we call $ \vec{x}^s(p,\alpha,\epsilon,t) $, which exists for
$ p \in [P^s, \infty) $.  

Given a location $ (p,\alpha) $ on $ \Gamma $, we first want to quantify the displacement of
between the perturbed unstable and stable manifolds, in the direction normal to $ \Gamma $:
\begin{equation}
	d(p,\alpha,\epsilon,t) = \frac{ \vec{f}(\bar{\vec{x}}(p, \alpha)) \wedge \bar{\vec{x}}_\alpha(p, \alpha)}{\lvert \vec{f}(\bar{\vec{x}}(p, \alpha)) \wedge \bar{\vec{x}}_\alpha(p, \alpha)\rvert} \cdot \left[\vec{x}^{u}(p,\alpha,\epsilon,t)- \vec{x}^{s}(p,\alpha,\epsilon,t)\right] 
	\quad , \quad (p,\alpha,t)  \in [P^s,P^u] \times \mathrm{S}^1 \times [T^s,T^u] \, .
\label{eq:distance}
\end{equation}  
Given our choice of labelling of the unperturbed heteroclinic trajectories on $ \Gamma $, we note that
a positive $ d(p,\alpha,\epsilon,t) $ implies that the unstable manifold is {\em outside} the stable manifold, while a negative
$ d $ means that the stable manifold is outside the unstable one at a location $ \bar{\vec{x}}(p,\alpha) $
at a time instance $ t $.

\begin{theorem}[Heteroclinic splitting]
For $ (p,\alpha,t)  \in [P^s,P^u] \times  \mathrm{S}^1 \times [T^s,T^u] $, 
the distance (\ref{eq:distance}) can be expanded in $\epsilon$ in the form 
\begin{equation}
	d(p,\alpha,\epsilon,t) = \epsilon \frac{ M(p, \alpha, t)}{\lvert \vec{f}(\bar{\vec{x}}(p, \alpha)) \wedge \bar{\vec{x}}_\alpha(p, \alpha)\rvert} + \mathcal{O}(\epsilon^2),
	\label{eq:distancemelnikov}
	\end{equation} 
where the Melnikov function is the convergent improper integral 
\begin{equation}
	M(p,\alpha,t) = \int^{\infty}_{-\infty} \exp\left[{\int_{\tau}^p \nabla \cdot \vec{f}(\bar{\vec{x}}(\xi, \alpha)) d\xi}\right] \left[ \vec{f}(\bar{\vec{x}}(\tau, \alpha)) \wedge \bar{\vec{x}}_{\alpha}(\tau, \alpha)\right] \cdot \vec{g}(\bar{\vec{x}}(\tau, \alpha),\tau + t -p)~ \d \tau \, .
\label{eq:melnikov}
\end{equation}
\label{theorem:heteroclinic} 
\end{theorem}
\begin{proof}
	See Appendix~\ref{sec:melnikov}.
\end{proof}

Note that the distance function $ d $ in (\ref{eq:distancemelnikov}) at fixed $ t $ can be thought of in the sense
of first taking a point on $ \Gamma $ parametrized by $ (p,\alpha) $, i.e., the point 
$ \bar{\vec{x}}(p,\alpha) $, drawing an outward-pointing normal vector to $ \Gamma $ at that point, and determining the signed distance along that normal vector. Thus, we can think of {\em projecting} this
distance information between perturbed stable and unstable manifolds on to $ \Gamma $.  

\begin{remark}[Shifts of heteroclinic points]
{\em 
Each heteroclinic intersection point simply `shift along' the $ p $-location as time $ t $
is varied in forward and backward time, because of the limiting behavior expressed by (\ref{eq:hetero_limit}). Indeed, (\ref{eq:melnikov}) shows that
\[
M(\tilde{p}+t,\tilde{\alpha},\tilde{t}+t) = \exp\left[{\int_{\tilde{p}}^{\tilde{p}+t} \nabla \cdot \vec{f}(\bar{\vec{x}}(\xi, \alpha)) d\xi}\right]  \,
M(\tilde{p},\tilde{\alpha},\tilde{t}) 
\]
for any time-shift $ t $, and hence if there is a zero at a time $ \tilde{t} $ at a location encoded by
the parameter $ \tilde{p} $, then there is correspondingly a zero at the shifted time $ \tilde{t} + t $ at a parameter value $ \tilde{p}+t $, at the same $ \tilde{\alpha} $ value (i.e., traversing along the same unperturbed heteroclinic trajectory).  Effectively, tracking the location of this point with time
gives the location subtended on $ \Gamma $ of the corresponding perturbed heteroclinic 
trajectory as a function of time $ t $.
This fact is illustrated most strikingly for the volume-preserving case, since $ M $ will depend on $ (p,t) $ not independently,
but in terms of the shift $ (p-t) $.
}
\label{remark:heteroclinicshift}
\end{remark}

\begin{remark}[Transverse intersection points]
{\em 
At a fixed value $ \tilde{t} $ of time, a transverse intersection between $ \Gamma_\epsilon^u(\vec{a}_\epsilon) $ and
$ \Gamma_\epsilon^s(\vec{b}_\epsilon) $ near a point $ \bar{\vec{x}}(p,\alpha) $ is guaranteed if $ M(p,\alpha,t) $
has a simple zero with respect to $ (p,\alpha) $; that is, if there exists $ (\tilde{p},\tilde{\alpha}) $ such that
$ M(\tilde{p},\tilde{\alpha},\tilde{t}) = 0 $ and 
\begin{equation}
\left| \frac{\partial M}{\partial p} (\tilde{p},\tilde{\alpha},\tilde{t})\right| + \left| \frac{\partial M}{\partial \alpha} (\tilde{p},\tilde{\alpha},\tilde{t})\right| \ne 0 \, .
\label{eq:transverse}
\end{equation}
This is a standard consequence of implicit function theorem type arguments in this setting (see \cite{arrowsmith_place,nonosci}).  Each such intersection point $ \vec{x}(\tilde{p},\tilde{\alpha},\tilde{t}) $ corresponds
to a heteroclinic trajectory, i.e., the trajectory $ \vec{x}(\tilde{p},\tilde{\alpha},t) $ passing through
$ \vec{x}(\tilde{p},\tilde{\alpha},\tilde{t}) $ at time $ \tilde{t} $ satisfies
\begin{equation}
\lim_{t \rightarrow -\infty} \left| \vec{x}(\tilde{p},\tilde{\alpha},t) - \vec{a}_\epsilon(t) \right| = 0 \quad {\mathrm{and}} \quad
\lim_{t \rightarrow \infty} \left| \vec{x}(\tilde{p},\tilde{\alpha},t) - \vec{b}_\epsilon(t) \right| = 0 \, .
\label{eq:hetero_limit}
\end{equation}
While this is a generic result in standard Melnikov developments, having an explicit form within the integral in (\ref{eq:melnikov}) in a dimension larger than $ 2 $ without a Hamiltonian structure is new.
}
\label{remark:transverse}
\end{remark}

\begin{remark}[Curves of heteroclinic points]
{\em 
Suppose there exists continuously differentiable functions $ \tilde{\alpha}(s) $ and $ \tilde{p}(s) $ for
$ s \in (0,1) $ such that at a fixed $ t $
\[
M\left( \tilde{p}(s),\tilde{\alpha}(s), t\right) = 0 \quad \mathrm{and}
\quad \vec{\nabla}_{p,\alpha} M \left( \tilde{p}(s),\tilde{\alpha}(s)\right) \ne \vec{0} 
\quad \mathrm{for~all}~s \in (0,1) \quad \, , 
\]
where $  \vec{\nabla}_{p,\alpha} $ is the two-dimensional gradient with respect to $ (p,\alpha) $. Assume
moreover that there exists a constant $ H $ such that
\[
 \sup_{s \in (0,1)} \left| 
\vec{\nabla}_{p,\alpha} M \left( \tilde{p}(s),\tilde{\alpha}(s)\right)  \right| < H \, .
\]
The parametrization $ \left( \tilde{p}(s), \tilde{\alpha}(s) \right) $
represents a non-degenerate {\em curve} of points, $ Q $, parametrized by $ s \in (0,1) $ and at the locations $ \bar{\vec{x}} \left( \tilde{p}(s),
\tilde{\alpha}(s) \right) $, along which $ M $ is zero.  (This is the generic expectation for the 
intersection between the two two-dimensional perturbed manifolds.) Applying the Banach space version of the
implicit function theorem to $ d(p,\alpha,\epsilon,t)/\epsilon $ then gives the persistence of a $ \mathrm{C}^1 $-curve $ Q^\star $ for small enough $ \left| \epsilon \right| $, which is
 associated with the parametrization $ \left( p^\star(s), \alpha^\star(s) \right) $ which is 
 $ {\mathcal O}(\epsilon)$-close to $ \left( \tilde{p}(s), \tilde{\alpha}(s) \right) $.  Note that
 the same implication arises if $ s \in \mathrm{S}^1 $, i.e., we think of $ Q $ as a {\em closed} curve.
 In other words,
 these conditions imply the presence of a curve of heteroclinic points $ Q^\star $ which is 
 $ {\mathcal O}(\epsilon)$-close
 to that predicted via the Melnikov function.  Thus, each point on $ Q^\star $ obeys Remark~\ref{remark:transverse}, and is an `initial condition' for a heteroclinic trajectory of the 
 time-varying problem which satisfies the conditions (\ref{eq:hetero_limit}).
}
\label{remark:heterocliniccurves}
\end{remark}

A possible intersection between a perturbed stable manifold $\Gamma_{\epsilon}^s\left(\vec{b}_\epsilon, t\right)$ and perturbed unstable manifold $\Gamma_{\epsilon}^u\left(\vec{a}_\epsilon, t\right)$ at a fixed time $ t $ is shown in Fig.~\ref{fig:hetero_pert}. Here, we have used red and blue colors to represent the perturbed unstable and stable manifolds respectively, a convention we will follow in the remainder of this paper.  We recall 
that in two-dimensional flows in which a one-dimensional heteroclinic splitting is captured via a Melnikov function,
the `lobe area' between two adjacent intersections of the split manifolds can be obtained by integrating the
Melnikov function \cite{romkedar,wigi_book}.  An analogous result is available in the present three-dimensional setting.
If the Melnikov function $ M $ has a closed curve $ Q $ on $ \Gamma $ along which $ M $ has non-degenerate zeros as explained in Remark~\ref{remark:heterocliniccurves}, then there is a $ {\mathcal O}(\epsilon) $-close closed curve $ Q^\star $ on 
$ \Gamma $ which corresponds to the normal projection on to $ \Gamma $ of the intersection ring between
perturbed stable and unstable manifolds.  If $ M $ is sign definite inside $ Q $, the interior of $ Q^\star $ is the `shadow' of the one lobe which is generated by this intersection, i.e., the projection of the lobe
on to $ \Gamma $ along normal vectors to $ \Gamma $.  Then, we can give an expression for the lobe volume in terms of the Melnikov function:

\begin{theorem}[Lobe volume]
Let $ t \in [T^s,T^u] $ be fixed, and suppose that there is an open region $ R $ on $ \Gamma $ in which $ M(p,\alpha,t) $ is sign definite, and such that $ M = 0 $ and $ \vec{\nabla}_{p,\alpha} M \ne \vec{0} $
at all points on its boundary $ Q $. Note that $ Q $ will generically consist of a finite number of
closed curves.  Let $ R' $ be the region in $ (p,\alpha) $-space corresponding to $ R $.  The region between the perturbed stable and manifolds which is subtended by $ R $ is a {\em lobe}, whose volume is given by
\begin{equation}
{\mathrm{Lobe~volume}} \, = \epsilon \int \! \! \! \! \int_{R'} \left| M(p,\alpha,t) \right| \, \d p \, \d \alpha + {\mathcal O}(\epsilon^2) \, .
\label{eq:lobevolume}
\end{equation}
\label{theorem:lobevolume}
\end{theorem}

\begin{proof}
	See Appendix~\ref{sec:lobevolume}.
\end{proof}

\begin{remark}
{\em The fact that a lobe volume can be represented to leading-order by an integral of the Melnikov function over the $ p $-variable (representing the time-parametrization along a heteroclinic trajectory) is exactly analogous to a well-known similar result for lobe areas in the two-dimensional situation
\cite{romkedar,wigi_book}.  We point out that it is {\em not} necessary to impose additional hypotheses
such as volume-preservation or time-periodicity for this result to be true; one just needs to know the $ (p,\alpha) $-region associated with the particular lobe of interest.
}
\end{remark}
 
 In the unperturbed situation, the stable and unstable manifold coincide to form a heteroclinic manifold,
 which separates the `inside' and the `outside' flow regimes.  However, after perturbation this entity splits into a stable and unstable manifold, and consequently is no longer a flow separator.  Hence, transport will now occur across the previously impermeable structure.  Understanding the lobe volume would seem to be relevant in quantifying this transport.  We show in Section~\ref{sec:lobedynamics} how the two-dimensional theory of lobe dynamics \cite{romkedar,wigi_book} can be extended to our situation, to allow for a lobe volume to quantify transport.  However, a lobe volume can only be unambiguously assigned as a measure of transport under several additional assumptions, including a restrictive type of time-periodicity of the perturbation $ \vec{g} $, and volume preservation of the unperturbed flow.  In more general situations, a lobe volume cannot be used to quantify transport, because it turns out that there may be either many different-sized lobes, or no lobes at all.  In this case, an alternative approach, which characterizes an instantaneous flux as a time-varying quantity, is necessary.  This more general approach is described in Section~\ref{sec:flux}, which builds on a similar approach in two dimensions \cite{aperiodic,rossbyflux,siam_book}.

\subsection{Lobe dynamics under additional conditions}
\label{sec:lobedynamics}

Lobe dynamics \cite{romkedar,wigi_book} is a well-established theory for describing and quantifying
transport across a broken heteroclinic manifold in two dimensions, under certain conditions. The principal
assumption is that the perturbation is periodic in time, and that there is an 
intersection between the stable and unstable manifolds at a time $ t $, at a point $ \vec{q} $.  In this section, we show how under similar assumptions, we can express transport in our three-dimensional situation as a direct
analog of lobe dynamics.

\begin{figure}[t]
	\centering
	\includegraphics[width=92mm,height=92mm,scale=0.7]{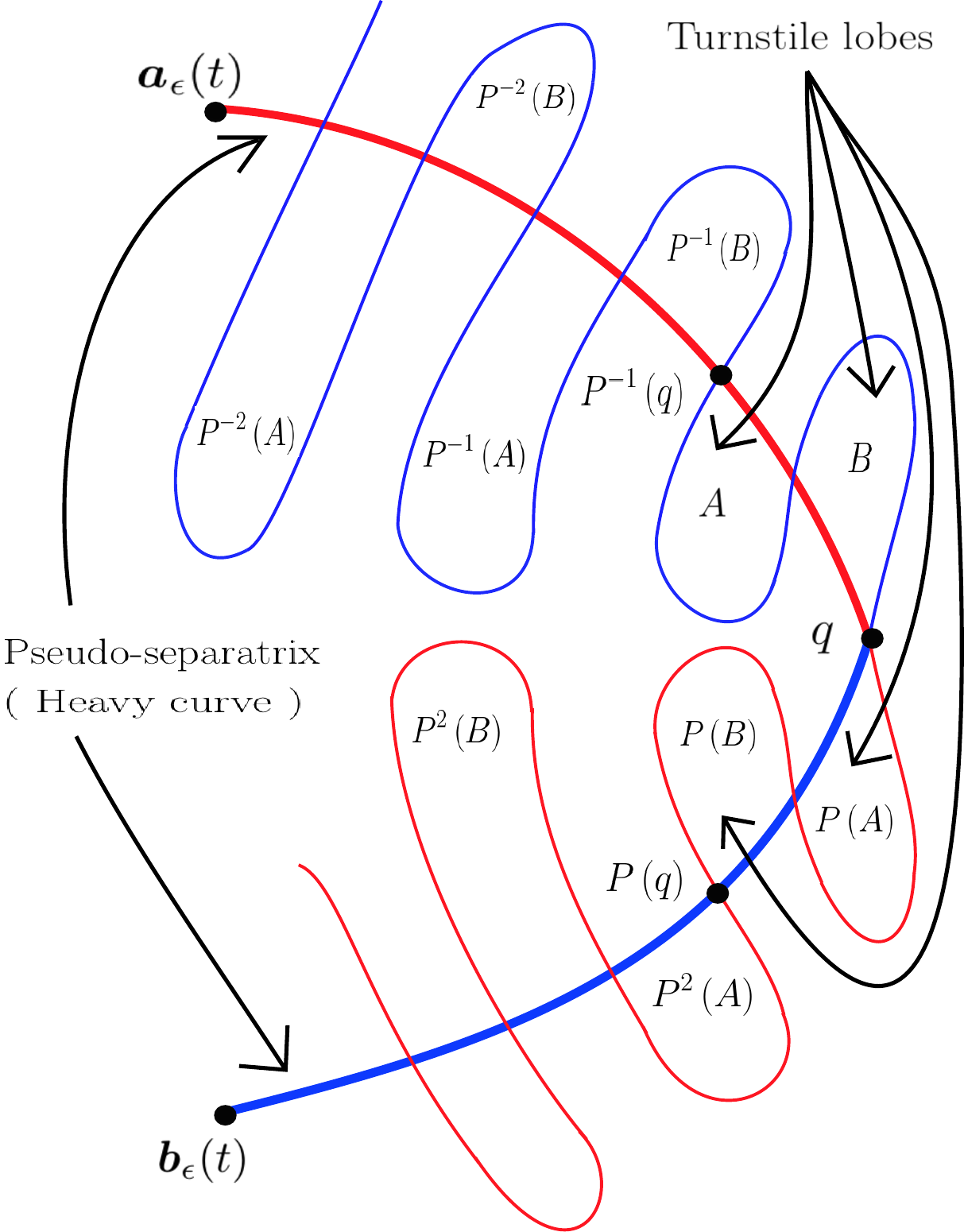}
	\caption{The two-dimensional lobe dynamics scenario \cite{romkedar,wigi_book}, occurring due to the
	intersection of the perturbed unstable manifold of $ \vec{a}_\epsilon(t) $ (red) and stable manifold of
	$ \vec{b}_\epsilon(t) $ (blue) at at a fixed time $ t $.  It is only the {\em turnstile lobes} between $ P^{-1}(\vec{q}) $ and $ P(\vec{q}) $ which are involved in crossing the pseudo-separatrix (heavy curve) per iteration of the Poincar\'{e} map. }
	\label{fig:2dlobes}
\end{figure}

It will help to briefly describe the elegant two-dimensional theory \cite{romkedar,wigi_book} first.  Suppose the perturbation is periodic in time (with period $ T $), and consider the manifold intersection pattern at a fixed time $ t $.  This picture is drawn under the assumption that the perturbation is {\em time-harmonic}, 
that is, can be written in the form $ \tilde{\vec{g}} \cos \left[ \omega t + \phi \right] $ for constant
frequency $ \omega \ne 0 $ and phase $ \phi $.  It turns out that then the (two-dimensional) Melnikov
function which captures intersections is itself sinusoidal with frequency $ \omega $, and the implication
is that the intersection of the perturbed unstable manifold in relation to the stable one is topologically
equivalent to the intersection of the sinusoidal function with the horizontal axis.
Thus there are infinitely many isolated transverse intersections. (This by itself does not imply chaotic motion in the heteroclinic situation; additional geometric conditions may be
necessary \cite{bertozzi,article_chaosdg}.  If homoclinic, though, the Smale-Birkhoff theorem 
\cite[e.g.]{arrowsmith_place} implies chaotic dynamics.)

Next, define a `pseudo-separatrix' at a chosen time $ t $ as follows: join the unstable manifold ($ \Gamma_\epsilon^u(\vec{a}_\epsilon,t) $, red) emanating from $ \vec{a}_\epsilon(t) $ up to a transverse intersection point, $ \vec{q} $,  to the stable manifold ($ \Gamma_\epsilon^s(\vec{b}_\epsilon,t) $, blue) emanating
from $ \vec{b}_\epsilon(t) $ (shown by the heavy curve in Fig.~\ref{fig:2dlobes}).  The intersection point $ \vec{q} $ corresponds to a heteroclinic point, because it lies on both manifolds.  What this means
is that if choosing an `initial' condition at this point at time $ t $, and defining $ \vec{x}(\tau) $ as being the trajectory going along this point which therefore obeys $ \vec{x}(t) = \vec{q} $, then
$ \left| \vec{x}(\tau) - \vec{a}_\epsilon(\tau) \right| $ decays to zero as $ \tau \rightarrow - \infty $, and moreover $ \left| \vec{x}(\tau) - \vec{b}_\epsilon(\tau) \right| $ decays to zero as $ \tau \rightarrow \infty $.  Now, one considers a Poincar\'{e} map $ P $ of time $ T  = 2 \pi / \omega $ (the period of the perturbation) on the phase space at the fixed time $ t $.  Given the time-periodicity of the flow, the phase space curves will be exactly the same at times $ t + n T $, for $ n \in {\mathbb Z} $.
Every intersection point in Fig.~\ref{fig:2dlobes}) must map to another under this map, since being on both the stable and the unstable manifolds is an invariant property with respect to the map.  Thus, the points $ P^{-1}(\vec{q}) $ and $ P(\vec{q}) $ are themselves intersection points in Fig.~\ref{fig:2dlobes}, and the region between $ P^{-1}(\vec{q} $ and $ \vec{q} $ (and similarly between $ \vec{q} $ and $ P(\vec{q}) $ must possess an
intersection pattern topologically equivalent to that of a sinusoidal curve intersecting the horizontal axis over one period.  The important region consists of the parts of the stable and unstable manifolds lying between $ P^{-1}(\vec{q}) $ and $ \vec{q} $, and also between $ \vec{q} $ and $ P(\vec{q}) $. With respect to the pseudo-separatrix, it is only the lobes in these regions (i.e., $ A $ and $ B $ in Fig.~\ref{fig:2dlobes}) which are involved in crossing the pseudo-separatrix under the action of $ P $ \cite{romkedar,wigi_book}
(or $ P^{-1} $). Consequently, determining the areas of these lobes gives a nice assessment of transport.

When is the lobe area a well-defined measure of transport?  There are two important assumptions to make
this work.  The first is that the perturbation $ \vec{g} $ is {\em time-harmonic}, which has already
been discussed. This assumption is needed
to ensure that the intersection pattern between $ P^{-1}(\vec{q}) $ and $ \vec{q} $ is topologically that of 
a sinusoidal function over one period, and also that the lobes $ A $ and $ B $ possess the same area
to leading-order (this can be shown using the pleasing connection between the Melnikov function and
lobe areas \cite{romkedar,wigi_book}).   The second assumption is that at least the unperturbed flow is area-preserving.  This ensures
that when a lobe gets mapped, it maps into a lobe of equal area to leading-order.  (If the perturbation
is also area-preserving, then the lobes must have equal areas, and not just to leading-order.) Consequently, {\em the} lobe area (or its leading-order expression) can be used as a well-defined measure of the transport that occurs, because all four of the turnstile lobes must
have the same area (at least to leading-order).    It is somewhat less well-known that if the time-periodicity is more general (i.e., not harmonic, but still periodic), a variety of possibilities exist: there may be many, differently-sized lobes between 
$ P^{-1}(\vec{q}) $ and $ \vec{q} $, or there may be {\em no} intersection points $ \vec{q} $ \cite{periodic}.
Understanding and quantifying transport via a Poincar\'{e} map therefore requires some subtlety \cite{periodic}.  Thus, even under general time-periodicity, using a lobe area to characterize transport due to the broken heteroclinic becomes more ambiguous.  Lobe areas may also not be equal if area-preservation is not
imposed, and thus once again using a lobe area to quantify transport is ill-defined.  For general time-{\em aperiodic} $ \vec{g} $, and general non-area-preserving flows, transport is better quantified in terms of an instantaneous flux \cite{aperiodic,rossbyflux,siam_book}, which we will describe and adopt to our current three-dimensional system in Section~\ref{sec:flux}.

To now describe our three-dimensional theory of lobe dynamics, we first impose the condition of a time-harmonic perturbation
\begin{equation}
\vec{g}(\vec{x},t)] = \tilde{g} \left( \vec{x} \right) \cos \left[ \omega t + \phi \right] \, , 
\label{eq:harmonic}
\end{equation}
where $ \tilde{\vec{g}} $ and $ D \tilde{\vec{g}} $ is bounded (as per hypotheses), and the frequency $ \omega \ne 0 $ and the
phase shift $ \phi $ are constant.  The period of the perturbation is then $ T = 2 \pi / \omega $.  

We first show that the Melnikov function takes a remarkably simple form---itself harmonic.  To express this, we make the following choice for the definition of the Fourier transform:
\[
{\mathcal F} \left\{ H(\tau) \right\}(\omega) := \int_{-\infty}^\infty e^{i \omega \tau} H(\tau) \, \d \tau \, ,
\]
for functions of time in $ \mathrm{L}^1(\mathbb{R}) $.  

\begin{theorem}[Melnikov function for harmonic perturbations]
If $ \vec{g} $ is harmonic as given in (\ref{eq:harmonic}), the Melnikov function is itself
harmonic, and expressible as
\begin{equation}
M(p,\alpha,t) = \left| {\mathcal F}\left\{ h(p,\alpha,\centerdot) \right\}(\omega) \right| \, \cos \left[ \omega \left( t -p\right) + \phi + \mathrm{arg} \left(
{\mathcal F} \left\{ h (p,\alpha,\centerdot) \right\} (\omega) \right) \right] 
\label{eq:melnikovharmonic}
\end{equation}
where
\begin{equation}
h(p,\alpha,\tau) := \exp\left[{\int_{\tau}^p \nabla \cdot \vec{f}(\bar{\vec{x}}(\xi, \alpha)) d\xi}\right] \left[ \vec{f}(\bar{\vec{x}}(\tau, \alpha)) \wedge \bar{\vec{x}}_{\alpha}(\tau, \alpha)\right] \cdot \tilde{\vec{g}}(\bar{\vec{x}}(\tau, \alpha)) \, .
\label{eq:fourierfunction}
\end{equation}
\end{theorem}

\begin{proof}
Under the harmonic assumption, the Melnikov
function (\ref{eq:melnikov}) becomes
\begin{align*}
    M(p,\alpha,t) & = \int^{\infty}_{-\infty} \exp\left[{\int_{\tau}^p \nabla \cdot \vec{f}(\bar{\vec{x}}(\xi, \alpha)) d\xi}\right] \left[ \vec{f}(\bar{\vec{x}}(\tau, \alpha)) \wedge \bar{\vec{x}}_{\alpha}(\tau, \alpha)\right] \cdot \tilde{\vec{g}}(\bar{\vec{x}}(\tau, \alpha)) \cos \left[ \omega \left( \tau + t -p\right) + \phi \right] ~ \d \tau \, .
\end{align*}
By treating the other parameters $ (p,\alpha,t) $ as constants when taking the Fourier transform with
respect to $ \tau $, 
the result follows from standard trigonometric manipulations (see \cite{periodic}). 
\end{proof}

\begin{remark}
{\em The `size' of the splitting is therefore encoded in the amplitude of the Melnikov function,
which is simply the modulus of a Fourier transform, i.e., $ \left| {\mathcal F}\left\{ h(p,\alpha,\centerdot) \right\}(\omega) \right| $. We will show that this is the primary measure
for transport. The remaining cosine term captures the harmonic
fluctuations in time $ t $ and location $ p $, occurring with frequency $ \omega $.  }
\end{remark}

Under the generic assumption that
$ {\mathcal F}\left\{ h \right\}(\omega) $ is not identically zero, it is clear that there are zeros of
$ M $ in (\ref{eq:melnikovharmonic}) for each fixed $ t $,  along the $ (p,\alpha) $ curves 
implicitly defined by 
\[
\omega \left( t -p\right) + \phi + \mathrm{arg} \left(
{\mathcal F} \left\{ h (p,\alpha,\centerdot) \right\} (\omega) \right) =  \frac{\pi( 2 k + 1) }{2} \quad ;
\quad k \in \mathbb{Z} \, .
\]

\begin{remark}[Fourier transform under volume-preservation]
{\em If the unperturbed flow is volume-preserving, i.e., if $ \vec{\nabla} \cdot \vec{f} = 0 $, then
$ h $ in (\ref{eq:fourierfunction}) loses its $ p $-dependence, and is
\begin{equation}
h(\alpha,\tau) :=  \left[ \vec{f}(\bar{\vec{x}}(\tau, \alpha)) \wedge \bar{\vec{x}}_{\alpha}(\tau, \alpha)\right] \cdot \tilde{\vec{g}}(\bar{\vec{x}}(\tau, \alpha)) \, .
\label{eq:fourierfunction_volumepreserving}
\end{equation}
Consequently, the amplitude of the Melnikov function (\ref{eq:melnikovharmonic})
is $ p $-independent, and $ t $ and $ p $ only occur in the combination $ t - p $.
}
\end{remark}

We now proceed to quantify stable/unstable manifold intersections, as well as lobe dynamics, under
the hypothesis that $ \vec{\nabla} \cdot \vec{f} = 0 $.  Note however that we are not imposing
volume-preservation on the perturbation $ \vec{g} $.  The implication of (\ref{eq:fourierfunction_volumepreserving}) is that the $ (p,\alpha) $ curves along which zeros occur can be solved explicitly for $ p $ as 
\begin{equation}
\tilde{p}(\alpha,k) = \frac{1}{\omega} \left[ \mathrm{arg} \left(
{\mathcal F} \left\{ h (\alpha,\centerdot) \right\} (\omega) \right) - \frac{\pi( 2 k + 1) }{2} + \omega t + \phi \right] \quad ; \quad (\alpha,k) \in \mathrm{S}^1 \times \mathbb{Z} \, .
\label{eq:heteroclinicring}
\end{equation}
For any chosen value of $ k \in \mathbb{Z} $ (say $ k = 0 $), $ p = \tilde{p}(\alpha,k) $ defines a curve of heteroclinic points which goes across all the $ \alpha $ values in $ \mathrm{S}^1 $.  More precisely, as stated in Remark~\ref{remark:heterocliniccurves}, there is a $ {\mathcal O}(\epsilon) $-close curve on $ \Gamma $ along whose normal (to $ \Gamma $) there is an intersection of the perturbed manifolds.  Geometrically, this means that there is a ring of heteroclinic points, $ Q $, going all the way around the unperturbed manifold. The $ p $-parametrization associated with each point on the ring is given by (\ref{eq:heteroclinicring}) with $ k = 0 $.  Note that $ p $ will be varying as a function of $ \alpha \in \mathrm{S}^1 $ in general. There are similar rings of heteroclinic points for every other  $ k $ value which are simple shifts as given in (\ref{eq:heteroclinicring}).  There being infinitely many heteroclinic rings is the analogous situation to there being infinitely many heteroclinic points in the two-dimensional case. We can visualize this most easily in the $ (\alpha,p) $ parameter space, as shown in Fig.~\ref{fig:lobesparameter}.
The curves represent (\ref{eq:heteroclinicring}) at various $ k $-values, and the regions in between will have
$ M $ being positive and negative in alternating fashion.  For example, in the region between the curves
$ p = \tilde{p}(\alpha,-1) $ and $ p = \tilde{p}(\alpha,0) $, $ M $ is strictly positive, which implies that the unstable 
manifold is {\em outside} the stable manifold in this region.  However, these coincide $ {\mathcal O}(\epsilon )$-close to the boundary curves, and thus correspond to a {\em lobe} which is bounded by the perturbed stable
and unstable manifolds.  More precisely, a $ {\mathcal O}(\epsilon)$ characterization of the lobe flanked by
$p = \tilde{p}(\alpha,k-1) $ and $ p = \tilde{p}(\alpha,k) $ is
\[
L_k = \left\{ \bar{\vec{x}}(p,\alpha) + 
\frac{\epsilon \left[ s M^u(p,\alpha,t) + (1-s) M^s(p,\alpha,t) \right] \hat{\vec{n}}(p,\alpha)}{{\left| \vec{f}(\bar{\vec{x}}(p, \alpha)) \wedge \bar{\vec{x}}_\alpha(p, \alpha) \right| }}  \, : \; \alpha \in \mathrm{S}^1 \, , \, 
\tilde{p}(\alpha,k-1) \le p \le \tilde{p}(\alpha,k) \, , \, s \in [0,1] \right\} \, , 
\]
expressed in terms of the unit normal vector
\begin{equation}
\hat{\vec{n}}(p,\alpha) := 
\frac{\vec{f}(\bar{\vec{x}}(p, \alpha)) \wedge \bar{\vec{x}}_\alpha(p, \alpha) }{\left| \vec{f}(\bar{\vec{x}}(p, \alpha)) \wedge \bar{\vec{x}}_\alpha(p, \alpha) \right|}
\label{eq:normal}
\end{equation}
at locations $ \bar{\vec{x}}(p,\alpha,t) $ on $ \Gamma $, and the unstable and stable Melnikov
functions $ M^u $ and $ M^s $.  This is set up so that when $ s = 0 $, we
are at the $ {\mathcal O}(\epsilon) $-representation of the point $ \vec{x}^s(p,\alpha,\epsilon,t) $, and when $ s = 1 $ we are at
$ \vec{x}^u(p,\alpha,\epsilon,t) $ on the unstable manifold. All such lobes, determined by taking adjacent values of $ k $ in (\ref{eq:heteroclinicring}), are topologically equivalent to solid rings (torii along with their interior) given that they wrap around $ \alpha \in \mathrm{S}^1 $.  Those indexed by even $ k $ have the unstable
manifold being outside the stable manifold (with the corresponding Melnikov function $ M > 0 $ in
the interior of the subtending $ (p,\alpha)$ domain), and those with odd $ k $ have the opposite. 

If instead we go along a constant $ \alpha $-line in Fig.~\ref{fig:lobesparameter}, $ M $ will periodically vary from positive to negative (and vice versa) as each $ p = \tilde{p}(\alpha,k) $ curve is crossed.   This means that if going along an unperturbed heteroclinic trajectory $ \bar{\vec{x}}(p,\alpha) $
with fixed $ \alpha $ on
$ \Gamma $, the perturbed stable and unstable manifold will periodically be intersecting (nearby the locations
$ \bar{\vec{x}}\left(p(\alpha,k),\alpha \right) $), with the stable and unstable manifolds interchanging relative
locations at each crossing of the curves $ p = \tilde{p}(\alpha,k) $.  So this picture would be topologically equivalent to the two-dimensional picture of Fig.~\ref{fig:2dlobes}, with infinitely many intersections
occurring in approaching the points $ \vec{a}_\epsilon(t) $ and $ \vec{b}_\epsilon(t) $.  Hence, the topological three-dimensional lobe structure can be visualized by `rotating' Fig.~\ref{fig:2dlobes}, i.e., taking this $ \alpha = $ constant situation and rotating around the line connecting $ \vec{a}_\epsilon(t) $ 
and $ \vec{b}_\epsilon(t) $ to generate the full picture associated with $ \alpha $ ranging in $ \mathrm{S}^1 $. So $ q $ rotates to become a a topological circle $ Q $, and each lobe area in
Fig.~\ref{fig:2dlobes} rotates to generate three-dimensional solid-ring lobes.

\begin{figure}[t]
	\centering
	\includegraphics[scale=0.5]{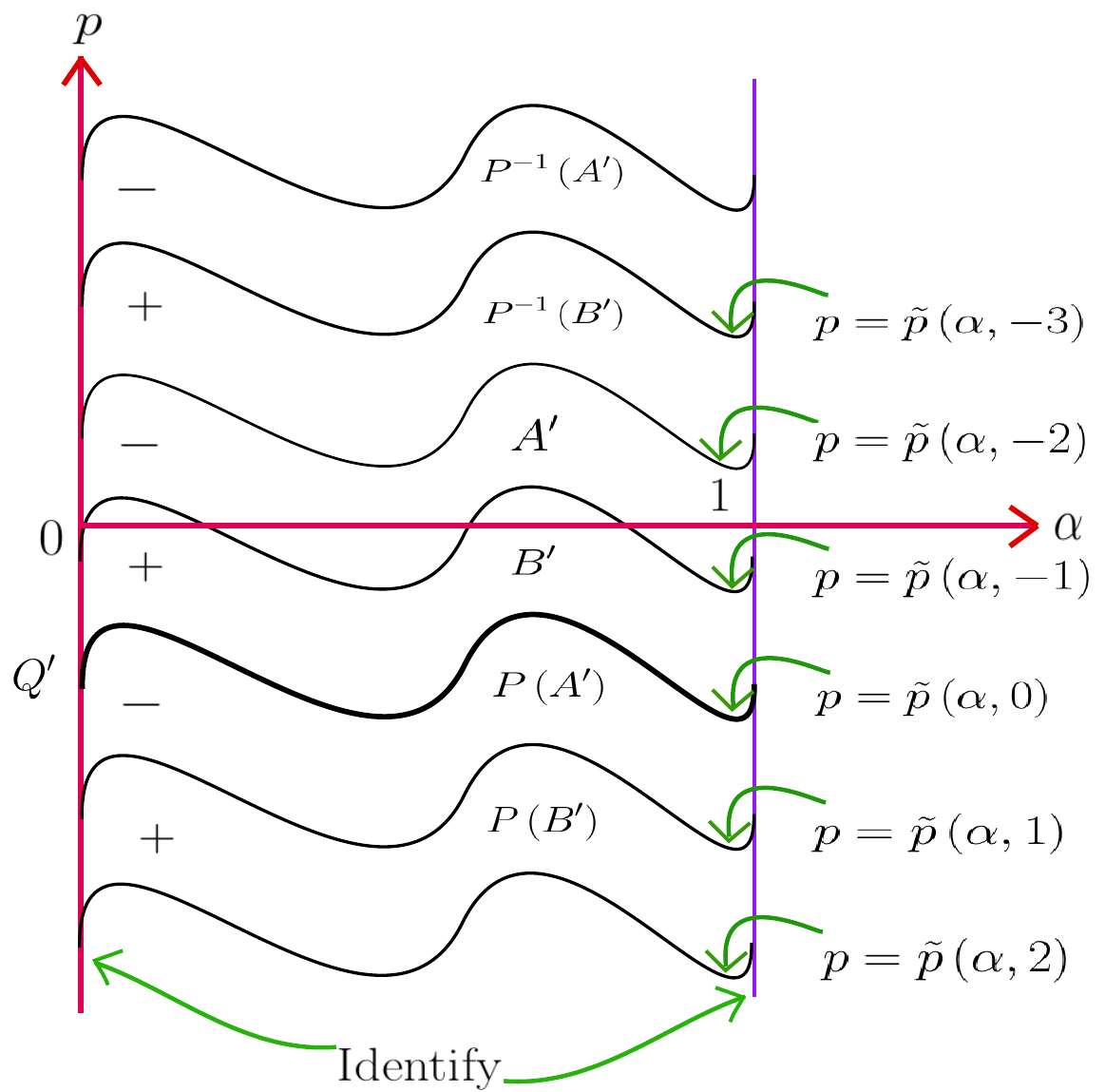}
	\caption{The behavior of the Melnikov function $ M(p,\alpha,t) $ (shown with positive and negative
	signs) at a fixed time-instance $ t $, with
	respect to $ (p,\alpha) $-space, under time-harmonic volume-preserving conditions.  The (infinitely many) curves correspond to $ p = \tilde{p}(\alpha,k) $ as defined
	in (\ref{eq:heteroclinicring}), and are simple shifts of one another. The labels with primes represent
	the regions in parameter space corresponding to three-dimensional lobe dynamics, which turns out to
	be a rotated (in $ \alpha$) version of Fig.~\ref{fig:2dlobes} as explained in the main text. }
	\label{fig:lobesparameter}
\end{figure}

Lobe dynamics can be now be applied in a similar spirit to the two-dimensional situation \cite{romkedar,wigi_book}.  Consider the two-dimensional unstable manifold $\Gamma_{\epsilon}^u\left(\vec{a}_\epsilon, t\right)$ emanating from $ \vec{a}_\epsilon(t) $ until it meets
the heteroclinic ring $ Q $.  Similarly, take the part of the two-dimensional stable manifold $\Gamma_{\epsilon}^s\left(\vec{b}_\epsilon, t\right)$ emanating from $ \vec{b}_\epsilon(t) $ until it reaches $ Q $.  Attaching these two restricted manifold segments together gives a pseudo-separatrix which demarcates the `inside' and the `outside' of what is topologically a sphere.  Now, under mappings of the Poincar\'{e} map $ P $ which has period $ T = 2 \pi / \omega $, the lobe volumes get mapped into one another, because the ring of heteroclinic points $ Q $ must itself map into similar rings.  
We have
in Fig.~\ref{fig:lobesparameter} used a `prime' notation to label regions in the parameter space 
corresponding to the structures in three-dimensional $ \vec{x} $-space.  For example $ Q' $ is the
manifestation in the parameter space of the ring of heterclinic points $ Q $, where we have chosen to take
the curve $ p = \tilde{p}(\alpha,0) $ as the relevant transverse intersection to help define the
pseudo-separatrix.  The turnstile lobes $ A $, $ B $, $ P(A) $ and $ P(B) $ are therefore associated
with the regions $ A' $, $ B' $, $ P(A') $ and $ P(B') $ in Fig.~\ref{fig:lobesparameter}, within which
the Melnikov function is respectively negative, positive, negative and positive. The lobe dynamics
then are as in a rotated version of Fig.~\ref{fig:2dlobes}, but with the same labelling.

As for the two-dimensional case, it is only the two ring-lobes between $ P^{-1}(Q) $ and $ Q $ which cross the pseudo-separatrix,
mapping to the lobe structures between $ Q $ and $ P(Q) $.  Consequently the volumes of these lobes can characterize the transport across the (now broken) heteroclinic.  

\begin{theorem}[Lobe volume for time-harmonic volume-preserving flow]
\label{theorem:lobevolume_harmonic}
Suppose $ \vec{g} $ is harmonic, and the unperturbed flow is volume-preserving (i.e., $ \vec{\nabla} 
\cdot \vec{f} = 0 $).  Then, each and every one of the lobes $ L_k $, for $ k \in \mathbb{Z} $ has the
volume
\begin{equation}
\mathrm{Volume} \left( L_k \right)
= \frac{ 2 \epsilon}{\omega} \int_0^1 \left| {\mathcal F}\left\{ h(\alpha,\centerdot) \right\}(\omega) \right| \, \d \alpha + {\mathcal O}(\epsilon^2) \, . 
\label{eq:lobevolume_harmonic}
\end{equation}
\end{theorem}

\begin{proof}
See Appendix~\ref{sec:lobevolume_harmonic}. 
\end{proof}

The $ k $-independence in the leading-order term in (\ref{eq:lobevolume_harmonic}) ensures that the volumes of {\em every} one of these lobes is equal to leading-order under our assumptions.  Note that this is true even if the perturbation $ \tilde{\vec{g}} $ were not volume preserving (then, volume equality of the ring-lobes only gets compromised at the next order
in $ \epsilon $).  Given the independence of the leading-order lobe volumes to the choice of the lobe,
we can therefore use (\ref{eq:lobevolume_harmonic}) to express the transport across the broken heteroclinic 
in the sense of the volume of fluid that is interchanged across the pseudo-separatrix due to the turnstile
lobes upon one iteration of the Poincar\'{e} map.  

\begin{remark}
{\em
As pointed out by Rom-Kedar and Poje \cite{romkedarpoje}, given the dependence on the time period of the Poincar\'e map, the lobe volume by itself may not be a good quantifier of the flux in the sense of fluid volume exchanged per unit time.  This is better obtained by dividing the volume by the associated time-of-flow $ T = 2 \pi / \omega $ of the Poincar\'e map.  This yields
\begin{equation}
    \mathrm{Flux~(harmonic)} = \frac{\epsilon}{\pi} \int_0^1 \left| {\mathcal F}\left\{ h(\alpha,\centerdot) \right\}(\omega) \right| \, \d \alpha + {\mathcal O}(\epsilon^2) \, .
    \label{eq:flux_lobes_harmonic}
\end{equation}
Interesting implications related to the flux as a function of the frequency $ \omega $ can now be made,
just as for the two-dimensional case.  Generically, the expectation is for the flux to increase at small 
$ \omega $ and eventually decay as $ \omega \rightarrow \infty $, implying the presence of a
flux optimizing frequency \cite{romkedarpoje}.  For a given unperturbed
flow for which the heteroclinic structure is known, the Fourier transform formula in (\ref{eq:flux_harmonic})
can therefore be employed easily, for example, to find the frequency resulting in the greatest flux (as has been done in two dimensions \cite{frequency}).  Alternatively, the formula can be analyzed with different
spatial perturbations $ \tilde{\vec{g}} $ at fixed $ \omega $, attempting to find transport-optimizing
perturbations (see \cite{optimal,l2mixer} for two-dimensional implementations). 
}
\end{remark}

\subsection{Instantaneous flux for the general situation}
\label{sec:flux}

\begin{figure}[t]
	\includegraphics[scale=0.38]{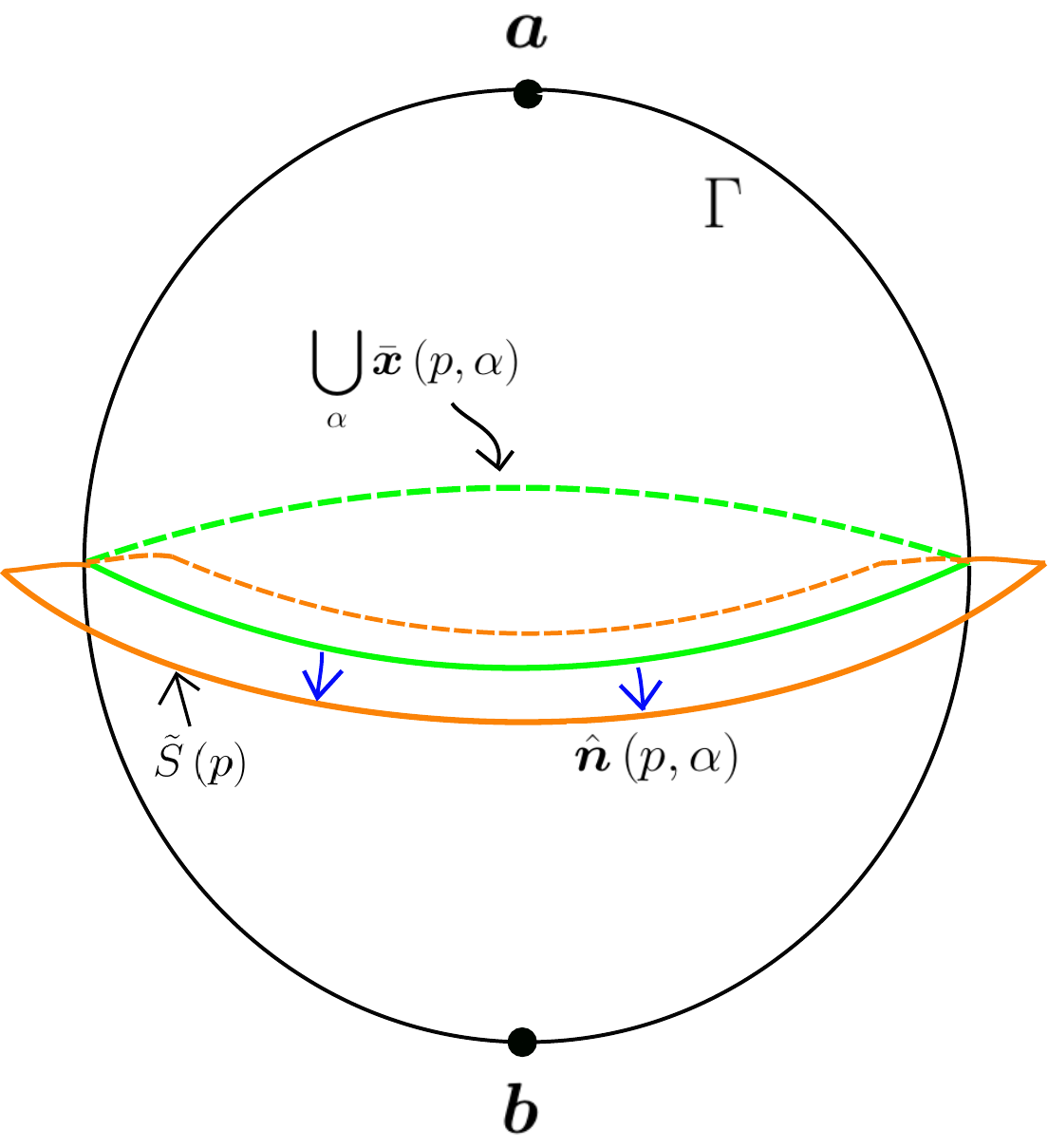}
	\hspace{0.25in}
	\includegraphics[width=67mm,height=74mm,scale=0.5]{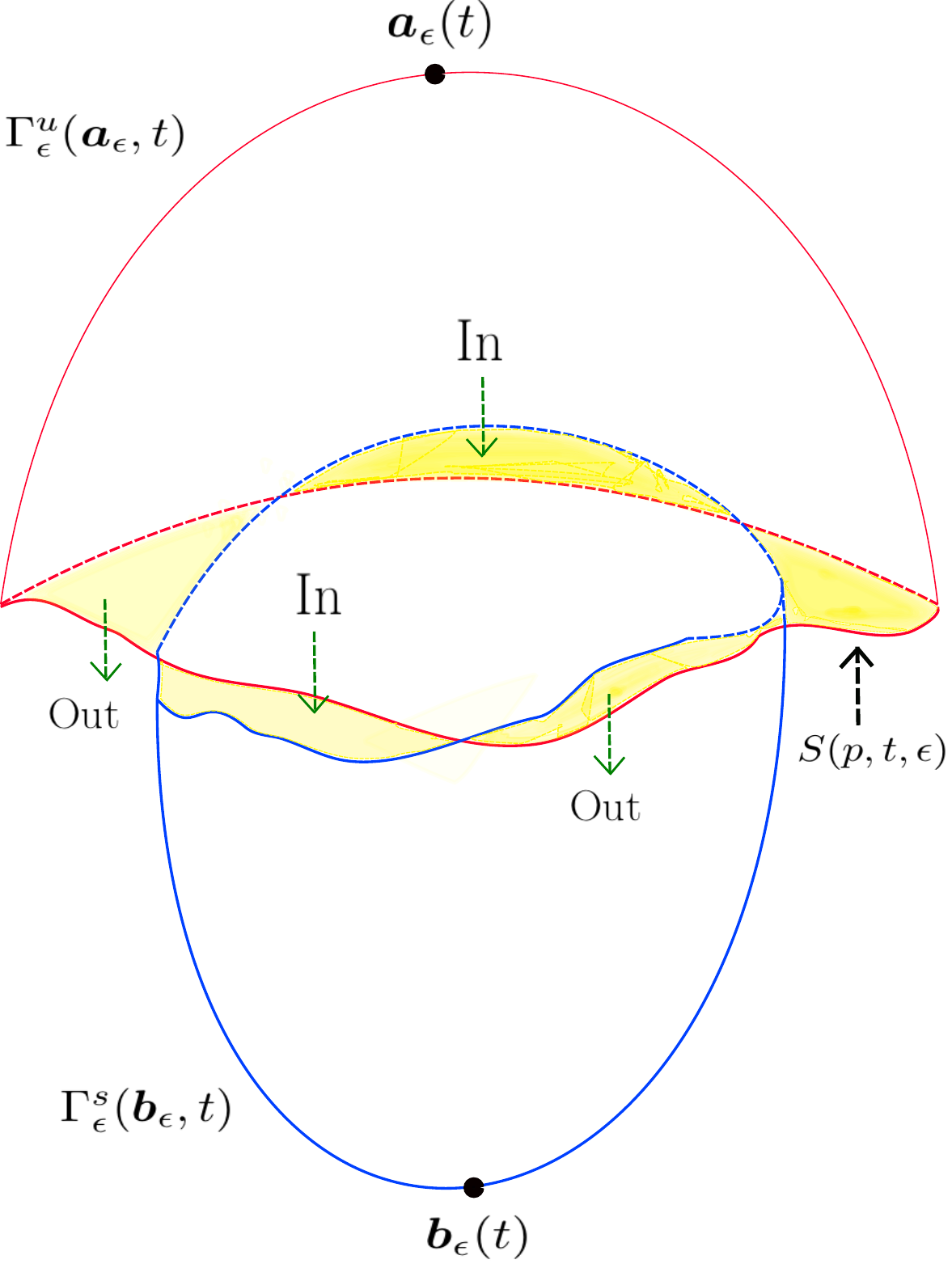}
	\caption{The construction of (a) the strip, and (b) the pseudo-separatrix at time $ t $, associated with making a choice of $ p $, and consisting of the unstable manifold (red) and stable manfold (blue), both extending until intersecting the strip $ S(p,\epsilon,t) $, as well as the strip
	$ S (p,\epsilon,t) $.  
	}
	\label{fig:pseudoseparatrix}
\end{figure}

Obtaining lobe volumes and flux specifically as a measure of transport across the broken separatrix was done in the previous section under several assumptions: notably, the harmonicity of the perturbation, and the volume-preserving nature of the unperturbed flow.  Even when the harmonicity is relaxed to general periodicity, the expressions cannot be used to quantify transport between there will in general be {\em differently-sized} lobes, exactly as in the two-dimensional situation \cite{periodic}.  Thus, in this section we will describe a more general way of quantifying the transport, analogous to the two-dimensional development of instantaneous flux \cite{aperiodic,rossbyflux}.  Most tellingly, we now relax both the assumptions on time-harmonicity and volume-preservation, allowing $ \vec{g}(\vec{x},t) $ to have any {\em general}
temporal behavior (but subject to the smoothness/boundedness assumptions detailed in Section~\ref{sec:manifolds}), and both $ \vec{f} $ and $ \vec{g} $ to not be divergence-free. 

An `obvious' way of quantifying a time-dependent flux across the previously impermeable $ \Gamma $
would be the Eulerian approach
\begin{align}
    \mathrm{Eulerian~flux~}(t,\epsilon) &= \iint_\Gamma \left[ \vec{f}(\bar{\vec{x}}(p,\alpha) + \epsilon \vec{g}\left( \bar{\vec{x}}(p,\alpha),t \right) \right] \cdot \hat{\vec{n}}(p,\alpha) \, \d S \nonumber \\
    &= \epsilon \int_0^1 \int_{-\infty}^\infty \vec{g} \left(\bar{\vec{x}}(p,\alpha),t \right) \cdot \left[ \vec{f}(\bar{\vec{x}}(p,\alpha) \wedge \bar{\vec{x}}_\alpha(p,\alpha) \right] \, \d p \, \d \alpha \, .
    \label{eq:eulerianflux}
\end{align}
The term `Eulerian' is used here in the fluid-mechanics context, in relation to the flux across a fixed 
surface.  The simplification above occurs because $ \vec{f} $ is tangential to $ \Gamma $, and thus does not contribute.

However, the Eulerian approach does not take into account {\em transport}.  The issue is that the manifolds
themselves move, as do trajectories on and adjacent to them.  Therefore a velocity flux across $ \Gamma $
does not capture the Lagrangian (following the flow of particle trajectories) transport.  As a simple example, imagine that the perturbation is such that the manifold $ \Gamma $ retains its structure as continuing to be heteroclinic, but simply `puffs out' to be slightly larger than the original $ \Gamma $.  Then, there should be no predicted transport, because the perturbed manifold structure persists.  On
the other hand, there could be a nonzero Eulerian flux because $ \vec{g} $ need not be zero on $ \Gamma $, nor perpendicular to $ \Gamma $.  It need only satisfy those conditions on the perturbed version of 
$ \Gamma $.  Computing the flux across the fixed unperturbed $ \Gamma $ without taking into account the
fact that the invariant manifolds have now perturbed, is incorrect.  In other words, in a transport computation the fact that the manifold is itself moving must be taken into account, and not simply the time-variation of the vector field.

A {\em Lagrangian} approach---which takes into account the perturbation on the
previously impermeable heteroclinic manifold---is thus necessary.
On the unperturbed $ \Gamma $, consider a fixed-$ p $ curve, that is $ \bigcup_{\alpha \in \mathrm{S}^1} \bar{\vec{x}}(p,\alpha) $.  This goes all the way around $ \Gamma $, and at each point on it has an outward-pointing  normal
vector $ \hat{\vec{n}}(p,\alpha) $ as given in (\ref{eq:normal}).  We form the unperturbed strip $ \tilde{S}(p) $ by taking the union of these normal vectors across $ \alpha \in \mathrm{S}^1 $, and extending outwards in both the positive and negative directions of $ \hat{\vec{n}}(p,\alpha) $.  At this instance, we will not specify how far the extension needs to be done; this will be clear shortly.  Thus, $ \tilde{S}(p) $ is 
an oriented strip going all the way around $ \Gamma $; a portion of this is displayed in Fig.~\ref{fig:pseudoseparatrix}(a).

Now consider the perturbed stable 
and unstable manifolds at a general time $ t $.  Using the results from Section~\ref{sec:manifolds}, we know that the distance to the perturbed unstable manifold along the normal vector direction is given by
\[
d^u(p,\alpha,\epsilon,t) = \hat{\vec{n}}(p,\alpha) \cdot \left[ \vec{x}^u(p,\alpha,\epsilon,t) -
\bar{\vec{x}}(p,\alpha) \right] \, 
\]
and to the perturbed stable manifold
\[
d^s(p,\alpha,\epsilon,t) = \hat{\vec{n}}(p,\alpha) \cdot \left[ \vec{x}^s(p,\alpha,\epsilon,t) -
\bar{\vec{x}}(p,\alpha) \right] \, 
\]
in terms of trajectories $ \vec{x}^u $ and $ \vec{x}^s $ respectively on the perturbed manifolds. We
now define the strip $ \S(p,\epsilon, t) \subset  \tilde{S}(p) $ by
\begin{equation}
    S(p,\epsilon,t) := \bigcup_{\alpha \in \mathrm{S}^1} \bigcup_{s \in [0,1]} \left\{ \bar{\vec{x}}(p,\alpha) + 
    \hat{\vec{n}}(p,\alpha) \left[ s d^u(p,\alpha,\epsilon,t) + (1-s) d^s(p,\alpha,\epsilon,t) \right] \right\} \, .
    \label{eq:strip}
\end{equation}
Thus, at each fixed time $ t $ and parameter $ p $, the strip $ S(p,\epsilon,t) $ is a $ {\mathcal O}(\epsilon)$-width `ribbon' which is attached to the closed curve $\bigcup_{\alpha \in \mathrm{S}^1} \bar{\vec{x}}(p,\alpha) $ on $ \Gamma $.  This is illustrated by the shaded segment in  Fig.~\ref{fig:pseudoseparatrix}(b).  At each location $ \alpha $, the strip traverses the normal direction to
$ \Gamma $ spanning from the stable to the unstable manifold.   When the
stable and unstable manifolds intersect, then the `ribbon' has zero width at that point. Therefore, the unperturbed strip $ \tilde{S}(p) $ needs to extend out to ensure that $ S(p,\epsilon,t) \subset \tilde{S}(p) $ for all times $ t $ and perturbative parameters $ \epsilon $ for which we want to characterize the flux. 

We want to define an instantaneous flux across the broken heteroclinic manifold, taking into account the
Lagrangian nature of the problem.  To do so, we now define the pseudo-separatrix  $ \tilde{\Gamma}(p,t,\epsilon) $ as
the union of the following two-dimensional surfaces:
\begin{itemize}
    \item The part of the unstable manifold $ \Gamma^u_\epsilon(\vec{a}_\epsilon,t) $ emanating from
    $ \vec{a}_\epsilon(t)$, until it first hits the strip $ S(p,\epsilon,t) $;
    \item The strip $ S(p,\epsilon,t) $; and 
    \item The part of the stable manifold $ \Gamma^s_\epsilon(\vec{b}_\epsilon,t) $ emanating from
    $ \vec{b}_\epsilon(t)$, until it first hits the strip $ S(p,\epsilon,t) $.
\end{itemize}
These entities are shown in Fig.~\ref{fig:pseudoseparatrix}(b).  We note that $ \tilde{\Gamma}(p,\epsilon,t) $ is $ {\mathcal O}(\epsilon)$-close the the unperturbed $ \Gamma $, which was a genuine flow separator between the its `inside' and `outside.'  The pseudo-separatrix $ \tilde{\Gamma} $ is itself a closed surface for any chosen $ p $, and at any time $ t $.  It is one way of attempting to define a semi-separator between the inside and the outside of the perturbed version of $ \Gamma $, which is in reality the combination of the perturbed stable and unstable manifolds which typically will intersect with each other.  Note that the definition takes into account
the perturbed versions of both the stable and unstable manifolds, and is hence Lagrangian. We can 
determine the instantaneous signed flux (volume per unit time) exiting $ \tilde{\Gamma} $ at time $ t $, subject to the
choice of the parameter $ p $.  Key to this quantification is the observation that the stable/unstable
manifold parts of $ \tilde{\Gamma}(p,\epsilon,t) $ are moving as invariant objects, and hence can have no flux crossing them.  Transport across $ \tilde{\Gamma}(p,\epsilon,t) $ occurs because of the flux across the strip where the stable and unstable manifolds connect.  The instantaneous flux $ \Phi $ can be
quantified elegantly in terms of the Melnikov function:

\begin{theorem}[Instantaneous flux]
\label{theorem:flux}
The instantaneous (signed) flux exiting the pseudo-separatrix $ \tilde{\Gamma}(p,t,\epsilon) $ is
\begin{equation}
    \Phi(p,t,\epsilon) = \epsilon \int_0^1 M(p,\alpha,t) \, \d \alpha + {\mathcal O}(\epsilon^2) \, .
    \label{eq:flux}
\end{equation}
\end{theorem}

\begin{proof}
See Appendix~\ref{sec:flux_proof}.
\end{proof}

We emphasize that this result is for the {\em general} form of the Melnikov function, without having
to make assumptions on time-harmonicity or volume-preservation.  It is even true if the manifolds do not intersect at all!  Theorem~\ref{theorem:flux} is therefore a very general result, 
which states that the leading-order instantaneous flux is given by integrating the Melnikov function over all
$ \alpha $ (representing all the unperturbed heteroclinic trajectories).  The instantaneous fllux
$ \Phi $ is of course time-dependent; this variation is captured by the $ t $ in the Melnikov function.
Moreover, $ \Phi $ is also dependent on the choice of $ p $, the location along which the the perturbed stable and unstable manifolds are joined.

The flux is explicitly the volume of phase space which crosses $ \tilde{\Gamma} $ per unit time.  This is best rationalized in the fluid mechanical context in which $ \vec{f} + \epsilon \vec{g} $ is a fluid velocity.
Then, the flux represents exactly the volume of fluid per unit time exiting $ \tilde{\Gamma} $.

It should also be pointed out that the flux $ \Phi $ is {\em signed}.  In parts where the unstable manifold is outside the stable
manifold on the strip $ S $, there will be flux {\em exiting} the closed surface $ \tilde{\Gamma} $ (see Fig.~\ref{fig:pseudoseparatrix}(b)). These will be encoded as positive, reflecting also the fact that a positive Melnikov function implies that $ d^u - d^s $, when projected on to the outwards-pointing normal vector direction $ \hat{\vec{n}}(p,\alpha) $, is positive.  Similarly, parts where the unstable manifold is inside the stable manifold are associated with flux
flowing into $ \tilde{\Gamma} $, and thus constitutes negative flux.  The expression $ \Phi $ in
Theorem~\ref{theorem:flux} sums all these flux contributions to obtain a signed net flux.  If $ \Phi > 0 $, that means that there is more volume of fluid instantaneously exiting $ \tilde{\Gamma} $ than there is entering it.  

\begin{remark}
{\em The instantaneous flux is dependent on the choice of $ p $; that is, the ring of locations on the heteroclinic trajectories at which the strip $ S $ is drawn.  In view of Remark~\ref{remark:heteroclinicshift}, though, it is clear that if the unperturbed flow is volume-preserving, then the flux $ \Phi $'s $ p $-dependence is equivalently a shift in $ t $.
}
\end{remark}

Having no assumptions on time-harmonicity or volume-preservation, the instantaneous flux interpretation
and expression given in Theorem~\ref{theorem:flux} is an important result.  However, for completeness,
we next discuss the implications of imposing time-harmonicity, and compare with the lobe dynamics approach. 

\begin{corollary}[Instantaneous flux for time-harmonic perturbations]
\label{corollary:flux_harmonic}
If the perturbation $ \vec{g} $ satisfies the time-harmonic assumption (\ref{eq:harmonic}), then
the instantaneous flux of Theorem~\ref{theorem:flux} can be written as
\begin{equation}
    \Phi(p,t,\epsilon) = \epsilon \left| \int_0^1 {\mathcal F} \left\{ h(p,\alpha,\centerdot) \right\} (\omega) \, 
    \d \alpha \right| \cos \left[ \omega (t-p) + \phi + \mathrm{arg} \left( \int_0^1 {\mathcal F} \left\{ h(p,\alpha,\centerdot) \right\} (\omega) \, 
    \d \alpha \right) \right] + {\mathcal O}(\epsilon^2) \, , 
    \label{eq:flux_harmonic}
\end{equation}
where $ h $ is defined in (\ref{eq:fourierfunction}).
\end{corollary}

\begin{proof}
See Appendix~\ref{sec:flux_harmonic_proof}.
\end{proof}

The leading-order term of the flux in this specialized instance is itself harmonic in time $ t $, and
subject to the same frequency $ \omega $ as the perturbation.  This means that as they evolve,  the stable and unstable manifolds periodically interchange their intersection locations along the normal vector direction, such that the net flux flips from being outwards, to being inwards, and outwards again.
The $ p $-dependence (i.e., the dependence on where the strip $ S(p,t,\epsilon) $ is positioned) is significantly more complicated. A time-averaged measure of the
leading-order instantaneous flux could be the amplitude of the harmonic, which we observe is almost the
same as the flux argued via the lobe dynamics approach for time-harmonic volume-preserving instances in (\ref{eq:flux_lobes_harmonic}).  That is, the comparison is between
\begin{equation}
F_L = \frac{1}{\pi} \int_0^1 \left| {\mathcal F} \left\{ h(\alpha,\centerdot) \right\} (\omega) \right| \, \d \alpha \,
\qquad \mathrm{and} \qquad
F_I = \left| \int_0^1 {\mathcal F} \left\{ h(p,\alpha,\centerdot) \right\} (\omega) \, \d \alpha \right| \, ,
\label{eq:flux_harmonic_measures}
\end{equation}
where $ F_L $ is the leading-order flux from the lobe dynamics approach (\ref{eq:flux_lobes_harmonic}),
and $ F_I $ is the amplitude of the leading-order term in the instantaneous flux approach of this section.
We detail the differences below.
\begin{itemize}
    \item There is a minor scaling factor of $ \pi $; this is inconsequential and related to the time-scaling used in converting the lobe volumes to a flux in (\ref{eq:flux_lobes_harmonic}.  If the time-average (rather than the amplitude) of (\ref{eq:flux_harmonic}) were used instead, a slightly
    different scaling factor involving $ \omega $ would result instead.
    \item The instantaneous flux version $ F_I $ is $ p $-dependent (via $ h $), unlike $ F_L $. The reason for this is that Corollory~\ref{corollary:flux_harmonic} does {\em not} require a 
volume-preservation assumption, whereas for the lobe dynamics approach, one needs this to ensure that the ``volume of a lobe'' is unambiguous to leading-order.  Consequently, in the more general framework
of an instantaneous flux as presented here, there will be dependence on the location $ p $ chosen to
define the strip.  The flux based on a different choice of strip will not necessarily be the same,
because a lack of volume-preservation means that the flux across two strips located at $ p = p_1 $
and $ p = p_2 $ are not the same.  If volume-preservation of $ \vec{f} $ was imposed, however, a straightforward application of the divergence theorem for the region bounded by the strips $ S(p_1,t,\epsilon) $ and $ S(p_2,t,\epsilon) $, and the perturbed stable and unstable manifolds
$ \Gamma_\epsilon^u(\vec{a}_\epsilon,t) $ and $ \Gamma_\epsilon^s(\vec{b}_\epsilon,t) $ implies
that the flux across the two strips are the same (to leading-order).  Of course, then the Fourier
transform $ \int_0^1 {\mathcal F} \left\{ h(p,\alpha,\centerdot) \right\}(\omega) \, \d \alpha $
reduces to $ \int_0^1 {\mathcal F} \left\{ h(\alpha,\centerdot) \right\}(\omega) \, \d \alpha $ where
the new $ h $ is $ p $-independent as given in (\ref{eq:fourierfunction_volumepreserving}) as 
opposed to (\ref{eq:fourierfunction}).
    \item The modulus signs are inside the $ \alpha $-integral in $ F_L $, but outside it in $ F_I $.
    The intuition for this is subtle.  In the lobe dynamics approach, a lobe volume is computed by integrating over a region in which the Melnikov function were sign definite,  Thus, the amplitude of the harmonic form of the Melnikov function (\ref{eq:melnikovharmonic}) remains the same sign over all $ \alpha $.  This is not necessarily so for the general instantaneous flux scenario, in which a ring of constant $ p $ is chosen to define the strip which forms the connection between the perturbed stable and unstable manifolds.  When going along this, the stable and unstable manifolds may interchange their locations; consequently, the Melnikov function will in general take on both positive and negative values.  The instantaneous flux takes all this into account, `adding everything up,' and this is accomplished with the modulus being taking {\em after} the net impact is computed.
\end{itemize}
Given these observations, the amplitude of the leading-order flux for the instantaneous flux interpretation, i.e., $ F_I $ in (\ref{eq:flux_harmonic_measures}), is a more general measure of the flux
for time-harmonic perturbations. In particular, it does {\em not} require volume-preservation, unlike 
in the lobe dynamics approach.

\section{Application to Hill's spherical vortex} 
\label{sec:hill}

In this section, our Melnikov theory is applied to Hill's spherical vortex, in particular in quantifying the splitting
of the stable and unstable manifolds after perturbation.  We consider both the classical (no-swirl) and the rotating (swirl) 
versions, which correspond respectively to the situations of purely real, and complex-conjugate eigenvalues.

\subsection{Classical Hill's spherical vortex}
\label{sec:hill_classical}

The classical Hill's spherical vortex is a solution of Euler's equations of motion for an inviscid fluid.  In $(r, \theta, \phi)$
spherical polar coordinates with $r \geq 0$ is the radial distance from the origin, $\theta \in [0, \pi]$ the polar angle
and $\phi \in [0, 2 \pi)$ the azimuthal angle, 
the (continuous) velocity field is given by \cite{hill1894,3d} 
\begin{equation}
\renewcommand{\arraystretch}{1.8}
\vec{f}(r, \theta, \phi) 
=  \left\{ \begin{array}{ll} 
\frac{3 \cos(\theta)}{2} \left(1 - r^2\right) \hat{\vec{r}} - \frac{3 \sin(\theta)}{2} \left(1 - 2 r^2\right) \hat{\vec{\theta}}& ~~~\text{if $r \le 1$}\\
-\frac{\cos(\theta)}{r^3}\left(r^3 - 1\right) \hat{\vec{r}} + \frac{\sin(\theta)}{2r^3}\left(2r^3 + 1\right)  \hat{\vec{\theta}} & ~~~\text{if $r > 1 $}
\end{array} \right. \, .
\label{eq:hill}
\end{equation}
It is easily verifiable that $ \vec{\nabla} \cdot \vec{f} = 0 $ here; the flow is volume-preserving.  The globe $ r = 1 $
is a heteroclinic manifold $ \Gamma $, associated with saddle points located at the north and south poles.
In $(r, \theta)$-coordinates, these points can be written as $ \vec{a} = (1, 0)$ and $
\vec{b} = (1, \pi)$, and the unstable manifold 
of $ \vec{a} $ coincides with the stable manifold of $ \vec{b} $, to form $ \Gamma $.  This manifold is foliated with heteroclinic trajectories which
have a constant $\phi$-value (constant longitude), and thus the trajectory-identifying parameter is $ \alpha = \phi / (2 \pi)  \in \mathrm{S}^1 $.
See Fig.~\ref{fig:hill}.  The splitting of $ \Gamma $, using the volume-preserving requirement
and a functional-analytic viewpoint, has been previously pursued \cite{3d}.  We now apply our more general theory, and
obtain locations of the perturbed manifolds as well as the Melnikov function.

\begin{figure}[t]
\centering
\includegraphics[scale=0.3]{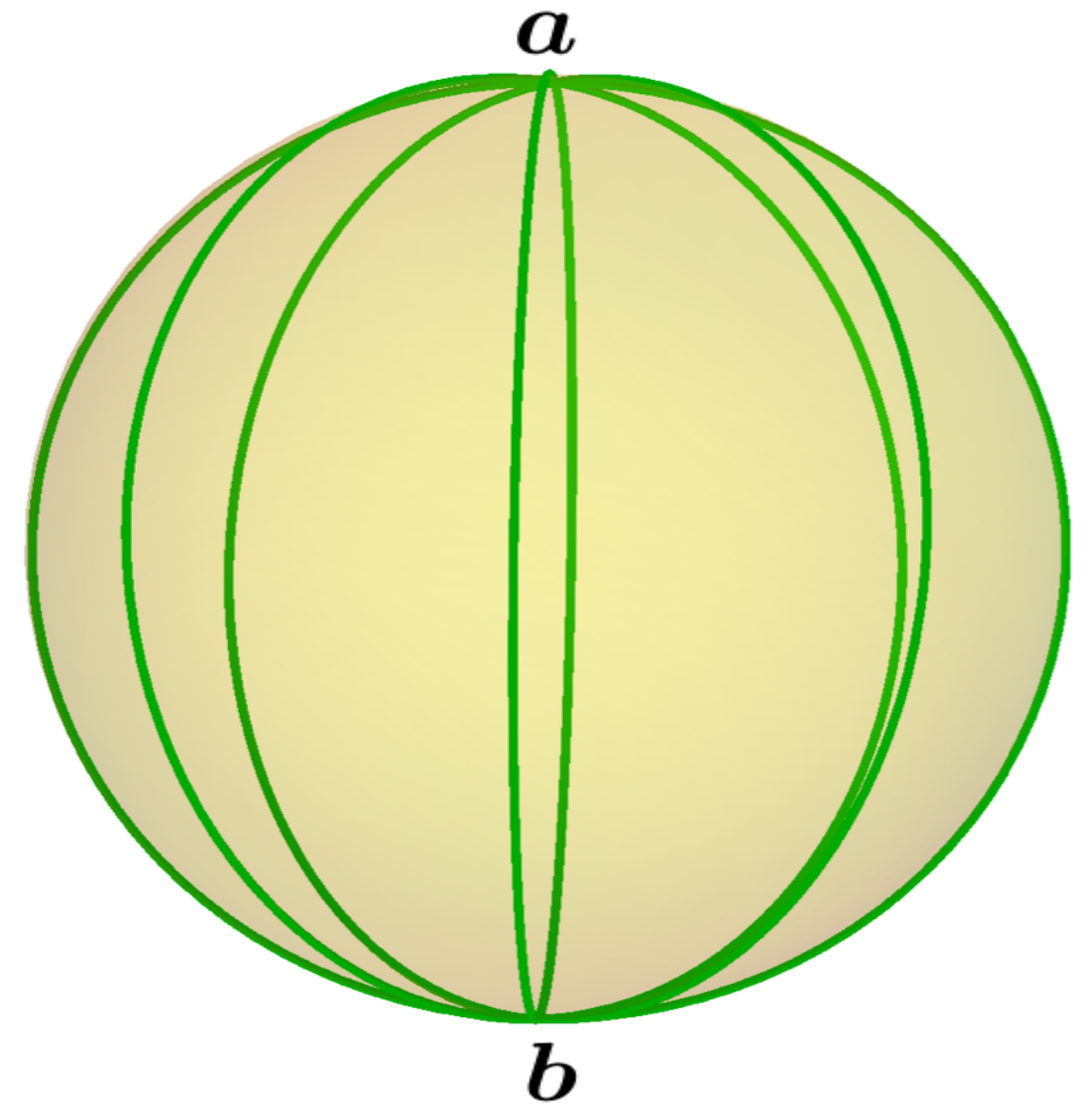}
\caption{The heteroclinic manifold $ \Gamma $ for the standard Hill's spherical vortex, with some $ \alpha = $
constant curves indicated in green.  In this steady flow, these curves are trajectories of associated with the velocity
field (\ref{eq:hill}), which flow along constant longitudes of the sphere from the north pole ($ \vec{a} $) to the south pole ($ \vec{b} $).}  
\label{fig:hill}
\end{figure}

Now, $p$ represents the time-variation along a heteroclinic trajectory $\bar{\vec{x}}(p,\alpha)$, and is thus
functionally related to $ \theta $, the latitude coordinate.  Each heteroclinic trajectory is given in $ (r,\theta,\phi) $-coordinates as
$(1, \bar{\theta}(p), \phi) $, where  $\bar{\theta}$ can be found via the velocity along a longitude:
\begin{equation}
\frac{d \bar{\theta}}{d p} = \frac{3}{2} \sin(\bar{\theta}) \quad ,  \quad {\mathrm{and~so}} \quad 
\bar{\theta}(p) = \cos^{-1} \left( - \tanh \frac{3 p}{2} \right) \, , 
\label{eq:theta}
\end{equation} 
where we have chosen $\bar{\theta}(0) = \pi/2 $, i.e., $ p = 0 $ at the equator for every heteroclinic trajectory,
and this form of inverse trigonometric function gives the principal branch $ \theta \in [0,\pi] $ as required.
The general heteroclinic-trajectory parameterization of
$ \Gamma $ (in $(r, \theta, \phi)$ form) is therefore given by
 \[
\bar{\vec{x}}(p, \alpha) = \left( 1, \cos^{-1} \left( - \tanh \frac{3 p}{2} \right),  2 \pi \alpha \right) \, , 
\]
in $ (r,\theta,\phi) $ components.  Given that $ \sin(\bar{\theta}(p)) = \sech \left(\frac{3 p}{2} \right) $, 
we note that
\[
\vec{f}(1, \bar{\theta}(p), \phi) = \frac{3}{2} \sin(\bar{\theta}(p)) \hat{\vec{\theta}} = \frac{3}{2} \sech \frac{3 p}{2} \, \hat{\vec{\theta}} \, .
\]
Next, the derivative of $ \bar{\vec{x}} $ with respect to $ \alpha $ is quantified by
\[
\bar{\vec{x}}_{\alpha }(p, \alpha) =  (1 \sin \bar{\theta}(p)) \frac{\partial(2 \pi \alpha)}{\partial \alpha} \hat{\vec{\phi}} 
= 2 \pi \sech \frac{3 p}{2} \, \hat{\vec{\phi}} \, , 
\]
and so
\[
\vec{f}(1, \bar{\theta}(p), \phi) \wedge \bar{\vec{x}}_{\alpha }(p, \alpha) = 3 \pi \sech^{\! 2} \frac{3 p}{2} ~ \hat{\vec{r}} \, .
\]
A general perturbation $\vec{g} $ in (\ref{sys}) would be expressible in $ (r,\theta,\phi) $-coordinates
as 
\[
\vec{g}(r,\theta,\phi, t) = g_r(r,\theta,\phi, t) \hat{\vec{r}} + g_{\theta}(r,\theta,\phi, t) \hat{\vec{\theta}} + g_{\phi}(r,\theta,\phi, t) \hat{\vec{\phi}} \, .
\]
We need neither specify that $ \vec{g} $ be volume-preserving, nor time-periodic.  When $ \epsilon \ne 0 $ but
is small, $ \vec{a} $ perturbs to a time-varying hyperbolic trajectory $ \vec{a}_\epsilon(t) $, retaining its unstable 
manifold $ \Gamma^u_\epsilon(\vec{a}_\epsilon,t) $ which remains close to $ \Gamma $. Now, its location is associated 
with the unstable Melnikov function, which using Theorem~\ref{theorem:melnikov_unstable} becomes
\[
M^u \left(p,\alpha, t \right) = 3 \pi \int_{-\infty}^{p} \sech^{\! 2} \frac{3 \tau}{2 } g_r \left(1,\cos^{-1} \left( - \tanh \frac{3 \tau}{2} \right) ,
2 \pi \alpha, \tau + t - p\right) \d \tau \, .
\]
By virtue of Remark~\ref{remark:unstableapprox}, this means that the part of $ \Gamma_\epsilon^u $  close to $ \Gamma $ can be approximately parameterized by
\begin{align*}
\vec{r}^u(p,\alpha,\epsilon,t) &\approx  \bar{\vec{x}}(p, \alpha) + \epsilon \, M^u(p, \alpha, t) \frac{ \vec{f}(\bar{\vec{x}}(p, \alpha)) \wedge \bar{\vec{x}}_\alpha(p, \alpha)}{\lvert \vec{f}(\bar{\vec{x}}(p, \alpha)) \wedge \bar{\vec{x}}_\alpha(p, \alpha)\rvert^2} \\
&= \left( 1 \! + \! \epsilon \cosh^2 \frac{3 p}{2} \int_{-\infty}^{p} \! \! \! \! \! \sech^{\! 2} \frac{3 \tau}{2 } \, g_r \left(1,\cos^{-1} \left( - \tanh \frac{3 \tau}{2} \right) ,2 \pi \alpha, \tau \! + \! t \! - \! p\right) \d \tau  ,\cos^{-1} \left( - \tanh \frac{3 p}{2} \right)  ,  2 \pi \alpha  \right) \,
 \, , 
\end{align*}
for $ (p,\alpha,t) \in (-\infty, P^u] \times {\mathrm{S}}^1 \times (-\infty,T^u] $ for finite
$ P^u $ and $ T^u $, in  $ (r,\theta,\phi ) $-component form.  We highlight that is only the $ r $-component of $ \vec{g} $ which contributes to the leading-order displacement of $ \Gamma $
in the direction normal to it.  
Similarly $ \vec{b}_\epsilon(t) $'s stable manifold $ \Gamma_\epsilon^s(\vec{b}_\epsilon,t) $ is approximately
parameterizable via (calculations not shown)
\[
\vec{r}^s(p,\alpha,\epsilon,t) \approx \left( 1 \! - \! 
\epsilon \cosh^2 \frac{3 p}{2} \int_{p}^{\infty} \! \! \! \sech^{\! 2} \frac{3 \tau}{2 } \, g_r \left(1,\cos^{-1} \left( - \tanh \frac{3 \tau}{2} \right)  ,2 \pi \alpha, \tau + t - p\right) \d \tau   ,  \cos^{-1} \left( - \tanh \frac{3 p}{2} \right) , 2 \pi \alpha \right) \, \, , 
\]
for $ (p,\alpha,t) \in [P^s,\infty) \times {\mathrm{S}}^1 \times [T^s,\infty) $ for finite
$ P^s $ and $ T^s $.  The stable and unstable manifolds will generically no longer coincide, and 
the distance between the manifolds at a space-time location $ (p,\alpha,t) $ is encoded within the Melnikov function
\begin{equation}
M \left(p, \alpha, t \right) = 3 \pi \int_{-\infty}^{\infty} \sech^{\! 2} \frac{3 \tau}{2 } \,  g_r \left(1, \cos^{-1} \left( - \tanh \frac{3 \tau}{2} \right) , 2 \pi \alpha, \tau + t - p\right) \d \tau \, ,
\label{eq:hill_melnikov}
\end{equation}
for $ (p,\alpha,t) \in [P^s,P^u] \times {\mathrm{S}}^1 \times [T^s,T^u] $
as obtained via Theorem~\ref{theorem:heteroclinic}.  We therefore have a fairly complete description of the perturbed
manifolds and their splitting, and can use (\ref{eq:hill_melnikov}), for example, to easily compute conditions on infinitely many transverse intersections
when $ \vec{g} $ has time-periodicity.

\begin{figure}[t]
\centering
\subfigure[]{\includegraphics[scale=0.3]{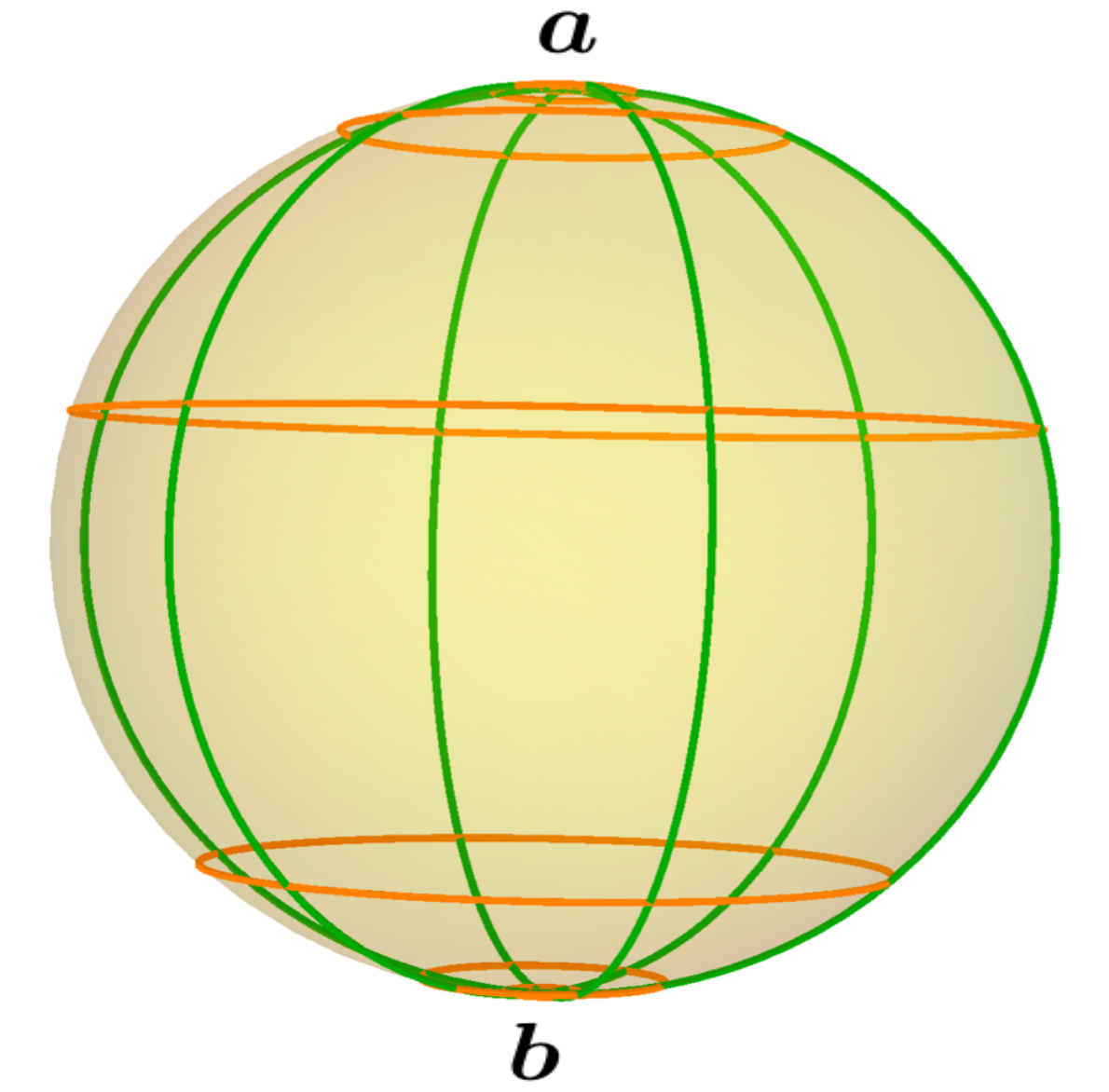}}
\hspace{0.15in}
\subfigure[]{\includegraphics[scale=0.35]{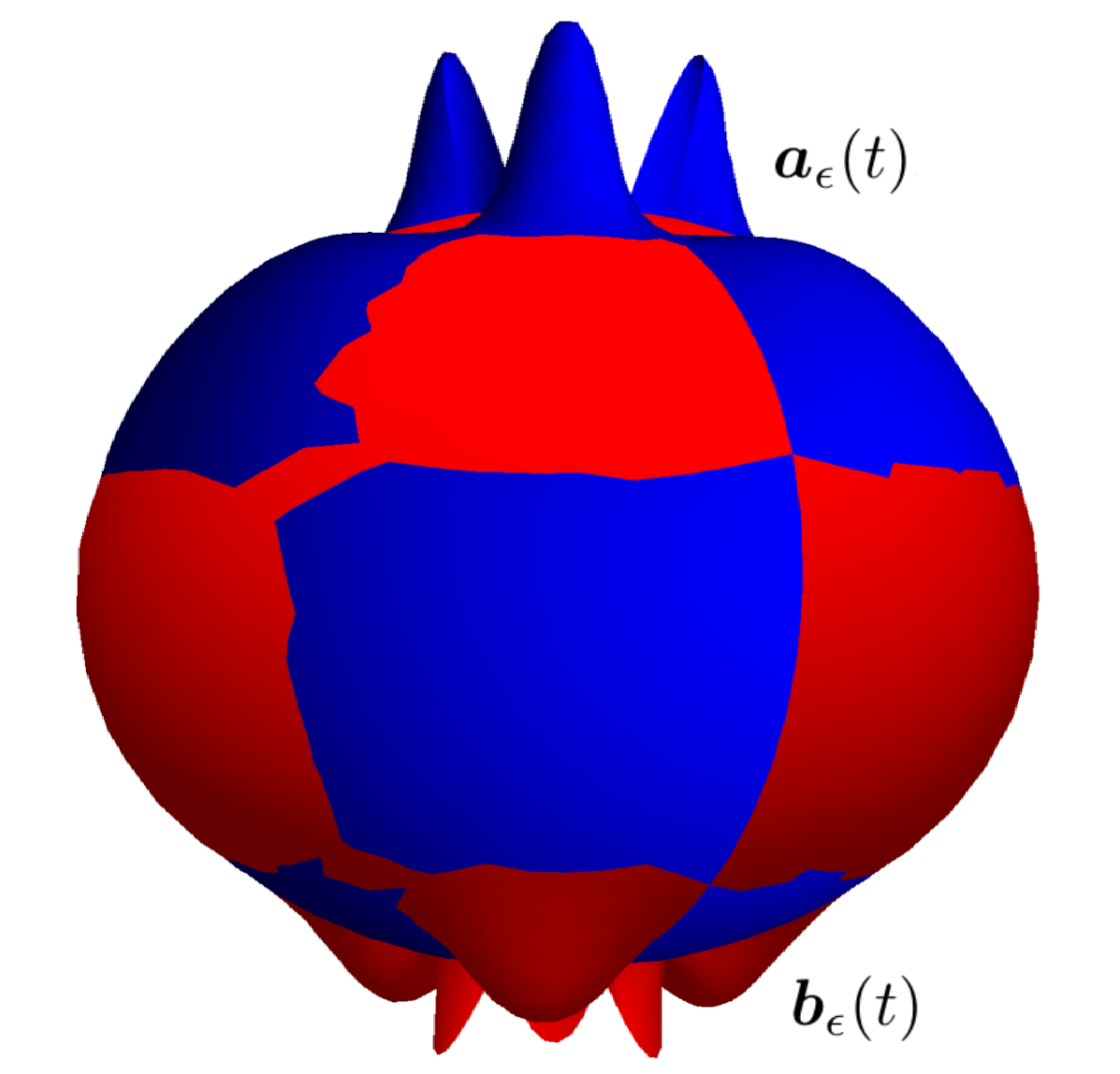}}
\caption{ (a) The zero contours of $ M $ associated with the perturbation
(\ref{eq:gr}) for the classical Hill's spherical vortex, at time $ t = 1 $ with the $ \theta $ and $ \phi $ constant curves shown in  green and orange respectively. (b) The resulting intersections between the
stable (blue) and unstable (red) manifolds, with the visible color being that of the exterior manifold.} 
\label{fig:hill_intersections}
\end{figure}

We will next demonstrate calculations for a chosen $ g_ r $, specifically
\begin{equation}
g_r \left(r, \theta, \phi, t \right) = r^2 \sin \theta \sin \left( 3 \phi \right)  \cos \left(4  t \right) \, .
\label{eq:gr}
\end{equation}
The integrand of each of the Melnikov functions then becomes
\[
\sech^{\! 2} \frac{3 \tau}{2}  \, r^2 \Big|_{r=1} \left( \sech \frac{3 \tau}{2} \right) \sin \left( 6 \pi \alpha \right) \cos \left[
4 \left( \tau + t - p \right) \right] = \sech^{\! 3} \frac{3 \tau}{2}  \sin \left( 6 \pi \alpha \right) \cos \left[
4 \left( \tau + t - p \right) \right] \, .
\]
Trigonometric addition formulas and odd/even-ness of integrands allow us to derive the {\em explicit}
analytic expression
\[
M(p,\alpha,t) = 6 \pi \sin(6 \pi \alpha)  \cos\left[ 4 (p-t)  \right] \int_{0}^{\infty} \! \sech^{\! 3} \left(\frac{3 \tau}{2}\right) \cos \left( 4 \tau \right) \d \tau = \frac{73 \, \pi^2}{9} \sech \frac{4 \pi}{3} \, \sin \left( 6 \pi \alpha \right) \cos \left[ 4 (p-t) \right] \, .
\]
Consider the perturbed manifolds at a fixed time $ t $.  If $ p $ satisfies $ p - t  \ne (2 k + 1) \pi/8 $ for $ k \in \mathbb{Z} $, $ M $ clearly has simple zeros
when $ \alpha = 0, 1/6, 1/3, 1/2, 2/3 $ and $ 5/6 $.  Thus, the two-dimensional perturbed manifolds intersect along along curves which are $ {\mathcal O}(\epsilon) $-close to these six constant longitude lines $ \phi = 0, \pi/3, 2 \pi/3, \pi, 
4 \pi/3 $ and $ 5 \pi/3 $ on $ \Gamma $.  Moreover, at
$ \alpha $-values not on these curves, $ M $ has simple zeros with respect to $ p $ when 
\[
p = p_k := t + (2k+1)\pi/8 \, , 
\]
i.e., on the latitudes defined by
\[
\theta = \theta_k := \cos^{-1} \left[ - \tanh \left( \frac{3 t}{2} + \frac{3(2k+1)\pi}{16} \right) \right] \quad ; \quad k \in \mathbb{Z} \, ,  
\]
which yields infinitely many unique values for $ \theta \in [0, \pi] $, accumulating towards both
$ \theta = 0 $ and $ \pi $.  At a chosen time $ t = 1 $, we illustrate in Fig.~\ref{fig:hill_intersections}(a) the zero contours
of $ M $.  Thus, the perturbed stable and unstable manifolds  at time $ 1 $ intersect along curves which are $ {\mathcal{O} }(\eps) $-close to these constant longitude and latitude
curves.     Crossing a zero contour implies that $ \Gamma_\eps^u $
flips from being outside $ \Gamma_\eps^s $ (or vice versa). The intersection of the approximated perturbed stable and unstable manifolds at time $t=1$ are illustrated in Fig.~\ref{fig:hill_intersections}(b).

 \begin{figure}[t]
	\centering
	\subfigure[]{\includegraphics[scale=0.29]{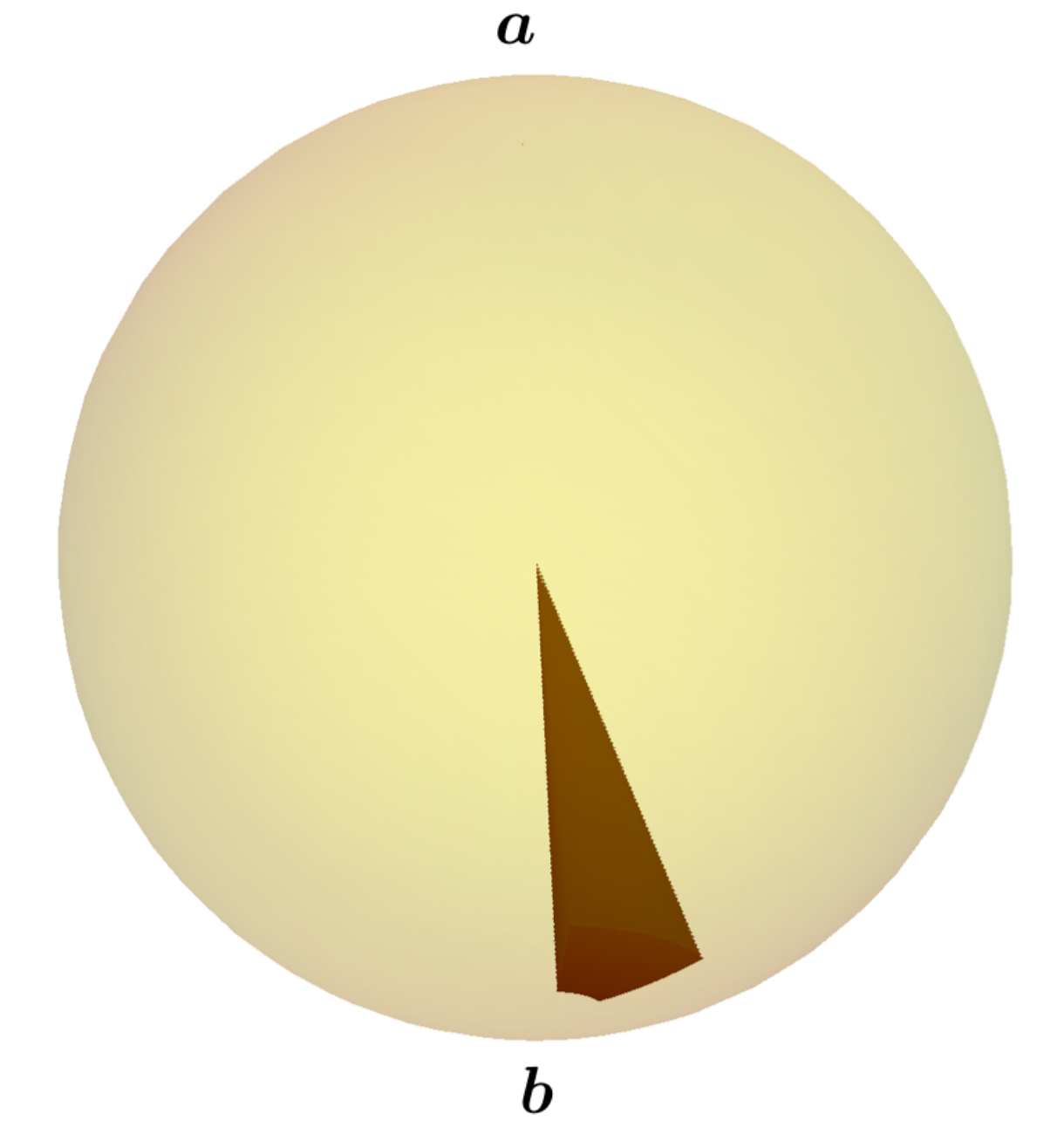}}
	\hspace{0.35in}
	\subfigure[]{\includegraphics[scale=0.288]{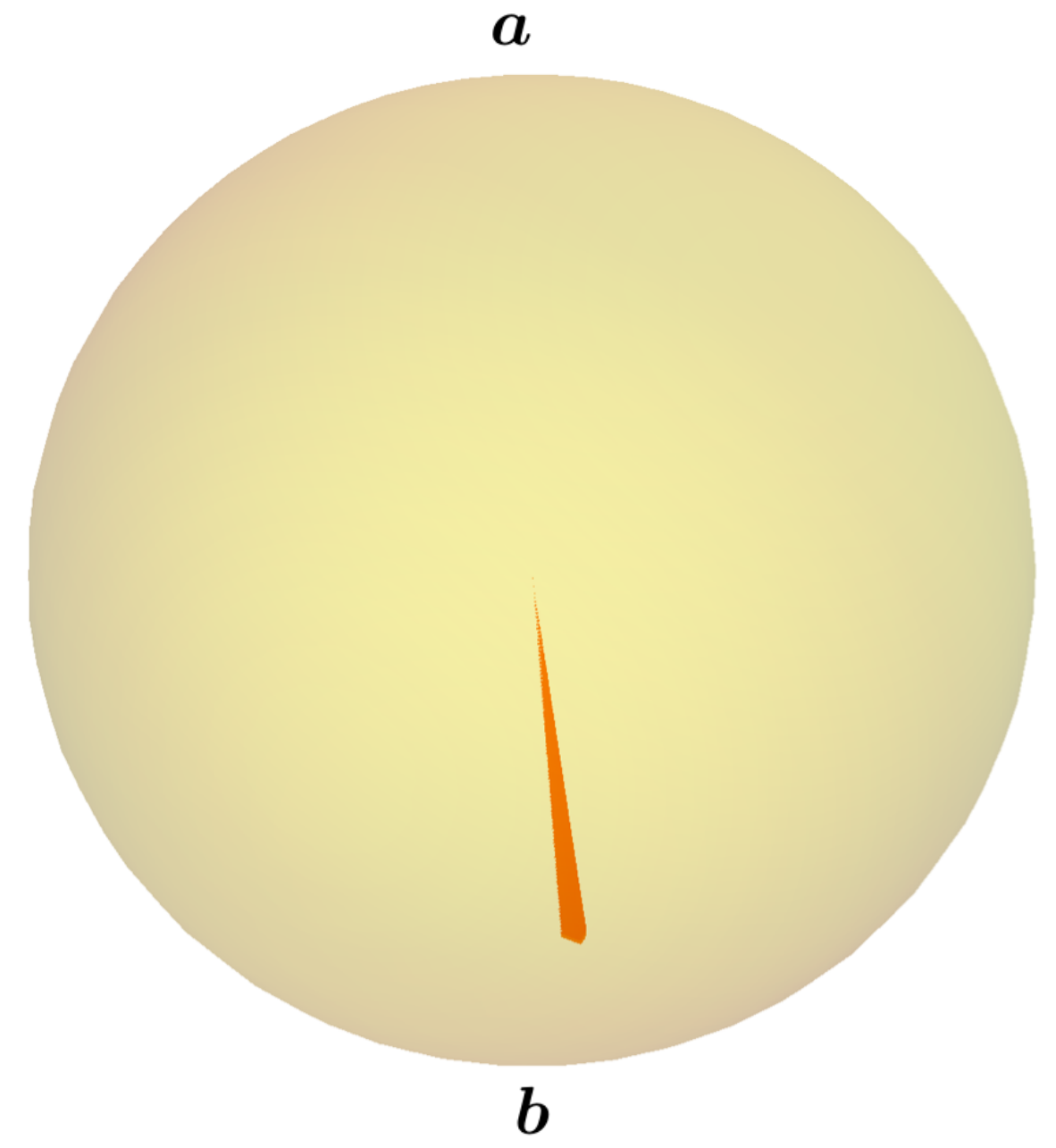}}
	\caption{ The lobes whose $ \alpha $ limits are between $ 0 $ and $ 1/6 $, and $ p $ limits are between 
		$ k = 1 $ and $ k = 2 $ as defined in $ p_k $ are shown for $t=0$ and $t=1$.  } 
	\label{fig:hill_region}
\end{figure}

We can approximate the volume of a lobe whose boundaries are given by adjacent zeros of $ p $ and $ \alpha $.  For example, consider the lobe whose $ \alpha $ limits are between $ 0 $ and $ 1/6 $, and $ p $ limits are between 
$ k = 1 $ and $ k = 2 $ as defined in $ p_k $.  The lobes whose are bounded by the above mentioned $p$ and $\alpha$ are shown for $t=0$ and $t=1$ in Fig.~\ref{fig:hill_region}. Using (\ref{eq:lobevolume}), we see that at a general value $t,$
\[
\mathrm{Lobe~volume} = \epsilon \frac{73 \, \pi^2}{9} \sech \frac{4 \pi}{3} \int_0^{1/6} \! \! \sin \left( 6 \pi \alpha \right) \, \d 
\alpha \int_{p_1}^{p_2} \! \! \cos \left[ 4(p-t) \right] \, \d p + {\mathcal O}(\epsilon^2) = 
\epsilon \frac{73 \, \pi}{54}  \sech \frac{4 \pi}{3} + {\mathcal O}(\epsilon^2) \, .
\]
As expected, the leading-order lobe volumes are identical for all lobes. They are also equal for
all times $ t $ (the apparent difference displayed in Fig.~\ref{fig:hill_region} is because while the
region in $ (p,\alpha) $ space subtended by the lobe reduces at $ t = 1 $, the lobe has a greater
extent in the normal direction to $ \Gamma $).

The instantaneous flux $\Phi \left(p, t, \eps\right)$ exiting the pseudo-separatrix, which is given in Theorem \ref{theorem:flux}, for this choice of perturbation can be computed as 
\begin{equation*}
   \Phi \left(p, t, \eps\right) = \frac{73 \,\eps\, \pi^2}{9} \sech \left(\frac{4 \pi}{3}\right) \,\cos \left[ 4 (p-t) \right]\, \int_0^1 \sin \left( 6 \pi \alpha \right)  \, d \alpha = 0.  
\end{equation*} 
The reason for the leading-order instantaneous flux to be zero (for any time $ t $ and any choice of gate location $ p $) is because of the symmetry of the splitting of the heteroclinic as intimated via Fig.~\ref{fig:hill_region}.  Regions along any constant latitude strip in which the unstable manifold is outside the stable manifold are complemented by regions in which the opposite occurs, while the leading-order velocity field along the strip remains constant.  Thus, there is an identical amount of flux crossing outwards as that crossing inwards. One would get a nonzero flux if the $ \sin (3 \phi) $ term in
$ g_r $ (which led to the integral $ \int_0^1 \sin \left( 6 \pi \alpha \right)  \, d \alpha $ in the above expression
for the flux) were replaced by a term which does {\em not} integrate to zero over all $ \phi $.

The expressions for $ \vec{r}^{u,s} $ also allow for approximating the perturbed manifolds for the choice of $ \vec{g} $
in (\ref{eq:gr}).  At a choice of time $ t = 1 $ and perturbation strength $ \eps = 0.1 $, the $ (p,\alpha) $-variables
 parameterize the manifolds, as given by the $ (r,\theta,\phi) $-coordinates in the expressions we have obtained for $ \vec{r}^{u,s} $.  We can visualize the manifolds by evaluating the curves $ \alpha = $ constant (with $ p $-varying along each
 such curve), and also the curves $ p = $ constant (with $ \alpha $ varying along each curve).  We show the approximate
 perturbed unstable manifold $ \Gamma_\eps^u $ in red in the left figure of Fig.~\ref{fig:hill_manifolds}, in comparison to the unperturbed
 $ \Gamma $ (black), which is also illustrated via plotting $ \alpha = $ constant and $ p =$ constant curves.  Similarly,
 we use $ \vec{r}^s $ in the same way to visualize the perturbed stable manifold $ \Gamma_\eps^s $ (in blue in the right
 figure). In
 these situations, we need to numerically approximate the integrals relevant to $ M^{u,s} $; this procedure could
 be followed in determining the perturbed manifolds at any time $ t $ for {\em any} given (sufficiently smooth and bounded) $ g_r $.
 
 \begin{figure}[t]
\centering
	\subfigure[]{\includegraphics[scale=0.40]{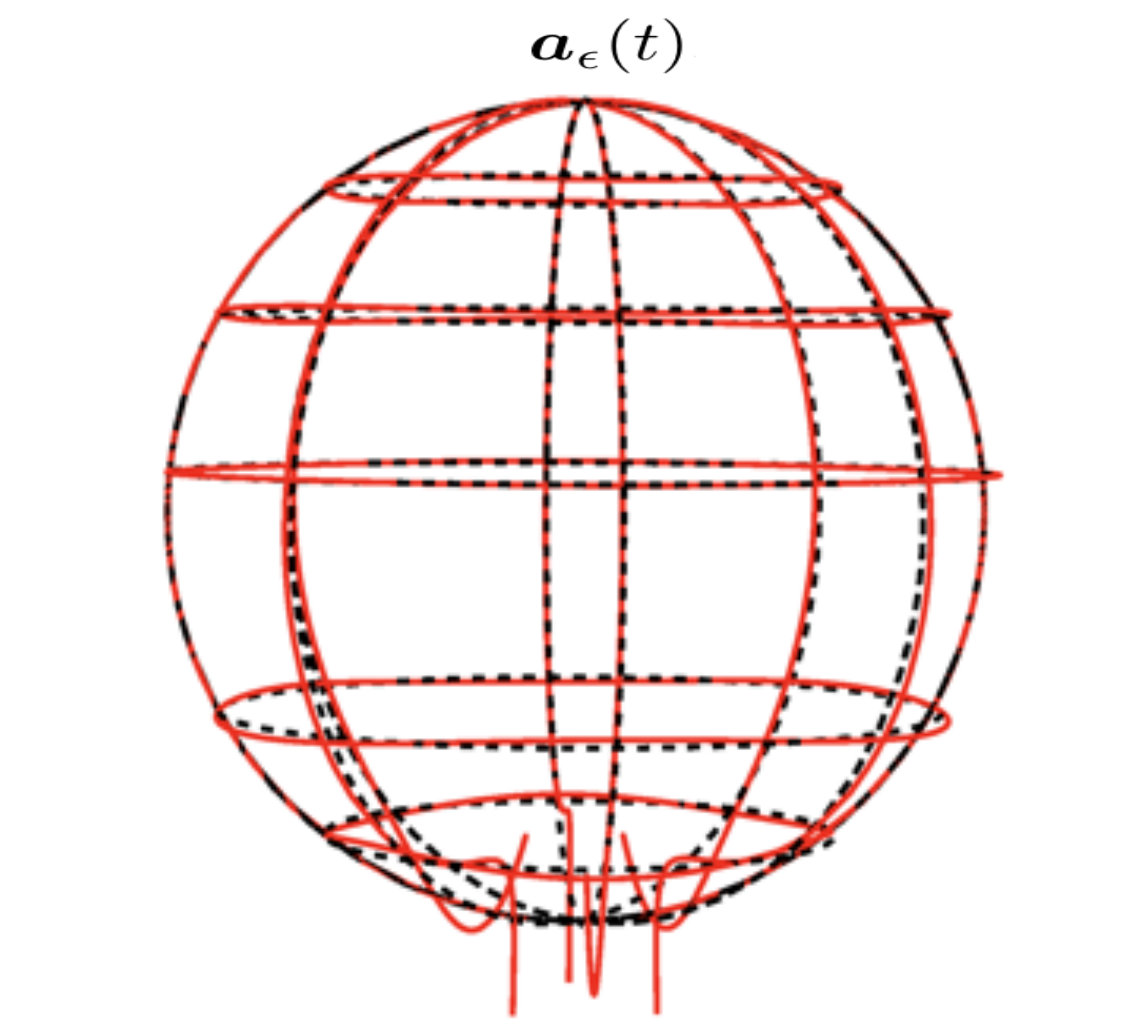}}
\hspace{0.15in}
\subfigure[]{\includegraphics[scale=0.42]{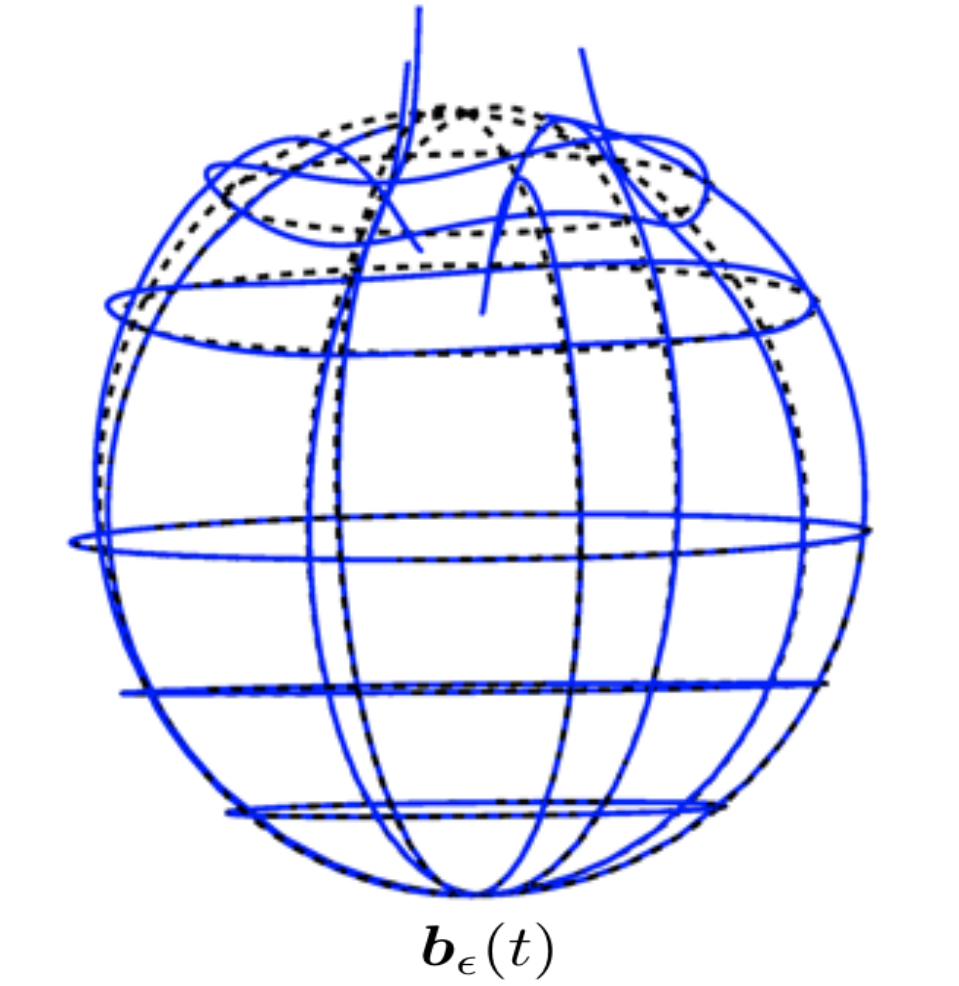}}
\caption{  For time $t = 1$ and $\epsilon = 0.1,$ (a) the perturbed unstable manifold (in red) and (b) the perturbed stable manifold (in blue), in comparison to the unperturbed manifold (in dashed black. Further, this illustration was performed via plotting $ \alpha = $ constant and $ p =$ constant curves, for the perturbation (\ref{eq:gr}).  } 
\label{fig:hill_manifolds}
\end{figure}

The approximation of the perturbed stable and unstable manifolds for the perturbation (\ref{eq:gr}) is given in Fig.~\ref{fig:hill_manifolds}. These two perturbed manifolds are drawn for $ t = 1 $ and $ \eps = 0.1. $ As the hyperbolic trajectory point $\vec{b}_\eps $ is approached along a constant 
longitude, the perturbed unstable manifold intersects the unperturbed manifold infinitely often. The classical `heteroclinic tangle' occurs when viewing the manifold intersections along such a constant
longitude.  The perturbative (in $ \eps $) approximation for the unstable manifold breaks down as this region is approached, because the manifolds stretch out substantially.  Similarly, we cannot approximate the perturbed stable manifold near the fixed point $\vec{a}$.  
We illustrat this behavior in Fig.~\ref{fig:hill_time}, where the intersections between the manifolds with several constant latitudes are shown at two different $ t $ values. The perturbed unstable trajectory approximated at time $t = 2$ near the ``Antarctic circle'' and the perturbed stable trajectory approximated at time $t = 2$ near the ``Arctic circle'' demonstrate significant 
deviation from the unperturbed manifold, an effect which is exacerbated at larger times. 


\begin{figure}[t]
\centering
\subfigure[]{\includegraphics[scale=0.30]{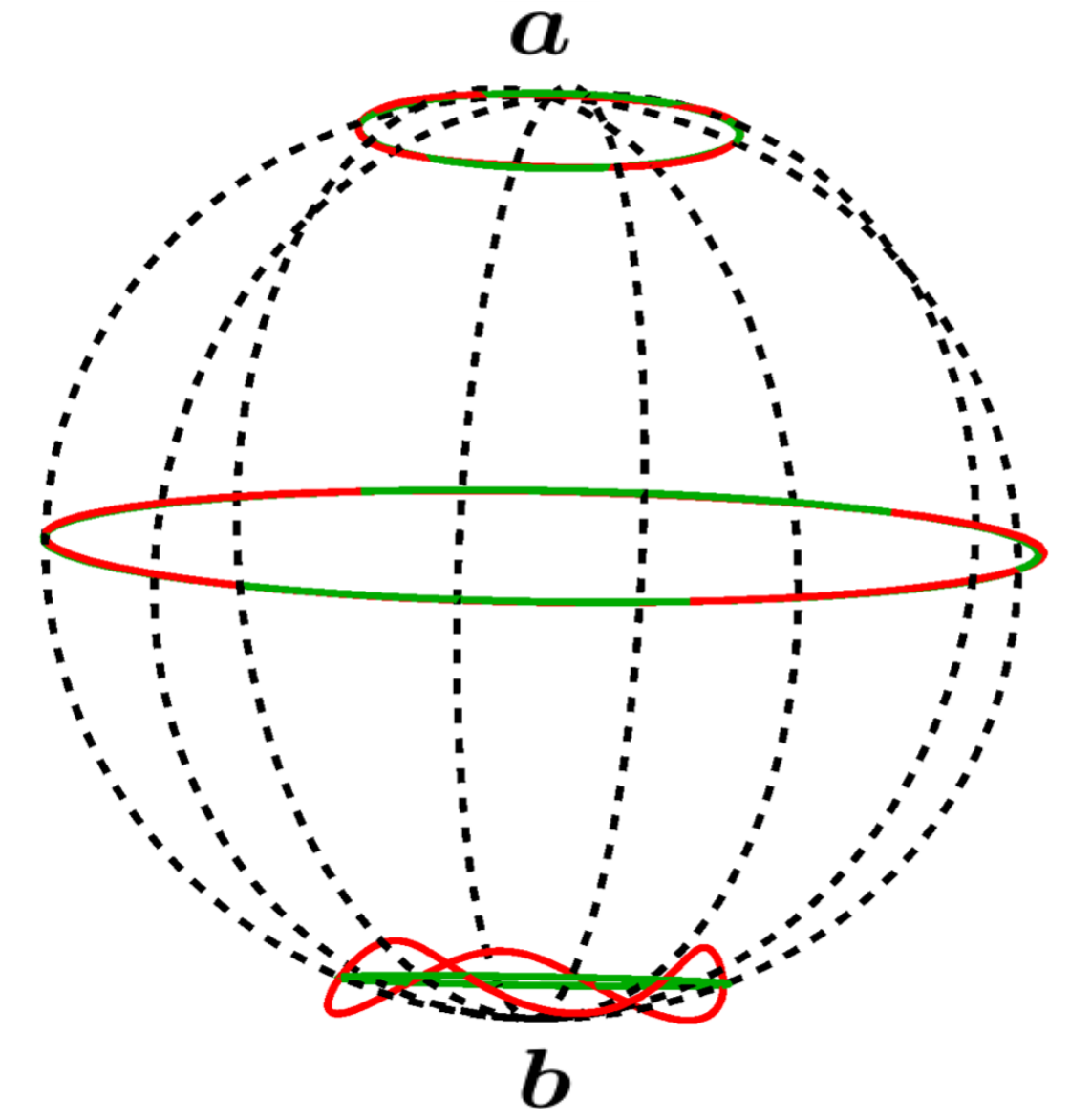}}
\hspace{0.25in}
\subfigure[]{\includegraphics[scale=0.30]{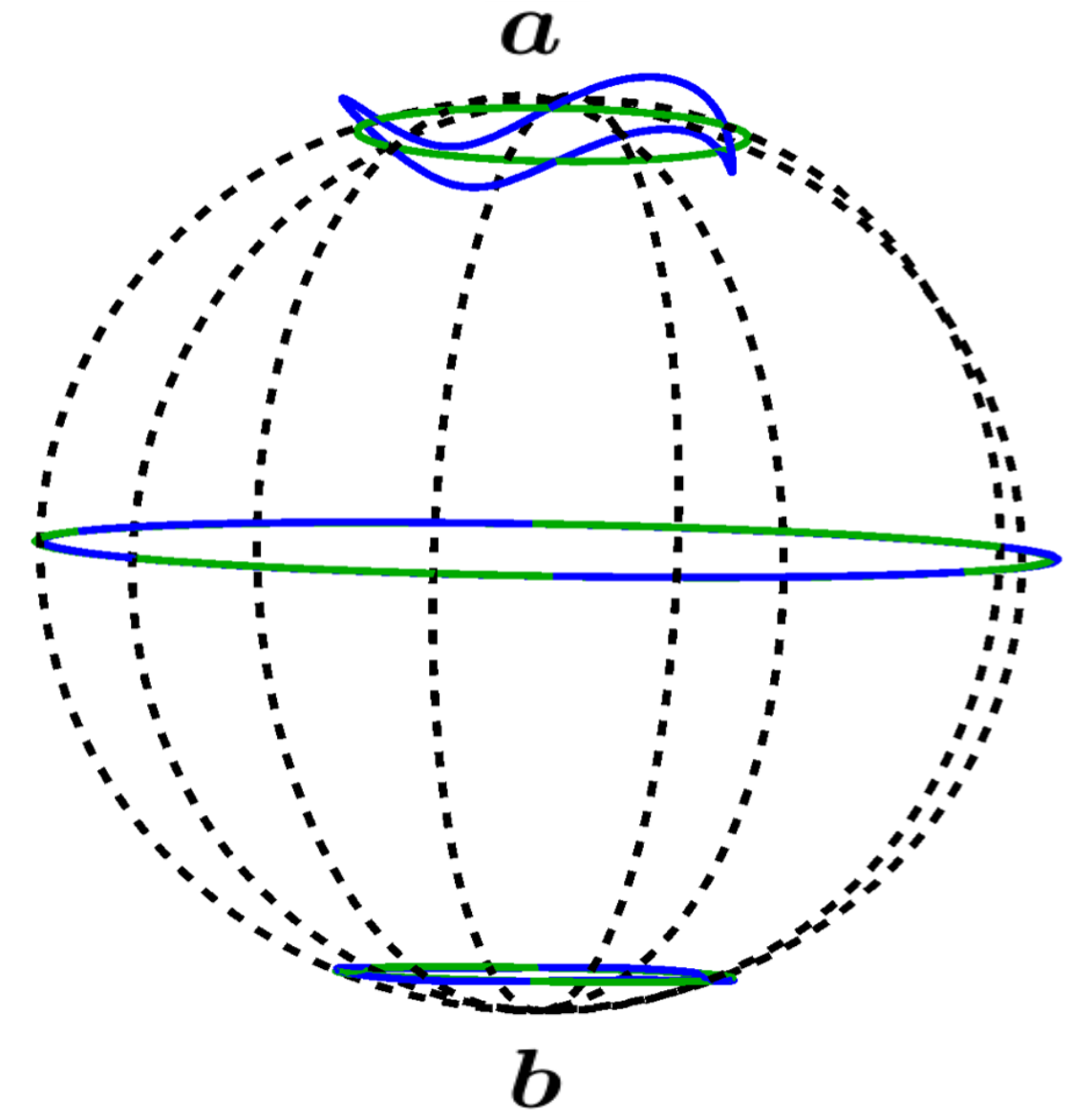}}
\subfigure[]{\includegraphics[scale=0.30]{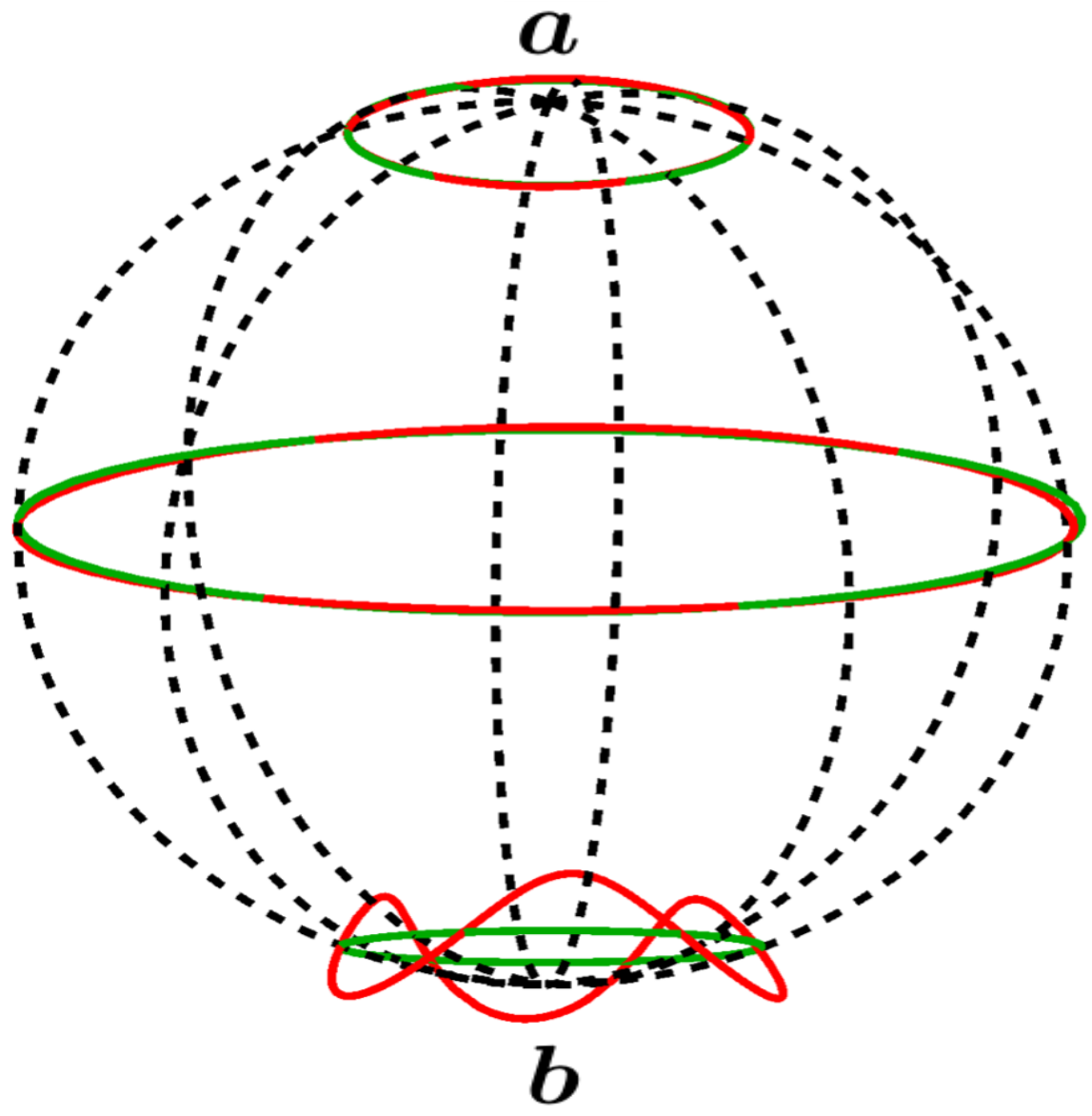}}
\hspace{0.25in}
\subfigure[]{\includegraphics[scale=0.30]{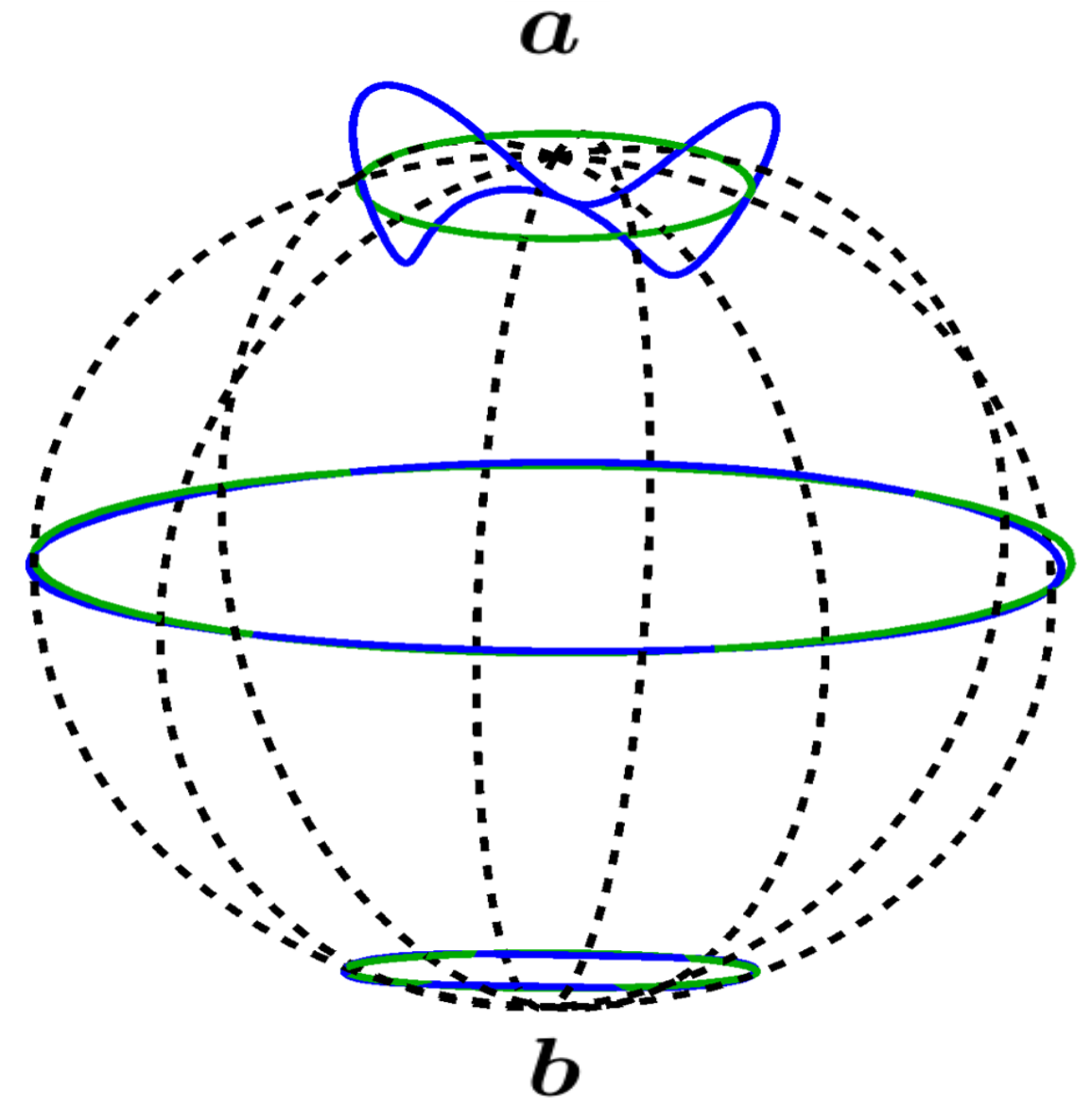}}
\caption{ Behavior of the perturbed unstable (red) and stable (blue) manifolds $ \Gamma_\eps^u $ and $ \Gamma_\eps^s $ with $ \eps = 0.1 $ for the classical Hill's spherical vortex, at $ \theta = 0.4 $ (approximate ``Arctic circle''), $ \theta = \pi/2 $ (``equator'') and $ \theta = \pi - 0.4 $ (approximate ``Antarctic circle''), with corresponding trajectories on unperturbed $ \Gamma $ (green), for the perturbation (\ref{eq:gr}), with the rows being $ t = 0 $ (top) and $ t = 2 $ (bottom).
} 
\label{fig:hill_time}
\end{figure}

\subsection{Hill's spherical vortex with swirl}
\label{sec:hill_swirl}

\begin{figure}[t]
\centering
\includegraphics[width=3.4in]{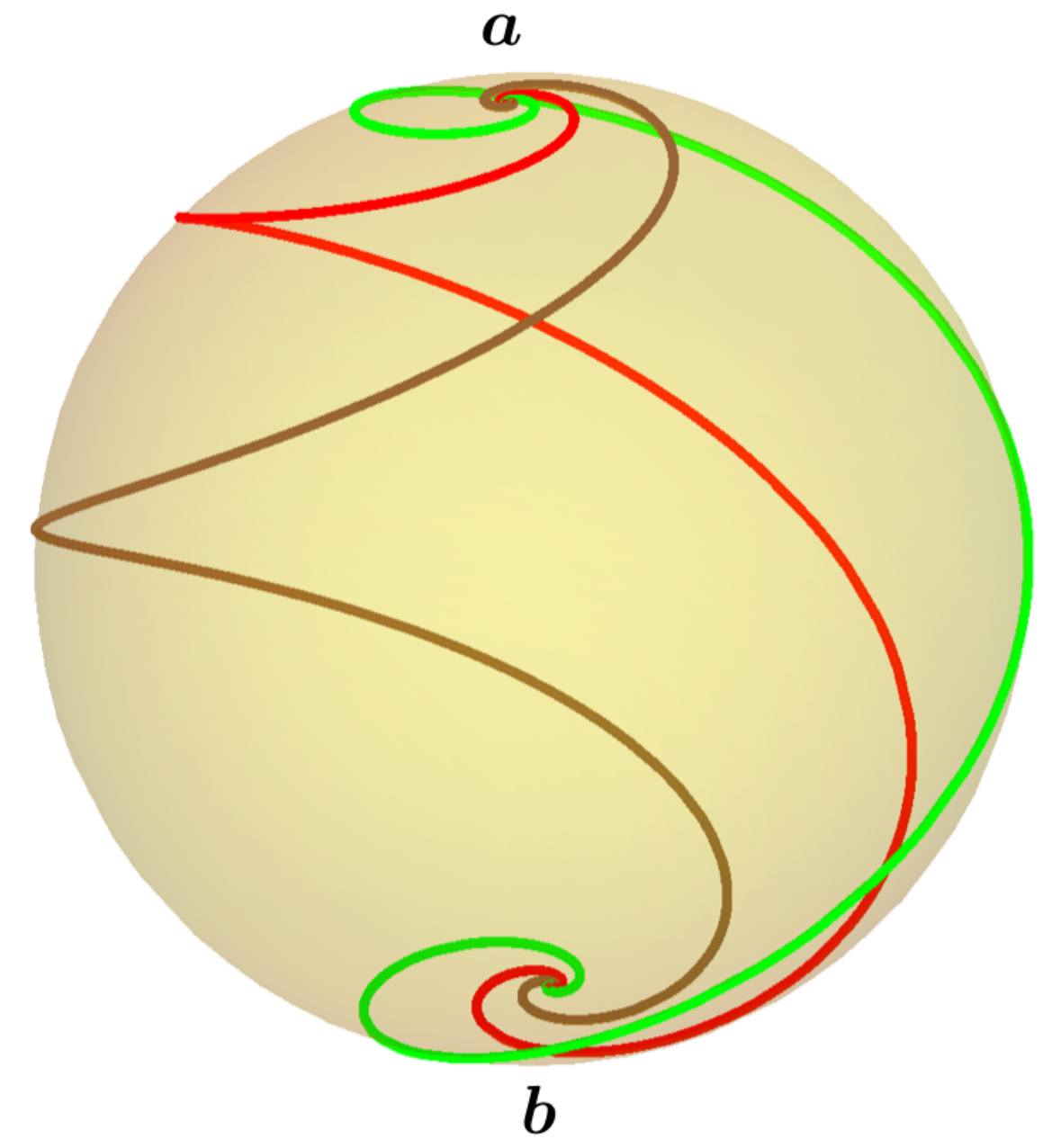}
\caption{Three different trajectories on $ \Gamma $ when $R_0=0.2$ for Hill's spherical vortex with swirl, identifiable
as $ \alpha = $ constant curves: $\alpha=0 $ (green), 
$ 1/4 $ (red) and $ 1/2 $ (brown).}
\label{fig:swirl_unp}
\end{figure} 

We now consider Hill's spherical vortex with an additional swirling component in the azimuthal direction, with 
the far-field flow remaining in the $ z $-direction as for the classical case.
The steady solution to the Euler equations in this case corresponds
to the continuous velocity field \cite{article_Swhill} 
\begin{equation}
\renewcommand{\arraystretch}{1.8}
\vec{f}(r, \theta, \phi) 
=  \left\{ \begin{array}{ll} 
\frac{3}{2} (1 - r^2) \cos(\theta)~ \hat{\vec{r}} - \frac{3}{2} (1 - 2 r^2) \sin(\theta) ~\hat{\vec{\theta}} - \frac{r \sin(\theta)}{2 R_0} ~\hat{\vec{\phi}} & ~~~~~ \text{if} ~~ r \le 1 \\
- \cos(\theta) \left \{1 - \frac{1}{r^2} \cos \left( \frac{r - 1}{R_0}\right) + \frac{R_0}{r^3} \sin \left( \frac{r - 1}{R_0}\right) \right \}  ~\hat{\vec{r}} & \\
~~ + \frac{\sin(\theta)}{2 r} \left \{2 r + \frac{1}{r} \cos \left( \frac{r - 1}{R_0}\right) + \left[\frac{1}{R_0} - \frac{R_0}{r^2} \right] \sin \left( \frac{r - 1}{R_0}\right) \right \}  ~\hat{\vec{\theta}} - \frac{r \sin(\theta)}{2 R_0} ~\hat{\vec{\phi}} & ~~~~~\text{if}  ~~ r > 1
\end{array} \right. \, , 
\end{equation}
where $ R_0 > 0 $ is the Rossby number. This velocity field is also divergence-free consonant with the incompressibility
assumption, and the classical Hill's spherical vortex is a special case in the
limit $R_0 \rightarrow \infty $.   The north and south poles $(r, \theta) = (1, 0)$ and $(1, \pi)$ are still saddle
points, and $ r = 1 $ is the heteroclinic manifold $ \Gamma $ connecting them.  However, heteroclinic trajectories swirl around
the globe, and spiral from the poles because the linearization at the poles results in complex 
conjugate eigenvalues.  As before, we parametrize $ \Gamma $  by $\bar{\vec{x}}(p,\alpha) = 
\left( 1, \bar{\theta}(p,\alpha), \bar{\phi}(p,\alpha) \right) $ in $ (r,\theta,\phi) $-coordinates, and we observe that
since the $ \theta $-component of the velocity is identical to the classical case, $ \bar{\theta} $ is
given by (\ref{eq:theta}), where we have chosen $ p = 0 $ to be on the equator.  Next, since $ \phi $ is changing 
at the constant rate $ - 1 / (2 R_0) $, we have
\[
\bar{\phi}(p,\alpha) = \bar{\phi}(0,\alpha) - \frac{p}{2 R_0} = 2 \pi \alpha - \frac{p}{2 R_0} \, , 
\]
where we choose the parameterization on the equator such that the $ \phi $-coordinate on the
equator divided by $ 2 \pi $ gives the trajectory-identifying parameter $ \alpha \in {\mathrm{S}}^1 $.  Thus, in
$ (r,\theta,\phi) $-component form, we have
\[
\bar{\vec{x}}(p,\alpha) = \left( 1, \cos^{-1} \left( - \tanh \frac{3 p}{2} \right), 2 \pi \alpha - \frac{p}{2 R_0}  \right) \, .
\]
We show three different trajectories on $ \Gamma $ in Fig.~\ref{fig:swirl_unp} using $ R_0 = 0.2 $ (which is the
choice used in all computations shown). 
 We note that 
the $ p $-variation in the $ \theta $-coordinate is equivalent to the time-variation in this steady situation, and changing $ \alpha $ only
affects the $ \phi $-component. 
Thus, despite $ \bar{\vec{x}} $ having a slight difference, we have $
\vec{f}(1, \bar{\theta}(p), \bar{\phi}(p,\alpha)) \wedge \bar{\vec{x}}_{\alpha }(\alpha , p) = 3 \pi \sech^{\! 2} (3 p / 2 )\hat{\vec{r}} $
as before.  Consequently, only very slight adjustments to the results for the classical Hill's vortex are necessary.
We will not bother to rewrite these equations, apart from stating that 
the $ \phi $-coordinate within the $ g_r $ function in all the integrals simply needs to change from
$ 2 \pi \alpha $ to $ 2 \pi \alpha - \tau / (2 R_0) $.  


\begin{figure}[t]
\centering
	\subfigure[]{\includegraphics[scale=0.35]{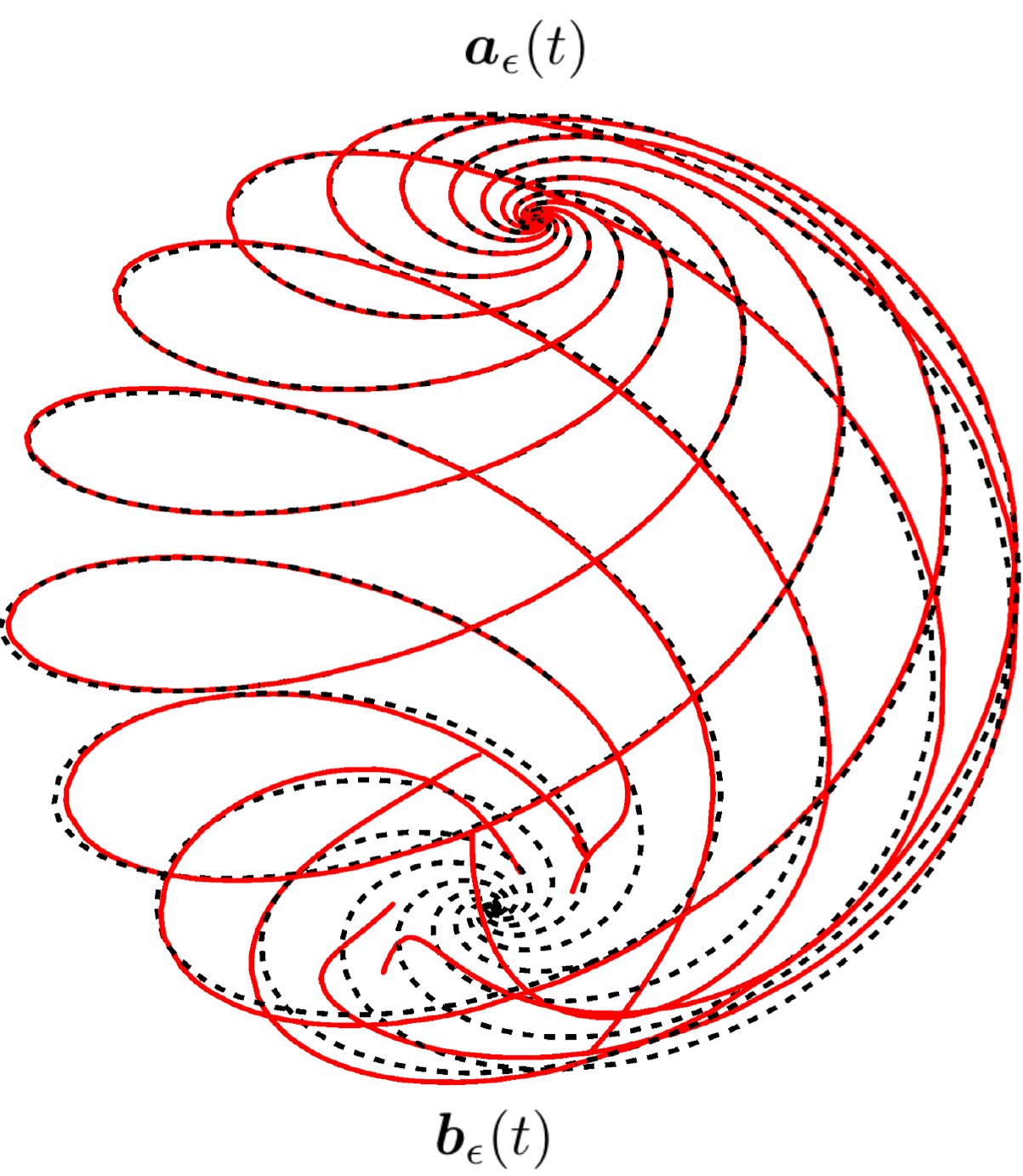}}
\hspace{0.15in}
\subfigure[]{\includegraphics[scale=0.36]{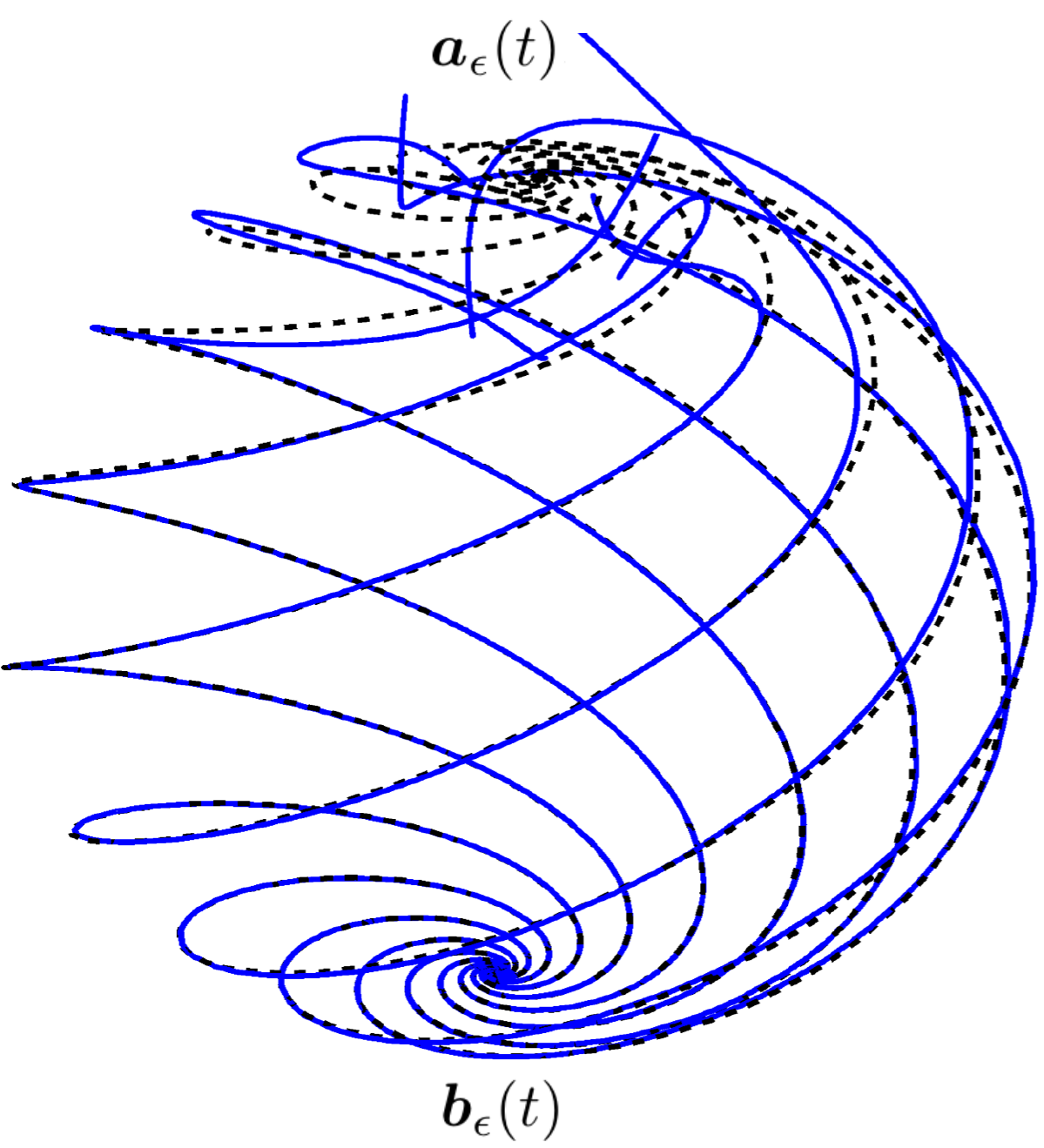}}
\caption{ The $\alpha =$ constant curves of perturbed unstable (red) and stable (blue) manifolds $ \Gamma_\eps^u $ and $ \Gamma_\eps^s $ at time $ t = 1 $ with $ \eps = 0.1, $ for Hill's spherical vortex with swirl $ R_0 = 0.2, $ for the perturbation (\ref{eq:gr}). The unperturbed $\alpha =$ constant curves are given in dashed-black.
}
\label{fig:swirl_manifolds}
\end{figure}

We can use the expressions for $ \vec{r}^{u,s} $ as before, but with this modification, to numerically approximate
the perturbed stable and unstable manifolds when $ \epsilon = 0.1 $.  Once again, we use
the perturbation $ \vec{g} $ as given in (\ref{eq:gr}).  The perturbed unstable (red) and stable (blue) manifolds at
time $ t = 1 $ are shown in Fig.~\ref{fig:swirl_manifolds}. We need to choose $p$ values carefully as these approximations for perturbed unstable and stable manifolds lose validity near the perturbed hyperbolic trajectory locations $\vec{b}_\eps(t)$ and $\vec{a}_\eps(t)$ respectively. 


The Melnikov function for this choice of $ g_r $ is
\begin{align*}
M(p,\alpha,t) &= 3 \pi \int_{-\infty}^\infty \sech^{\! 3} \frac{3 \tau}{2}  \sin \left( 6 \pi \alpha - \frac{3 \tau}{2 R_0} \right) \cos \left[
4 \left( \tau + t - p \right) \right] \, \d \tau \\
&= 6 \pi \sqrt{A^2 \sin^2 \left( 6 \pi \alpha \right) + B^2 \cos^2 \left( 6 \pi \alpha \right)} \cos \left[  4(t-p) -
\tan^{-1} \frac{B \cot \left( 6 \pi \alpha \right)}{A} \right] \, , 
\end{align*}
where $ A $ and $ B $ are the values of the two improper integrals, which evaluate respectively to
\begin{align*}
A &=  \frac{\pi \left[ \left( 73 R_0^2 - 48 R_0 + 9 \right) \sech \frac{\pi(8 R_0 - 3)}{6 R_0} +
\left( 73 R_0^2 + 48 R_0 + 9 \right) \sech \frac{\pi(8 R_0 + 3)}{6 R_0} \right]}{108 R_0^2} \, , \quad  {\mathrm{and}} \\
B &=  \frac{\pi \left[ \left( 73 R_0^2 - 48 R_0 + 9 \right) \sech \frac{\pi(8 R_0 - 3)}{6 R_0} -
\left( 73 R_0^2 + 48 R_0 + 9 \right) \sech \frac{\pi(8 R_0 + 3)}{6 R_0} \right]}{108 R_0^2} \, .
\end{align*}
Consequently, at a fixed time $ t $, simple zeros of $ M $ occur along the collection of $ (p,\alpha) $ curves given by
\[
\tan \left[ 4(t-p) - \frac{(2k+1)\pi}{2} \right] = \frac{B}{A} \cot \left( 6 \pi \alpha \right) \quad ; \quad
k \in \mathbb{Z} \, .
\]
In view of the $ \pi $-periodicity of the tangent function, the $ k $-dependence disappears, and we
may as well take $ k = 0 $.
When expressed in $ (\theta,\phi) $ coordinates, the condition therefore becomes
	\begin{equation}
		- \cot \left[ 4 \left( t + \frac{2}{3} \tanh^{-1} \cos \theta \right) \right] = \frac{B}{A}
	\cot \left( 3 \phi - \frac{1}{R_0} \tanh^{-1} \cos \theta \right) \, .	
	\label{trigeq}
\end{equation}
There are six zero contour curves (in green) on $ \Gamma $, as shown in Fig.~\ref{fig:swirl_intersections}(a). 
\begin{figure}[t]
\centering
\subfigure[]{\includegraphics[scale=0.37]{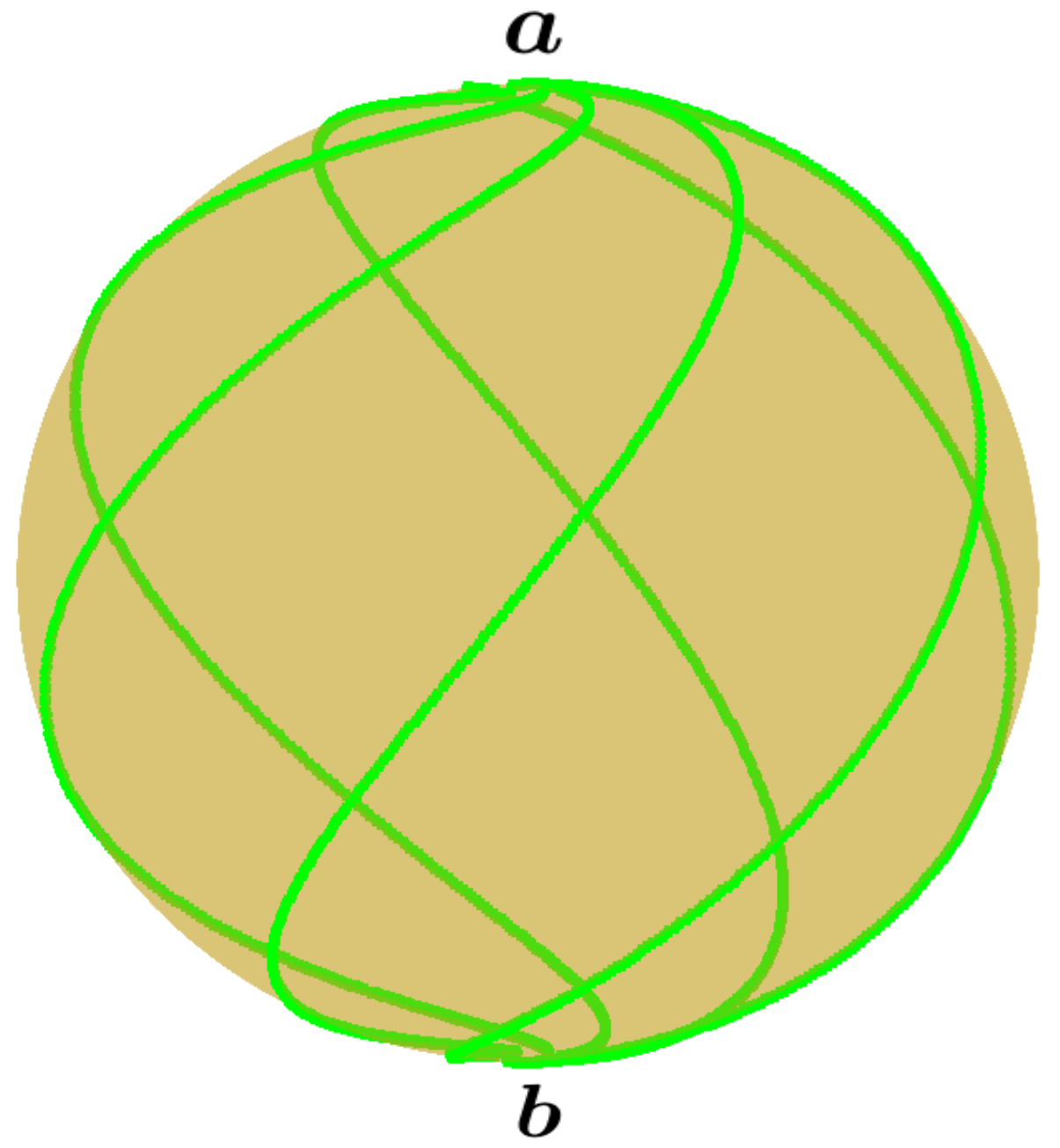}}
\hspace{0.25in}
\subfigure[]{\includegraphics[scale=0.37]{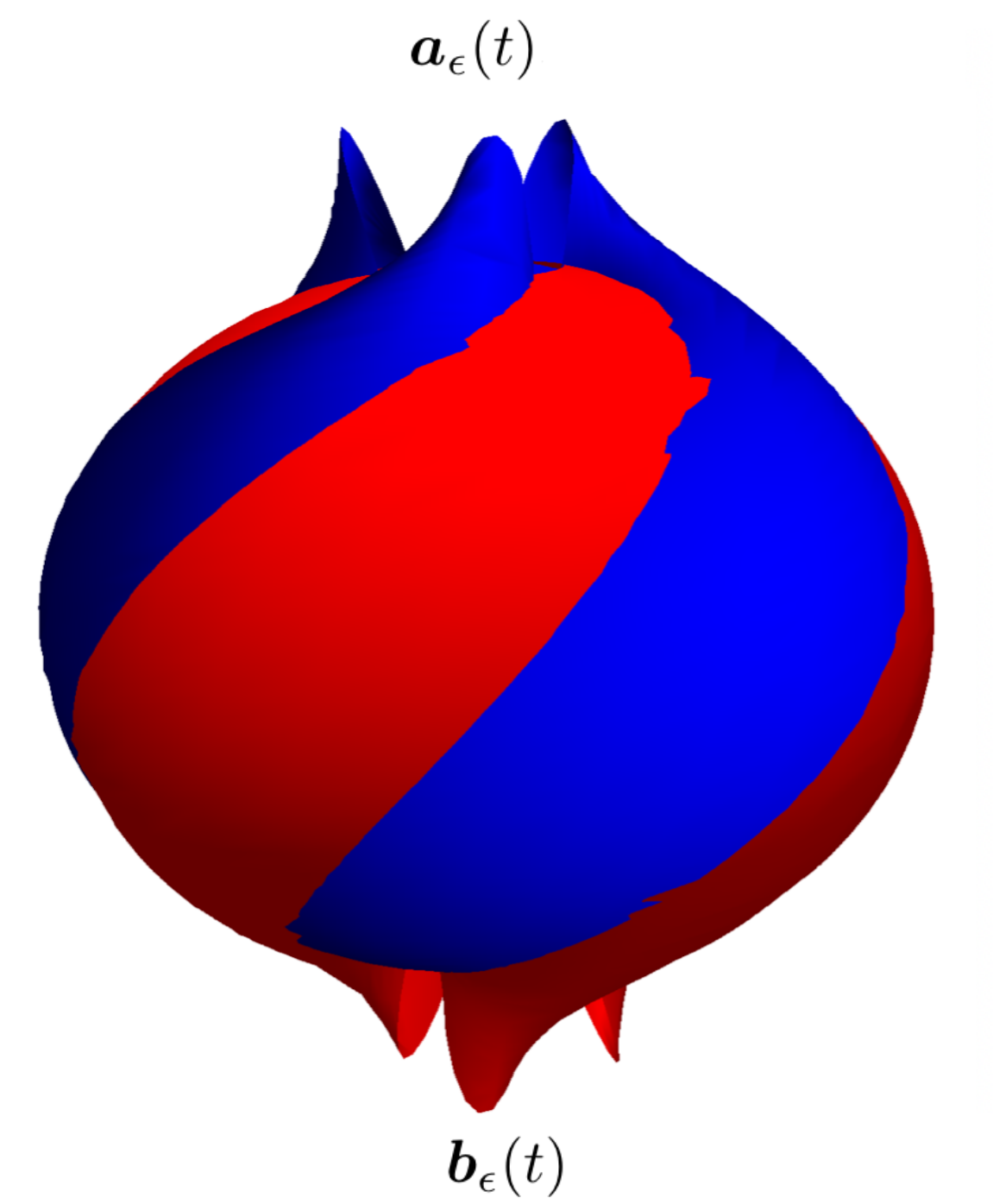}}
\caption{The zero contours of $ M $ associated with the perturbation
(\ref{eq:gr}) for Hill's spherical vortex with swirl with $ R_0 = 0.2 $, at time $ t = 1 $.} 
\label{fig:swirl_intersections}
\end{figure}
Next we display the intersection of perturbed stable (blue) and unstable (red) manifolds in Fig.~\ref{fig:swirl_intersections}(b) at $t = 1$ for the perturbation (\ref{eq:gr}). 

Finally the instantaneous flux $\Phi \left(p, t, \eps\right)$ exiting the pseudo-separatrix for the  Hill's spherical vortex with swirl is
\begin{equation*}
   \Phi \left(p, t, \eps\right) = 3 \pi \int_{-\infty}^\infty \sech^{\! 3} \left(\frac{3 \tau}{2} \right) \cos \left[
4 \left( \tau - t + p \right) \right] \,\int_0^1 \sin \left( 6 \pi \alpha - \frac{3 \tau}{2 R_0} \right) \,\d \alpha \, \d \tau = 0 \, ;
\end{equation*} 
again, this is zero because of the symmetry of the perturbation.  

\section{Concluding remarks}
\label{sec:conclude}

In this paper, we have developed a Melnikov theory to examine two-dimensional stable/unstable manifolds of hyperbolic points in three-dimensional flows. We require neither volume-preservation (in either the unperturbed or perturbed flow), nor time-periodicity.  Under fairly general conditions on the perturbation, we derive leading-order expressions for the time-varying location of the perturbed two-dimensional stable (or unstable) manifold. There is no requirement for the manifold to have been homoclinic or heteroclinic in this development.

The second goal of this paper is to characterize transport due a time-varying perturbation breaking apart a two-dimensional heteroclinic manifold in a three-dimensional flow.   We describe the three-dimensional analog of lobe dynamics \cite{romkedar,wigi_book} for the situation in which the perturbation is time-periodic, and develop expressions in terms of the Melnikov function for the leading-order volumes of lobes which lie between the perturbed stable and unstable manifolds at a general time.  If the unperturbed flow were volume-preserving, we show that the leading-order volumes of all lobes are identical, thereby allowing for this lobe volume to be a good quantifier of transport across the broken heteroclinic manifold via a three-dimensonal version of lobe dynamics.  Our more general contribution is that we can quantify the instantaneous flux engendered across the formerly impermeable heteroclinic manifold, building on a similar idea in two dimensions \cite{aperiodic}. If thinking in terms of three-dimensional fluid flows, we are thus able to characterize, in terms of the Melnikov function, the leading-order fluid flux, as a volume of fluid per unit time, crossing the broken heteroclinic manifold.  This is in a Lagrangian (as opposed to Eulerian) sense; the flux quantifies the transport of fluid particles following their flow history.  Thus, transport between the inside and the outside of the unperturbed heteroclinic manifold is captured by this theory, which allows for both non-volume-preservation and general time-dependence.

Higher-dimensional Melnikov methods, which develop a Melnikov function whose zeros are associated with persistent heteroclinic connections, usually have a function inside the integral which is known only as a fundamental solution to the adjoint of the equation of variations along an unperturbed heteroclinic trajectory.  This is not an explicit representation (except in the case of Hamiltonian systems, which are of course moreover limited to even dimensions), and therefore the Melnikov function is not {\em computable}.  Through our formulation in three dimensions, we derive this function automatically; this is what is inside the integral for both the general theory of locating two-dimensional manifolds, and in evaluating transport across a broken heteroclinic.  We note that this function we derive is valid even if volume is not preserved.  

The three-dimensional situation, with two-dimensional separating surfaces, is natural to study in the context of fluid flows.  As such, our work is expected to be of value in fluid transport: notably in quantifying locations of two-dimensional flow separators, and flux across broken ones.  Our formulation in the time-sinusoidal context in particular allows for a tool for optimizing mixing across separating surfaces, analogously to what has been done in optimizing mixing across one-dimensional separators in two-dimensional fluid flows \cite{optimal,frequency,l2mixer}.  Thus, applications to either maximizing transport (to empower good mixing of a fluid, or a two-phase fluid, in industrial applications), or minimizing it (to avoid pollutants contaminating a fluid) can be examined using the tools that we have developed in this article.

\vspace*{0.2cm}
\noindent {\bf {\em Acknowledgments:}} 
SB acknowledges support from the Australian Research Council under Grant DP200101764. EB is supported by the Army Research Office (N68164-EG) and DARPA, and  KGDSP is supported directly by Clarkson University.

\appendix

\section{Proof of Theorem~\ref{theorem:melnikov_unstable} (Displacement of unstable manifold)}
\label{sec:normaldis3} 

Suppose $ \epsilon \ne 0 $.  Fix $t \in (-\infty, T]$ and $ (p,\alpha) \in (-\infty,P] \times \mathrm{S}^1 $.  Let $\tau \in (-\infty, t]$ be a general
time value.  If  $\vec{x}^u(p, \alpha,\epsilon, t)$ (a point on the unstable manifold at time $ t $) is close to the point $\bar{\vec{x}}^u(\alpha, p)$ on the unperturbed manifold., then we realize that the appropriate parameterization at a general time $ \tau $ should
be chosen such that  $\vec{x}^u(p, \alpha,\epsilon, \tau)$ is the point on perturbed unstable manifold  close to the point 
$\bar{\vec{x}}^u(\tau - t + p,\alpha)$.  To this end, we define
\begin{equation}
\vec{z} ^u(p,\alpha,\epsilon,\tau) := \frac{1}{\epsilon} \left[ \vec{x}^u(p, \alpha,\epsilon, \tau) - \bar{\vec{x}}^u(\tau - t + p,\alpha) 
\right] \quad , \quad
\tau \in (-\infty, t] \, .
\label{eq:zu_define}
\end{equation}
Since $ T $ is finite, the difference between the perturbed and unperturbed trajectories is $ {\mathcal O}(\epsilon) $, 
and thus $ \vec{z}^u = {\mathcal O}(1) $, or more precisely there exists $ K \in \rm I\!R$, which is independent of $\tau$ and $\epsilon \in [0,\epsilon_{0}]$ for some $\epsilon_{0}$ and for finite $T$, 
\begin{equation}
\lvert \vec{z}^{u}(p,\alpha,\epsilon, \tau) \rvert \leq K \quad \text{for}~ \tau \in (-\infty,T].
\label{eq1bounded}
\end{equation}
Next, we define $\widetilde{M}^u(p,\alpha,\epsilon,\tau)$ as 
\begin{equation}
\widetilde{M}^u(p,\alpha,\epsilon,\tau) := \left[\vec{f}(\bar{\vec{x}}^{u}(\tau - t + p,\alpha)) \wedge \bar{\vec{x}}_\alpha^{u}
(\tau - t + p,\alpha) \right] \cdot \vec{z}^{u}(p,\alpha,\epsilon,\tau). 
\label{eq:mel_define}
\end{equation}
Given that 
\begin{align*}
\widetilde{M}^u(p,\alpha,\epsilon,t) &= \left[\vec{f}(\bar{\vec{x}}^{u}(p,\alpha)) \wedge \bar{\vec{x}}_\alpha^{u}
(p,\alpha) \right] \cdot \vec{z}^{u}(p,\alpha,\epsilon,t) \\
&= \left[\vec{f}(\bar{\vec{x}}^{u}(p,\alpha)) \wedge \bar{\vec{x}}_\alpha^{u}
(p,\alpha) \right] \cdot \frac{1}{\epsilon} \left[ \vec{x}^u(p, \alpha,\epsilon, t) - \bar{\vec{x}}^u(p,\alpha) \right] \, , 
\end{align*}
we note from (\ref{eq:distance_unstable}) that
\begin{equation}
d^u(p,\alpha,\epsilon,t) = \epsilon \frac{\widetilde{M}^u(p,\alpha,\epsilon,t) }{\lvert \vec{f}(\bar{\vec{x}}^{u}(p, \alpha)) \wedge \bar{\vec{x}}_\alpha^{u}(p, \alpha)\rvert} \, .
\label{eq:distance_melnikov}
\end{equation}
Consequently, we will determine $ d^u $ via an evolution equation for $ \widetilde{M}^u $ with respect to the
temporal variable $ \tau $.  Using the subscript as the notation for the partial derivative, taking $ \tau $-partial derivative
of (\ref{eq:mel_define}) yields
\begin{equation}
\begin{split}
\widetilde{M}^u_\tau(p,\alpha,\epsilon,\tau) &= \left[D\vec{f}(\bar{\vec{x}}^{u}(\tau - t + p,\alpha))~ \bar{\vec{x}}^{u}_\tau(\tau - t + p,\alpha) \wedge \bar{\vec{x}}_\alpha^{u}(\tau - t + p, \alpha) \right] \cdot \vec{z}^{u}(p,\alpha,\epsilon,\tau) \\& + \left[\vec{f}(\bar{\vec{x}}^{u}(\tau - t + p,\alpha)) \wedge \bar{\vec{x}}_{\alpha \tau}^{u}(\tau - t + p,\alpha) \right] \cdot \vec{z}^{u}(p,\alpha,\epsilon,\tau) \\&+ \left[\vec{f}(\bar{\vec{x}}^{u}(\tau - t + p,\alpha)) \wedge \bar{\vec{x}}_{\alpha}^{u}(\tau - t + p,\alpha) \right] \cdot \vec{z}^{u}_\tau(p,\alpha,\epsilon,\tau).
\label{eq8}
\end{split}
\end{equation}
We will build up simplifications for the many terms in (\ref{eq8}).
Since $\bar{\vec{x}}^{u}(\tau - t + p,\alpha)$ is a solution of unperturbed system $\dot{\vec{x}} = \vec{f}(\vec{x}) $, we
have 
\begin{equation}
\bar{\vec{x}}^{u}_\tau(\tau - t + p,\alpha) = \vec{f}\left(\bar{\vec{x}}^{u} (\tau - t + p, \alpha)\right) \, , 
\label{eq:temp1}
\end{equation}
Furthermore, since $\vec{x}^{u}(p,\alpha,\epsilon,\tau)$ is the solution of perturbed system $\dot{\vec{x}} = \vec{f}(\vec{x}) + \epsilon \vec{g}(\vec{x},t) $, we can write
\begin{align*}
\vec{x}^{u}_\tau(p,\alpha,\epsilon,\tau) &= \vec{f}(\vec{x}^{u}(p,\alpha,\epsilon,\tau)) + \epsilon \vec{g}(\vec{x}^{u}(p,\alpha,\epsilon,\tau),\tau) \\
&= \vec{f}(\bar{\vec{x}}^{u}(\tau - t + p,\alpha)) + \epsilon~ D\vec{f}(\bar{\vec{x}}^{u}(\tau - t + p,\alpha)) \vec{z}^{u}(p,\alpha,\epsilon,\tau)  \\
& \hspace*{0.6cm} + \frac{\epsilon^2}{2} \left[\vec{z}^{u}(p,\alpha,\epsilon,\tau)\right]^{\top} D^2\vec{f}(\vec{y}_1) \vec{z}^{u}(p,\alpha,\epsilon,\tau) \\
& + \epsilon \left[ \vec{g}(\bar{\vec{x}}^{u}(\tau - t + p,\alpha),\tau) + \epsilon~ D\vec{g}(\vec{y}_2,\tau)~ \vec{z}^{u}(p,\alpha,\epsilon,\tau) \right] \, , 
\end{align*}
where we have applied Taylor's theorem to both $ \vec{f} $ and $ \vec{g} $ about the spatial value $ \bar{\vec{x}}^{u} (\tau - t + p, \alpha) $ with deviation $ \epsilon \vec{z}^u(p,\alpha,\epsilon,\tau) $.  The notation $ D^2 $ represents the
Hessian matrix, and the unknown values $\vec{y}_i = \vec{y}_i(p,\alpha, \epsilon,\tau) $ ($ i = 1, 2 $) are located within $\epsilon K$ of  ~$\bar{\vec{x}}^{u}(\alpha,\tau - t + p)$.  Thus
\begin{align*}
\vec{z}^{u}_\tau(p,\alpha,\epsilon,\tau) &=  \frac{1}{\epsilon} \left[ \vec{x}^{u}_\tau(p,\alpha,\epsilon,\tau) - \bar{\vec{x}}^{u}_\tau (\tau - t + p, \alpha)\right] \\
&=
D\vec{f}(\bar{\vec{x}}^{u}(\tau - t + p,\alpha)) \vec{z}^{u}(p,\alpha,\epsilon,\tau) + 
\vec{g}(\bar{\vec{x}}^{u}(\tau - t + p, \alpha),\tau)\\& + \epsilon \left[ \frac{1}{2} \left[\vec{z}^{u}(p,\alpha,\epsilon,\tau)\right]^{\top} D^2\vec{f}(\vec{y}_1) + D\vec{g}(\vec{y}_2,\tau)\right] \left[\vec{z}^{u}(p,\alpha,\epsilon,\tau)\right] \, .
\end{align*}
Next, by the chain rule applied to (\ref{eq:temp1}), 
 \[
\bar{\vec{x}}^u_{\alpha \tau}(\tau - t + p,\alpha)= D \vec{f}(\bar{\vec{x}}^u(\tau - t + p,\alpha))~ \bar{\vec{x}}^u_\alpha(\tau - t + p, \alpha) \, , 
\]
By substituting all these into (\ref{eq8}), we get by separating out in orders of $ \epsilon $, 
	\[
	\begin{split}
	\widetilde{M}_\tau^{u}(p,\alpha,\epsilon,\tau)& = \left[D\vec{f}(\bar{\vec{x}}^{u}(\tau - t + p, \alpha))~ \vec{f}(\bar{\vec{x}}^{u}(\tau - t + p, \alpha)) \wedge \bar{\vec{x}}_\alpha^{u}(\tau - t + p, \alpha) \right] \cdot \vec{z}^{u}(p,\alpha,\epsilon,\tau) \\
	& + \left[\vec{f}(\bar{\vec{x}}^{u}(\tau - t + p, \alpha)) \wedge D \vec{f}(\bar{\vec{x}}^u(\tau - t + p, \alpha))~ \bar{\vec{x}}^u_\alpha(\tau - t + p, \alpha) \right] \cdot \vec{z}^{u}(p,\alpha,\epsilon,\tau)\\
	& + \left[\vec{f}(\bar{\vec{x}}^{u}(\tau - t + p, \alpha)) \wedge \bar{\vec{x}}_{\alpha}^{u}(\tau - t + p, \alpha) \right] \cdot \left[D\vec{f}(\bar{\vec{x}}^{u}(\tau - t + p, \alpha ))~ \vec{z}^{u}(p,\alpha,\epsilon,\tau) \right]\\& +\left[\vec{f}(\bar{\vec{x}}^{u}(\tau - t + p, \alpha)) \wedge \bar{\vec{x}}_{\alpha}^{u}(\tau - t + p, \alpha) \right] \cdot \vec{g}(\bar{\vec{x}}^{u}(\tau - t + p, \alpha),\tau)\\
	& + \frac{\epsilon}{2} ~\left[\vec{f}(\bar{\vec{x}}^{u}(\tau - t + p, \alpha)) \wedge \bar{\vec{x}}_{\alpha}^{u}(\tau - t + p, \alpha) \right] \cdot \left[  \left[\vec{z}^{u}(p,\alpha,\epsilon,\tau)\right]^{\top} D^2\vec{f}(\vec{y}_1)  \left[\vec{z}^{u}(p,\alpha,\epsilon,\tau)\right] \right] \\& + \epsilon ~\left[\vec{f}(\bar{\vec{x}}^{u}(\tau - t + p, \alpha)) \wedge \bar{\vec{x}}_{\alpha}^{u}(\tau - t + p, \alpha ) \right] \cdot \left[  D\vec{g}(\vec{y}_2,\tau)  \left[\vec{z}^{u}(p,\alpha,\epsilon,\tau)\right] \right] \, . 
	\end{split}
	\]
We now apply the identity given in Lemma~\ref{id1} (Appendix~\ref{sec:identi}) to simplify the first three terms 
(given in the first three lines) above.  By choosing $A = D\vec{f}(\bar{\vec{x}}^{u}(\alpha, \tau - t + p)),~ \vec{b} = \vec{f}(\bar{\vec{x}}^{u}(\alpha, \tau - t + p)), ~\vec{c} = \bar{\vec{x}}^{u}_\alpha(\alpha, \tau - t + p)$ and $\vec{d}=\vec{z}^u(p,\alpha,\epsilon,\tau)$,  
$\widetilde{M}_\tau^{u}(p,\alpha,\epsilon,\tau)$ can be recast as
	\begin{equation}
	\begin{split}
	\frac{\partial}{\partial \tau} \widetilde{M}^{u}(p,\alpha,\epsilon,\tau) &= \vec{\nabla} \cdot \vec{f}(\bar{\vec{x}}^{u}(\tau - t + p, \alpha))~ \widetilde{M}^{u}(p,\alpha,\epsilon,\tau)\\& +\left[\vec{f}(\bar{\vec{x}}^{u}(\tau - t + p, \alpha)) \wedge \bar{\vec{x}}_{\alpha}^{u}(\tau - t + p, \alpha) \right] \cdot \vec{g}(\bar{\vec{x}}^{u}(\tau - t + p, \alpha),\tau)
	+ \epsilon \, H(\tau) \, .
\label{eq:meldiffpert}
	\end{split} 	
	\end{equation}
where
\begin{align}
H(\tau) &:= \frac{1}{2} ~\left[\vec{f}(\bar{\vec{x}}^{u}(\tau - t + p, \alpha)) \wedge \bar{\vec{x}}_{\alpha}^{u}(\tau - t + p, \alpha) \right] \cdot \left[  \left[\vec{z}^{u}(p,\alpha,\epsilon,\tau)\right]^{\top} D^2\vec{f}(\vec{y}_1)  \left[\vec{z}^{u}(p,\alpha,\epsilon,\tau)\right] \right] \nonumber \\
	& + ~\left[\vec{f}(\bar{\vec{x}}^{u}(\tau - t + p, \alpha)) \wedge \bar{\vec{x}}_{\alpha}^{u}(\tau - t + p, \alpha) \right] \cdot \left[  D\vec{g}(\vec{y}_2,\tau)  \left[\vec{z}^{u}(p,\alpha,\epsilon,\tau)\right] \right] \, ,
	\label{eq:Htau}
	\end{align}
and we use $ {\mathrm{Tr}} \, D \vec{f} =
\vec{\nabla} \cdot \vec{f} $ (i.e., the divergence of $ \vec{f} $).
The differential equation (\ref{eq:meldiffpert}) is to be considered with the condition
\begin{equation}
\lim_{\tau \rightarrow \infty} \widetilde{M}^u(p,\alpha,\epsilon,\tau) = 0 \, , 
\label{eq:boundary}
\end{equation}
because from (\ref{eq:mel_define}) we see that $ \vec{f} \left(\bar{\vec{x}}^u(\tau - t + p, \alpha) \right) \rightarrow \vec{f}(\vec{a}) =
\vec{0} $ in this limit, with other terms in the definition remaining bounded.

We note that $ H(\tau) $ is bounded for $ \tau \in (-\infty,t] $ because of the boundedness of $ D^2 \vec{f} $, $ D \vec{g} $,
$ \bar{\vec{x}}^u(\tau-t+p,\alpha) $ (because this converges to $ \vec{a} $ as $ \tau \rightarrow - \infty $) and 
$ \vec{z}^u(p,\alpha,\epsilon,\tau) $.  Hence it makes sense to also consider
(\ref{eq:meldiffpert}) with $ \epsilon = 0 $, i.e., 
\begin{equation}
	\begin{split}
	\frac{\partial}{\partial \tau} {M}^{u}(p,\alpha,\tau) &= \vec{\nabla \cdot} \vec{f}(\bar{\vec{x}}^{u}(\tau - t + p, \alpha))~ {M}^{u}(p,\alpha,\tau)\\& +\left[\vec{f}(\bar{\vec{x}}^{u}(\tau - t + p, \alpha)) \wedge \bar{\vec{x}}_{\alpha}^{u}(\tau - t + p, \alpha) \right] \cdot \vec{g}(\bar{\vec{x}}^{u}(\tau - t + p, \alpha),\tau) \, , 
\label{eq:meldiffunp}
	\end{split} 	
	\end{equation}
whose solution $ M^u $ is also subject to the boundary condition (\ref{eq:boundary}).  It is easy to verify
that the linear differential equation (\ref{eq:meldiffunp}) with condition (\ref{eq:boundary}) has a solution
\begin{equation}
M^{u}(p,\alpha,t) = \int_{-\infty}^t \! \! e^{\int_{\tau}^t \nabla \cdot \vec{f}(\bar{\vec{x}}^{u}(\xi - t + p,\alpha)) d\xi} \left[ \vec{f}(\bar{\vec{x}}^{u}(\tau \! - \! t \! + \! p,\alpha)) \wedge \bar{\vec{x}}_{\alpha}^{u}(\tau \! - \! t \! + \! p,\alpha)\right] \cdot \vec{g}(\bar{\vec{x}}^{u}(\tau - t + p,\alpha),\tau)~ \d \tau \, . 
\label{eq:melu}
\end{equation}
by working via the integrating factor
\begin{equation}
\mu(\tau) = \exp\left[{-\int_0^\tau \nabla \cdot \vec{f}(\bar{\vec{x}}^{u}(\xi - t + p, \alpha)) \, \d \xi}\right] \, .
\label{eq:integratingfactor}
\end{equation}
The quantity $ M^u $ in (\ref{eq:melu}) is identical to the unstable Melnikov function as defined in (\ref{eq:melnikov_unstable}) after a change of integration variable $
\tau-t+p \rightarrow \tau $.  We show in Appendix~\ref{sec:convergent} that the improper integral (\ref{eq:melu})
is convergent, and so $ M^u $ is a well-defined solution to (\ref{eq:meldiffpert}) when $ \epsilon = 0 $.

Next, we show in Appendix~\ref{sec:close} that the solution $ \widetilde{M}^u(p,\alpha,\epsilon,t) $ of (\ref{eq:meldiffpert}) is within $\mathcal{O}(\epsilon)$ of $ M^u(p,\alpha,t) $. (This is not immediately obvious because integration over the noncompact
domain $ (-\infty,t] $ is necessary.)  This enables the replacement of $ \widetilde{M}^u(p,\alpha,\epsilon,t) $  with
$ M^u(p,\alpha,t) + {\mathcal O}(\epsilon) $ in (\ref{eq:distance_melnikov}), and consequently proves
Theorem~\ref{theorem:melnikov_unstable}.

\section{An important identity}
\label{sec:identi}

We introduce the following elementary identity which is valid for any $3 \times 3$ matrix and any $3 \times 1$ vectors. 
 
\begin{lemma}
	The following identity holds for any $3 \times 3$ matrix $A$ and any $\vec{b}, \vec{c}$ and $\vec{d}$ which are $3 \times 1$ vectors:
	\begin{equation}
		[ (A \vec{b}) \wedge \vec{c}] \cdot \vec{d} + \left[ \vec{b} \wedge (A \vec{c})\right] \cdot \vec{d} + \left[\vec{b} \wedge \vec{c}\right] \cdot (A \vec{d}) = {\mathrm{Tr}} \, (A)] \left[ (\vec{b} \wedge \vec{c}) \cdot \vec{d}\right] \, , 
	\end{equation}
	where $ \mathrm{Tr} \, \left( \centerdot \right) $ represents the trace operator.
	\label{id1}
\end{lemma}

\begin{proof} 
This can be verified by a straightforward though tedious computation, having defined 
\[	A = \left[ {\begin{array}{ccc}
	a_{11}  & a_{12} & a_{13}  \\
	a_{21}  & a_{22} & a_{23}  \\
	a_{31}  & a_{32} & a_{33}\\
	\end{array} } \right],\quad \vec{b} = \left[ {\begin{array}{c}
	b_{1} \\
	b_{2} \\
	b_{3}\\
	\end{array} } \right],\quad \vec{c} = \left[ {\begin{array}{c}
	c_{1} \\
	c_{2} \\
	c_{3}\\
	\end{array} } \right]~ \text{and} \quad \vec{d} = \left[ {\begin{array}{c}
	d_{1} \\
	d_{2} \\
	d_{3}\\
	\end{array} } \right].\]
The term $[ (A \vec{b}) \wedge \vec{c}] \cdot \vec{d}$ can be written as 
\begin{small}
\[
	\begin{split}
	[ (A \vec{b}) \wedge \vec{c}] \cdot \vec{d}&=d_{1}\left[c_{3}\left(a_{21}b_{1}+a_{22}b_{2}+a_{23}b_{3}\right)-c_{2}\left(a_{31}b_{1}+a_{32}b_{2}+a_{33}b_{3}\right)\right]\\& -d_{2}\left[c_{3}\left(a_{11}b_{1}+a_{12}b_{2}+a_{13}b_{3}\right)-c_{1}\left(a_{31}b_{1}+a_{32}b_{2}+a_{33}b_{3}\right)\right]\\& + d_{3}\left[c_{2}\left(a_{11}b_{1}+a_{12}b_{2}+a_{13}b_{3}\right)-c_{1}\left(a_{21}b_{1}+a_{22}b_{2}+a_{23}b_{3}\right)\right],
	\end{split}
	\label{id11}
\]
\end{small}
and the term $\left[ \vec{b} \wedge (A \vec{c})\right] \cdot \vec{d}$ is simplified to
\begin{small}
\[
	\begin{split}
	\left[ \vec{b} \wedge (A \vec{c})\right] \cdot \vec{d}&=d_{1}\left[b_{2}\left(a_{31}c_{1}+a_{32}c_{2}+a_{33}c_{3}\right)-b_{3}\left(a_{21}c_{1}+a_{22}c_{2}+a_{23}c_{3}\right)\right]\\& -d_{2}\left[b_{1}\left(a_{31}c_{1}+a_{32}c_{2}+a_{33}c_{3}\right)-b_{3}\left(a_{11}c_{1}+a_{12}c_{2}+a_{13}c_{3}\right)\right]\\& + d_{3}\left[b_{1}\left(a_{21}c_{1}+a_{22}c_{2}+a_{23}c_{3}\right)-b_{2}\left(a_{11}c_{1}+a_{12}c_{2}+a_{13}c_{3}\right)\right].
	\end{split}
	\label{id2}
	\]
\end{small}
Also, the term $\left[\vec{b} \wedge \vec{c}\right] \cdot (A \vec{d})$ can be simplified as
\[
\begin{split}
\left[\vec{b} \wedge \vec{c}\right] \cdot (A \vec{d})&=\left(a_{11}d_{1}+a_{12}d_{2}+a_{13}d_{3}\right)\left(b_{2}c_{3}-c_{2}b_{3}\right)\\& -\left(a_{21}d_{1}+a_{22}d_{2}+a_{23}d_{3}\right)\left(b_{1}c_{3}-c_{1}b_{3}\right)\\& + \left(a_{31}d_{1}+a_{32}d_{2}+a_{33}d_{3}\right)\left(b_{1}c_{2}-c_{1}b_{2}\right).
\end{split}
\label{id3}
\]
So, the left hand side of the identity can be obtained by adding these three equations together, resulting in 
\[ 
\begin{split}
&[ (A \vec{b}) \wedge \vec{c}] \cdot \vec{d} ~+~ \left[ \vec{b} \wedge (A \vec{c})\right] \cdot \vec{d} ~+~ \left[\vec{b} \wedge \vec{c}\right] \cdot (A \vec{d})\\&=\left(a_{11}+a_{22}+a_{33}\right) \left[ d_{1} \left( b_{2}c_{3}-c_{2}b_{3}\right) - d_{2} \left( b_{1}c_{3}-c_{1}b_{3} \right) + d_{3} \left( b_{1}c_{2}-b_{2}c_{1} \right) \right].
\end{split}
\]
The term $\left(a_{11}+a_{22}+a_{33}\right)$ is $ \mathrm{Tr} \, (A)$, and the triple scalar product of $\left[ (\vec{b} \wedge \vec{c}) \cdot \vec{d}\right]$ is equal to 
\begin{equation*}
\left[ (\vec{b} \wedge \vec{c}) \cdot \vec{d}\right]=\left[ d_{1} \left( b_{2}c_{3}-c_{2}b_{3}\right) - d_{2} \left( b_{1}c_{3}-c_{1}b_{3} \right) + d_{3} \left( b_{1}c_{2}-b_{2}c_{1} \right) \right] \, , 
\end{equation*}
which establishes the required result.
\end{proof}

\section{Convergence of the unstable Melnikov function $ M^u $}
\label{sec:convergent}

Since $ \bar{\vec{x}}^u $ is a trajectory on the two-dimension unstable manifold of $ \vec{a} $, we know
that we are in case~1, where $ D \vec{f} (\vec{a}) $ has two eigenvalues $\lambda_1^u$ and $\lambda_2^u$ 
with positive real part, and one eigenvalue $\lambda^s < 0 $.  Let is take the (potentially complex valued) 
eigenvectors $\vec{v}_1^u$ and $\vec{v}_2^u$.
 corresponding to $ \lambda_{1,2}^u $ as being normalized.  By assumption, $ \vec{v}_1 $ and $ \vec{v}_2 $
 are linearly independent; this also subsumes the situation of a repeated eigenvalue $ \lambda_1^u = \lambda_2^u $
 with geometric multiplicity $ 2 $.   The eigenspace spanned by $ \vec{v}_1^u $ and $ \vec{v}_2^u $ forms the
 tangent plane to $ \Gamma^u $ at $ \vec{a} $.  The deviation of a trajectory from the point $ \vec{a} $ is therefore
 governed by the linearized flow as $ \tau \rightarrow - \infty $, i.e., 
 \[
 \left| \bar{\vec{x}}^{u}(\tau - t + p, \alpha) - \vec{a} - A(\alpha) \vec{v}_1^u e^{ \lambda_1^u (\tau - t + p)} - B(\alpha) \vec{v}_2^u e^{\lambda_2^u (\tau - t + p)} \right| 
\rightarrow 0 \quad {\mathrm{as}} \quad \tau \rightarrow - \infty \, , 
\]
where $A(\alpha)$ and $B(\alpha)$ are (potentially complex-valued) scalars which are differentiable in $ \alpha $.  Thus, as $\tau$ approaches to negative infinity, 
 \[
\bar{\vec{x}}^{u}(\tau - t + p, \alpha) - \vec{a} \sim  A(\alpha) \vec{v}_1^u e^{ \lambda_1^u (\tau - t + p)} + B(\alpha) \vec{v}_2^u e^{\lambda_2^u (\tau - t + p)} \, , 
\]
Since  $\bar{\vec{x}}^{u}(\tau - t + p, \alpha)$ satisfies the equation $\dot{\vec{x}} = \vec{f}(\vec{x})$, we know that 
\[
\bar{\vec{x}}_{\tau}^{u}(\tau - t + p, \alpha) = \vec{f}(\bar{\vec{x}}^{u}(\tau - t + p, \alpha)),
\]
and so 
\[
\vec{f}(\bar{\vec{x}}^{u}(\tau - t + p, \alpha)) \sim A(\alpha) \lambda_1^u \vec{v}_1^u e^{ \lambda_1^u (\tau - t + p)} + B(\alpha) \lambda_2^u \vec{v}_2^u e^{ \lambda_2^u (\tau - t + p)}.
\]
Furthermore, the $ \alpha $-partial derivative is then
\[
\bar{\vec{x}}^{u}_{\alpha}(\tau - t + p, \alpha) \sim A'(\alpha) \vec{v}_1^u e^{  \lambda_1^u (\tau - t + p)} + B'(\alpha) \vec{v}_2^u e^{\lambda_2^u (\tau - t + p)}.
\]
\begin{lemma}
There exists a constant $ K_2 $ such that for all $ (p,\alpha,\tau) \in (-\infty,P] \times [0, 2 \pi) \times
(-\infty,t] $, and for all $ t \in (-\infty,T] $, 
\begin{equation}
e^{\int_{\tau}^t \nabla \cdot \vec{f}(\bar{\vec{x}}^{u}(\xi - t + p,\alpha)) d\xi}  \bigg \lvert \vec{f}(\bar{\vec{x}}^{u}(\tau - t + p,\alpha)) \wedge \bar{\vec{x}}^{u}_{\alpha}(\tau - t + p, \alpha) \bigg \rvert \le K_3 e^{\lambda^s (t-\tau)} 
 \, .
\label{eq:fxalpha}
\end{equation}
\label{lemma:fxalpha}
\end{lemma}
\begin{proof}
Based on the previous estimates, we have
\[ 
\begin{split}
	&	\vec{f}(\bar{\vec{x}}^{u}(\tau - t + p, \alpha)) \wedge \bar{\vec{x}}^{u}_{\alpha}(\tau - t + p, \alpha)\\
	& \sim \left[ A(\alpha) \lambda_1^u \vec{v}_1^u e^{ \lambda_1^u (\tau - t + p)} + B(\alpha) \lambda_2^u \vec{v}_2^u e^{ \lambda_2^u (\tau - t + p)} \right] \wedge \left[ A'(\alpha) \vec{v}_1^u e^{  \lambda_1^u (\tau - t + p)} + B'(\alpha) \vec{v}_2^u e^{\lambda_2^u (\tau - t + p)} \right] \\
	& = \left[ A(\alpha) B'(\alpha)  \lambda_1^u + A'(\alpha) B(\alpha) \lambda_2^u \right] e^{ (\lambda_1^u + \lambda_2^u) (\tau - t + p)} \left[ \vec{v}_1^u \wedge \vec{v}_2^u \right]. 
	\end{split} 
\]
Now, we note that $ \left| \vec{v}_1^u \wedge \vec{v}_2^u \right| \le 1 $ (the eigenvectors are normalized), and 
the functions $ A $ and
$ B $ and its derivatives are bounded on the compact set $ \alpha \in 2 \pi \, \mathrm{S}^1 $. Since we must have
 $ {\mathrm{Im}} \, \lambda_1^u = - {\mathrm{Im}} \, \lambda_2^u $, we obtain
 \[
 \bigg \lvert \vec{f}(\bar{\vec{x}}^{u}(\tau - t + p,\alpha)) \wedge \bar{\vec{x}}^{u}_{\alpha}(\tau - t + p, \alpha) \bigg \rvert \le K_1 e^{\mathrm{Re} (\lambda_1^u + \lambda_2^u) (\tau - t + p)} \, , 
 \]
 for some constant $ K_1 $.  
 Moreover, since the trace of $ D \vec{f} ( \bar{\vec{x}}^{u}(\alpha, \tau - t + p)) $
approaches $ {\mathrm{Tr}} \, D \vec{f}(\vec{a}) =  \lambda_1^u + \lambda_2^u + \lambda^s = 
{\mathrm{Re}} \, (\lambda_1^u + \lambda_2^u) + \lambda^s $ as $ \tau \rightarrow - \infty $, we have
 \[
e^{\int_{\tau}^t \nabla \cdot \vec{f}(\bar{\vec{x}}^{u}(\xi - t + p,\alpha)) d\xi} \sim e^{\int_{\tau}^t 
\left(  \mathrm{Re} \left( \lambda_1^u + \lambda_2^u \right) + \lambda^s \right) \d \xi} = e^{\left( 
\mathrm{Re} \left(  \lambda_1^u + \lambda_2^u \right) + \lambda^s \right) 
(t- \tau)} \, .
\]
Consequently, the exponential term can be bounded by a constant $ K_2 $ times the term on the right.
We can now estimate the product by
\begin{align*}
e^{\int_{\tau}^t \nabla \cdot \vec{f}(\bar{\vec{x}}^{u}(\xi - t + p,\alpha)) d\xi} 
\bigg \lvert \vec{f}(\bar{\vec{x}}^{u}(\tau - t + p,\alpha)) \wedge \bar{\vec{x}}^{u}_{\alpha}(\tau - t + p, \alpha) \bigg \rvert 
&\le 
K_1 e^{\mathrm{Re} (\lambda_1^u + \lambda_2^u) (\tau - t + p)} K_2 e^{\left( 
\mathrm{Re} \left(  \lambda_1^u + \lambda_2^u \right) + \lambda^s \right) 
(t- \tau)} \\
&= K_1 K_2 e^{\mathrm{Re} \left( \lambda_1^u + \lambda_2^u \right) p} e^{\lambda^s(t-\tau)} \\
&= K_3 e^{\lambda^s (t-\tau)} \, , 
\end{align*}
for a constant $ K_3 $, as desired.
 \end{proof}

Using the result of Lemma~\ref{lemma:fxalpha}, since $\vec{z}^{u}(p,\alpha, \tau)$ is bounded (say by a constant
$ K_4 $), from (\ref{eq:melu}), we obtain the bound
\[
\left| M^u(p,\alpha,t) \right| \le 
K_3 K_4  e^{\lambda^s t} \int_{-\infty}^t e^{-\lambda^s \tau} \, \d \tau = K_3 K_4 e^{\lambda^s t} \frac{e^{-\lambda^s \tau}}{- \lambda^s} \Big|_{-\infty}^t = \frac{K_3 K_4}{-\lambda^s} \, ,   
\]
where the limit is convergent because $ \lambda^s < 0 $.

\section{Proof that $ \widetilde{M}^u $ and $ M^u $ are $ {\mathcal O}(\epsilon) $-close}
\label{sec:close}

Let $ m(\tau) := \widetilde{M}^{u}(p,\alpha,\epsilon,\tau) - M^{u}(p,\alpha,\tau) $ be the difference in
the two functions at a general time $ \tau $; we need to show that $ m(t) = {\mathcal O}(\epsilon) $.
Subtracting the equation (\ref{eq:meldiffunp}) from (\ref{eq:meldiffpert}), 
and multiplying by the integrating factor $ \mu(\tau) $ we get
\[
\frac{\partial}{\partial \tau}\left[\mu(\tau) m(\tau)\right] 
 = \epsilon \mu(\tau) H(\tau)	
\]
subject to the condition $ m(-\infty) = 0 $.  This has a solution
\[
m(t) = \epsilon \int_{-\infty}^t \frac{\mu(\tau)}{\mu(t)}  H(\tau) \, \d \tau = \eps \int_{-\infty}^t  \exp\left[{\int_\tau^t \nabla \cdot \vec{f}(\bar{\vec{x}}^{u}(\xi - t + p, \alpha)) \d \xi}\right] H(\tau) \, \d \tau \, .
\]
Now, $ H(\tau) $ in (\ref{eq:Htau}) can be factored: one term consists of exactly the left-hand side of
(\ref{eq:fxalpha}), whereas the remainder of the terms are bounded because of the boundedness of $ \vec{z}^u $
(as argued in Appendix~\ref{sec:normaldis3}), and of 
$ D^2 \vec{f} $ and $ D \vec{g} $ (by hypothesis).  Applying Lemma~\ref{lemma:fxalpha}, we therefore obtain
\[
\left| m(t) \right| \le \epsilon K_5 \int_{-\infty}^t e^{\lambda^s (t-\tau)} \, \d \tau = \frac{\epsilon K_5}{-\lambda^s} \, .
\] 
for some constant $ K_5 $.  Hence, $ m(t) = {\mathcal O}(\epsilon) $ as desired.

\section{Proof of Theorem~\ref{theorem:heteroclinic} (Heteroclinic manifold splitting)}
\label{sec:melnikov}

For fixed $ (p,\alpha,t) $ in the relevant domains, we know that $ d^u $ in Theorem~\ref{theorem:melnikov_unstable} provides the displacement of $ \Gamma_\epsilon^u(\vec{a}_\epsilon) $ from $ \bar{\vec{x}}(p,\alpha) $ in the direction normal to $ \Gamma $,
and similarly, $ d^s $ in Theorem~\ref{theorem:melnikov_stable} the displacement of $ \Gamma_\epsilon^s(\vec{b}_\epsilon) $ in the same direction.  Since $ \bar{\vec{x}} = \bar{\vec{x}}^u = \bar{\vec{x}}^s $ in this instance,
\begin{align*}
d(p,\alpha,t,\epsilon &= d^u(p,\alpha,t,\epsilon) - d^s(p,\alpha,t,\epsilon) \\
&= \epsilon \frac{ M^u(p, \alpha, t)}{\lvert \vec{f}(\bar{\vec{x}}(p,\alpha)) \wedge \bar{\vec{x}}_\alpha(p, \alpha)\rvert} - \epsilon \frac{ M^s(p, \alpha, t)}{\lvert \vec{f}(\bar{\vec{x}}(p, \alpha)) \wedge \bar{\vec{x}}_\alpha(p, \alpha)\rvert} + \mathcal{O}(\epsilon^2) \\
&=  \epsilon \frac{ M^u(p, \alpha, t) - M^s(p,\alpha,t)}{\lvert \vec{f}(\bar{\vec{x}}(p,\alpha)) \wedge \bar{\vec{x}}_\alpha(p, \alpha)\rvert}+ \mathcal{O}(\epsilon^2) \\
&=: \epsilon \frac{ M(p, \alpha, t) }{\lvert \vec{f}(\bar{\vec{x}}(p,\alpha)) \wedge \bar{\vec{x}}_\alpha(p, \alpha)\rvert} + \mathcal{O}(\epsilon^2) \, , 
\end{align*}
where from (\ref{eq:melnikov_unstable}) and (\ref{eq:melnikov_stable}), we get
\begin{align*}
M(p, \alpha, t) &= M^u(p, \alpha, t) - M^s(p, \alpha, t), \\ 
&= \int_{-\infty}^p \exp\left[{\int_{\tau}^p \nabla \cdot \vec{f}(\bar{\vec{x}}(\xi, \alpha)) d\xi}\right] \left[ \vec{f}(\bar{\vec{x}}(\tau, \alpha)) \wedge \bar{\vec{x}}_{\alpha}(\tau, \alpha)\right] \cdot \vec{g}(\bar{\vec{x}}(\tau, \alpha),\tau + t -p)~ \d \tau \\
& - \left( - \int^{\infty}_p \exp\left[{\int_{\tau}^p \nabla \cdot \vec{f}(\bar{\vec{x}}(\xi, \alpha)) d\xi}\right] \left[ \vec{f}(\bar{\vec{x}}(\tau, \alpha)) \wedge \bar{\vec{x}}_{\alpha}(\tau, \alpha)\right] \cdot \vec{g}(\bar{\vec{x}}(\tau, \alpha),\tau + t -p)~ \d \tau \right) \\
& = \int^{\infty}_{-\infty} \exp\left[{\int_{\tau}^p \nabla \cdot \vec{f}(\bar{\vec{x}}(\alpha, \xi)) d\xi}\right] \left[ \vec{f}(\bar{\vec{x}}(\alpha, \tau)) \wedge \bar{\vec{x}}_{\alpha}(\alpha, \tau)\right] \cdot \vec{g}(\bar{\vec{x}}(\alpha, \tau),\tau + t -p)~ \d \tau \, , 
\end{align*}
as desired.

\section{Proof of Theorem~\ref{theorem:lobevolume} (Lobe volume)}
\label{sec:lobevolume}

We note that there is a nearby region, $ R^\star $, such that $ d(p,\alpha,\epsilon,t) $ sign-definite
on $ R^\star $, and moreover $ R^\star $'s boundary is $ Q^\star $, which consists of closed curves
which are $ {\mathcal O}(\epsilon)$-close to $ Q $.  While the lobe volume should properly be calculated
by integrating $ d $ over $ R^\star $, the error in integrating $ \epsilon M $ over $ R $ instead is 
of higher-order in $ \epsilon $.  Consequently, the leading-order lobe volume only requires leading-order
information.

Since $ \Gamma $ is $ (p,\alpha) $-parametrized by $ \bar{\vec{x}}(p,\alpha) $, we can write the vector
surface element on $ \Gamma $ by
\[
\vec{d S} = \bar{\vec{x}}_p(p,\alpha)  \wedge \bar{\vec{x}}_\alpha(p,\alpha) \, \d p \, \d \alpha 
= \vec{f} \left( \bar{\vec{x}}(p,\alpha) \right) \wedge \bar{\vec{x}}_\alpha(p,\alpha) \, \d p \, \d \alpha
\]
However, we know that the signed distance between the perturbed stable and unstable
manifolds, measured perpendicular to $ \Gamma $ at $ \bar{\vec{x}}(p,\alpha) $, is given by $ d $ in (\ref{eq:distancemelnikov}).  Noting moreover that using $ R $ rather than $ R^\star $ results in a higher-order error, and $ d $ itself is $ {\mathcal O}(\epsilon) $, we can write the volume of the lobe lying between the manifolds as
\begin{align*}
{\mathrm{Lobe~volume}} \, &= \int \! \! \! \! \int_{R^\star} \left| d(p,\alpha,\eps,t) \right| \, \left| \vec{d S} \right| \\
&= \int \! \! \! \! \int_{R} \left| d(p,\alpha,\eps,t) \right| \, \left| \vec{d S} \right| + {\mathcal O}(\epsilon^2) \\
&= \int \! \! \! \! \int_R \left|  \epsilon \frac{ M(p, \alpha, t)}{\lvert \vec{f}(\bar{\vec{x}}(p, \alpha)) \wedge \bar{\vec{x}}_\alpha(p, \alpha)\rvert} + \mathcal{O}(\epsilon^2) \right| \,  \left| \vec{f} \left( \bar{\vec{x}}(p,\alpha) \right) \wedge \bar{\vec{x}}_\alpha(p,\alpha) \right| \, \d p \, \d \alpha 
+ {\mathcal O}(\epsilon^2) \, , 
\end{align*}
which immediately gives the desired result.

\section{Proof of Theorem~\ref{theorem:lobevolume_harmonic} (Lobe volume for harmonic perturbations in the volume-preserving situation)}
\label{sec:lobevolume_harmonic}

Consider any one of the ring-lobes $ L_k $.  By Theorem~\ref{theorem:lobevolume}, its volume to leading-order
in $ \epsilon $ is given by
\[
\mathrm{Volume} \left( L_k \right) = \epsilon \int_0^1 \int_{\tilde{p}(\alpha,k-1)}^{\tilde{p}(\alpha,k)} \left| M(p,\alpha,t) \right| \, \d p \, \d \alpha + {\mathcal O}(\epsilon^2) \, .
\]
Employing (\ref{eq:melnikovharmonic}), and under the harmonic assumption in which $ h $ is independent
of $ p$, we get
\begin{align*}
\mathrm{Volume} \left( L_k \right) &= \epsilon \int_0^1 \int_{\tilde{p}(\alpha,k-1)}^{\tilde{p}(\alpha,k)} \left| {\mathcal F}\left\{ h(\alpha,\centerdot) \right\}(\omega) \right| \, \left| \cos \left[ \omega \left( t -p\right) + \phi + \mathrm{arg} \left(
{\mathcal F} \left\{ h (\alpha,\centerdot) \right\} (\omega) \right) \right] \right| \, \d p \, \d \alpha 
+ {\mathcal O}(\epsilon^2) \\
&= \epsilon \int_0^1 \left| {\mathcal F}\left\{ h(\alpha,\centerdot) \right\}(\omega) \right| \int_{\tilde{p}(\alpha,k-1)}^{\tilde{p}(\alpha,k)}  \left| \cos \left[ \omega \left( t -p\right) + \phi + \mathrm{arg} \left(
{\mathcal F} \left\{ h (\alpha,\centerdot) \right\} (\omega) \right) \right] \right| \, \d p \, \d \alpha 
+ {\mathcal O}(\epsilon^2) 
\end{align*}
The inner $ p $-integral is between adjacent zeros of the cosine function, specifically as given by the 
functions $ \tilde{p} $.  Note that this is of the absolute value of the cosine function of $ -\omega p $ plus a phase shift.  The
phase shift does not affect the integral because it is between adjacent zeros of $ p $.  Thus, we can simply shift the integral to be between any two adjacent zeros, and discard the entire phase shift
$ \omega t + \phi + \mathrm{arg} \left( {\mathcal F} \left\{ h (\alpha,\centerdot) \right\} (\omega) \right) $.  We choose $ \omega p $ to be between $ - \pi/2 $ and $ \pi/2 $, i.e., $ p $ between $ - \pi/(2 \omega) $ and $ \pi/(2 \omega) $.  Thus,
\begin{align*}
\mathrm{Volume} \left( L_k \right) &= 
\epsilon \int_0^1 \left| {\mathcal F}\left\{ h(\alpha,\centerdot) \right\}(\omega) \right| \int_{-\pi/(2 \omega)}^{\pi/(2 \omega)}  \cos \left[ - \omega p \right]  \, \d p \, \d \alpha 
+ {\mathcal O}(\epsilon^2) \nonumber \\
&= \frac{ 2 \epsilon}{\omega} \int_0^1 \left| {\mathcal F}\left\{ h(\alpha,\centerdot) \right\}(\omega) \right| \, \d \alpha + {\mathcal O}(\epsilon^2) \, , 
\label{eq:lobevolume_harmonic}
\end{align*}
whose leading-order term in $\epsilon $ is independent of $ k $.  Thus, in the time-harmonic situation in
which the unperturbed flow is volume-preserving, we get ring-lobes, {\em all} of which have the same
$ {\mathcal O}(\epsilon) $ volume.  

\section{Proof of Theorem~\ref{theorem:flux} (Instantaneous flux)}
\label{sec:flux_proof}

The pseudo-separatrix consists of three different segments. There is {\em no} Lagrangian flux 
across the stable and unstable manifold parts, because these are invariant objects.  They move with time,
but remain material surfaces.  The only flux that can occur is that crossing the strip.  We note from
Fig.~\ref{fig:pseudoseparatrix} that in parts of the strip where the unstable manifold is outside the
stable one, the flux will be outwards, and hence will be positive.  Thus, in $ (p,\alpha,t) $ regions
in which the Melnikov function is positive, a positive contribution to the flux occurs.  Conversely, if
the unstable manifold is inside the stable one, the flux is into the closed surface and hence negative,
again consonant with the sign of the Melnikov function at such points.  

A general point $ \vec{r} $ on the strip $ S $, as given in (\ref{eq:strip}), is
\begin{align*}
\vec{r}(s,\alpha,\epsilon,t) &= \bar{\vec{x}}(p,\alpha) + 
    \hat{\vec{n}}(p,\alpha) \left[ s d^u(p,\alpha,\epsilon,t) + (1-s) d^s(p,\alpha,\epsilon,t) \right] \quad ; \quad (s,\alpha) \in [0,1] \times \mathrm{S}^1 \\
    &= \bar{\vec{x}}(p,\alpha) + 
    \epsilon \frac{\vec{f} \left( \bar{\vec{x}}(p,\alpha) \right) \wedge \bar{\vec{x}}_\alpha(p,\alpha)}{\left|\vec{f} \left( \bar{\vec{x}}(p,\alpha) \right) \wedge \bar{\vec{x}}_\alpha(p,\alpha)\right| }
    \frac{s M^u(p,\alpha,t) + (1-s) M^s(p,\alpha,t)}{ \left| \vec{f} \left( \bar{\vec{x}}(p,\alpha) \right) \wedge \bar{\vec{x}}_\alpha(p,\alpha) \right| } + {\mathcal O}(\epsilon^2) \, .
\end{align*}
The vector surface element on $ S $ in terms of the $ (s,\alpha) $-parametrization, chosen
so that the outward normal is positive, is therefore
\begin{align*}
    \vec{\d S} &= \vec{r}_\alpha(s,\alpha,\epsilon,t) \wedge \vec{r}_s(s,\alpha,\epsilon,t) \, \d s \, \d \alpha \\
    &= \left[ \bar{\vec{x}}_\alpha(p,\alpha) + {\mathcal O}(\epsilon) \right] \wedge \epsilon \frac{\vec{f} \left( \bar{\vec{x}}(p,\alpha) \right) \wedge \bar{\vec{x}}_\alpha(p,\alpha)}{\left|\vec{f} \left( \bar{\vec{x}}(p,\alpha) \right) \wedge \bar{\vec{x}}_\alpha(p,\alpha)\right|^2} \left[ M^u(p,\alpha,t) - M^s(p,\alpha,t) \right]  \, \d s \, \d \alpha + {\mathcal O}(\epsilon^2) \\
    &= \epsilon \frac{M(p,\alpha,t) \, \bar{\vec{x}}_\alpha(p,\alpha) \wedge \left[ \vec{f} \left( \bar{\vec{x}}(p,\alpha) \right) \wedge \bar{\vec{x}}_\alpha(p,\alpha) \right]}{\left|\vec{f} \left( \bar{\vec{x}}(p,\alpha) \right) \wedge \bar{\vec{x}}_\alpha(p,\alpha)\right|^2} \, \d s \, \d \alpha + {\mathcal O}(\epsilon^2) .
\end{align*}
Therefore, the Lagrangian flux crossing $ S $ is
\begin{small}
\begin{align*}
\Phi(p,t,\epsilon) &= \iint_S 
 \left[ \vec{f}(\bar{\vec{x}}(p,\alpha) + \epsilon \vec{g}\left( \bar{\vec{x}}(p,\alpha),t \right) \right] \cdot \vec{\d S} \nonumber \\
&= \int_0^1 \int_0^1
 \left[ \vec{f}(\bar{\vec{x}}(p,\alpha) + \epsilon \vec{g}\left( \bar{\vec{x}}(p,\alpha),t \right) \right] \cdot  \epsilon \frac{M(p,\alpha,t) \, \bar{\vec{x}}_\alpha(p,\alpha) \wedge \left[ \vec{f} \left( \bar{\vec{x}}(p,\alpha) \right) \wedge \bar{\vec{x}}_\alpha(p,\alpha) \right]}{\left|\vec{f} \left( \bar{\vec{x}}(p,\alpha) \right) \wedge \bar{\vec{x}}_\alpha(p,\alpha)\right|^2} \, \d s \, \d \alpha + {\mathcal O}(\epsilon^2) \nonumber \\
 &= \eps \int_0^1 \int_0^1 M(p,\alpha,t) \frac{\vec{f}(\bar{\vec{x}}(p,\alpha) \cdot \left[ \bar{\vec{x}}_\alpha(p,\alpha) \wedge \left[ \vec{f} \left( \bar{\vec{x}}(p,\alpha) \right) \wedge  \bar{\vec{x}}_\alpha(p,\alpha) \right] \right] }{\left|\vec{f} \left( \bar{\vec{x}}(p,\alpha) \right) \wedge \bar{\vec{x}}_\alpha(p,\alpha)\right|^2} \, \d s \, \d \alpha + {\mathcal O}(\epsilon^2) \nonumber \\
 &= \eps \int_0^1 \int_0^1 M(p,\alpha,t) \frac{\vec{f}(\bar{\vec{x}}(p,\alpha) \cdot \left[ \vec{f}(\bar{\vec{x}}(p,\alpha)) \left[ \bar{\vec{x}}_\alpha(p,\alpha) \cdot \bar{\vec{x}}_\alpha(p,\alpha) \right] - \bar{\vec{x}}_\alpha(p,\alpha) \left[ \vec{f}(\bar{x}(p,\alpha)) \cdot \bar{\vec{x}}_\alpha(p,\alpha) \right]\right] }{\left|\vec{f} \left( \bar{\vec{x}}(p,\alpha) \right) \wedge \bar{\vec{x}}_\alpha(p,\alpha)\right|^2} \, \d s \, \d \alpha + {\mathcal O}(\epsilon^2) \nonumber \\
 &= \eps \int_0^1 \int_0^1 M(p,\alpha,t) \frac{\left| \vec{f}(\bar{\vec{x}}(p,\alpha) \right|^2 \left| \bar{\vec{x}}_\alpha(p,\alpha) \right|^2 - \left| \vec{f}(\bar{\vec{x}}(p,\alpha) \cdot \bar{\vec{x}}_\alpha(p,\alpha) \right|^2 }{\left|\vec{f} \left( \bar{\vec{x}}(p,\alpha) \right) \wedge \bar{\vec{x}}_\alpha(p,\alpha)\right|^2} \, \d s \, \d \alpha + {\mathcal O}(\epsilon^2) \nonumber \\ 
 &= \eps \int_0^1 \int_0^1 M(p,\alpha,t) \, \d s \, \d \alpha + {\mathcal O}(\epsilon^2) \nonumber \\
 &= \eps \int_0^1 M(p,\alpha,t) \, \d \alpha + {\mathcal O}(\epsilon^2) \, , 
\end{align*}
\end{small}
as required. In this derivation, we have used standard vector identities in three-dimensions: the `bac-cab' rule and Lagrange's identity.

\section{Proof of Corollary~\ref{corollary:flux_harmonic} (Instantaneous flux for harmonic perturbations)}
\label{sec:flux_harmonic_proof}

In this proof, we will use the shorthand notation
\[
F(p,\alpha) := {\mathcal F} \left\{ h \left(p, \alpha, \centerdot \right) \right\} (\omega) \, .\]
Inserting the expression for the time-harmonic Melnikov function (\ref{eq:melnikovharmonic}) into the instantaneous flux formula (\ref{eq:flux}), and performing standard trigonometric manipulations, we get
\begin{align*}
    \Phi(p,t,\epsilon) &= \epsilon \int_0^1 \left| F(p,\alpha) \right| \cos \left[ \omega (t-p) + \phi + \mathrm{arg}\left( F(p,\alpha) \right) \right] \, \d \alpha + {\mathcal O}(\epsilon^2) \\
    &= \epsilon \Bigg\{ \cos \left[ \omega(t-p)+\phi \right] \int_0^1 \left| F(p,\alpha) \right| \cos \left[ 
    \mathrm{arg} \left( F(p,\alpha) \right) \right] \, \d \alpha \\
    & \hspace*{1cm} - \sin \left[ \omega(t-p)+\phi \right] \int_0^1 \left| F(p,\alpha) \right| \sin \left[
    \mathrm{arg} \left( F(p,\alpha) \right) \right] \, \d \alpha \Bigg\} + {\mathcal O}(\epsilon^2) \\
    &= \epsilon \left\{ \cos \left[ \omega(t-p)+\phi \right] \int_0^1 \mathrm{Re} \left(
    F(p,\alpha) \right) \, \d \alpha -  \sin \left[ \omega(t-p)+\phi \right] \int_0^1 \mathrm{Im} \left(
    F(p,\alpha) \right) \, \d \alpha \right\} + {\mathcal O}(\epsilon^2) \\
    &= \epsilon \sqrt{ \left( \int_0^1 \mathrm{Re} \left(
    F(p,\alpha) \right) \, \d \alpha \right)^2 + \left( \int_0^1 \mathrm{Im} \left(
    F(p,\alpha) \right) \, \d \alpha \right)^2 } \Bigg\{ \frac{ \cos \left[ \omega(t-p)+\phi \right] \int_0^1 \mathrm{Re} \left(
    F(p,\alpha) \right) \, \d \alpha}{
    \sqrt{ \left( \int_0^1 \mathrm{Re} \left(
    F(p,\alpha) \right) \, \d \alpha \right)^2 + \left( \int_0^1 \mathrm{Im} \left(
    F(p,\alpha) \right) \, \d \alpha \right)^2 }} \\
    & \hspace*{1cm} - \frac{ \sin\left[ \omega(t-p)+\phi \right] \int_0^1 \mathrm{Im} \left(
    F(p,\alpha) \right) \, \d \alpha}{
    \sqrt{ \left( \int_0^1 \mathrm{Re} \left(
    F(p,\alpha) \right) \, \d \alpha \right)^2 + \left( \int_0^1 \mathrm{Im} \left(
    F(p,\alpha) \right) \, \d \alpha \right)^2 }} \Bigg\} + {\mathcal O}(\epsilon^2) \\
    &= \epsilon \left| \int_0^1 F(p,\alpha) \, \d \alpha \right| \Bigg\{ 
    \cos \left[ \omega(t-p) + \phi \right] \frac{ \mathrm{Re} \left( \int_0^1 F(p,\alpha) \, \d \alpha
    \right)}{
    \left| \int_0^1 F(p,\alpha) \, \d \alpha \right|} \\
    & \hspace*{1cm} - \sin \left[ \omega(t-p) + \phi \right] \frac{ \mathrm{Im} \left( \int_0^1 F(p,\alpha) \, \d \alpha
    \right)}{
    \left| \int_0^1 F(p,\alpha) \, \d \alpha \right|} \Bigg\} + {\mathcal O}(\epsilon^2) \\
    &= \epsilon \left| \int_0^1 F(p,\alpha) \, \d \alpha \right| \Bigg\{ 
    \cos \left[ \omega(t-p) + \phi \right]  \cos \left[ \mathrm{arg} \left( \int_0^1 F(p,\alpha) \, 
    \d \alpha \right) \right] \\
    & \hspace*{1cm} -  \sin \left[ \omega(t-p) + \phi \right]  \sin \left[ \mathrm{arg} \left( \int_0^1 F(p,\alpha)    \d \alpha \right) \right] \Bigg\} + {\mathcal O}(\epsilon^2) \\
    &= \epsilon \left| \int_0^1 F(p,\alpha) \, \d \alpha \right|  \cos \left[ \omega(t-p) + \phi+ 
    \mathrm{arg} \left( \int_0^1 F(p,\alpha) \, \d \alpha \right) \right] + {\mathcal O}(\epsilon^2) \, , 
\end{align*}
as required.

\bibliographystyle{plain}
\bibliography{introbib}

\end{document}